\documentclass[11pt]{amsart}  

\usepackage{amsmath}
\usepackage{amsthm}
\usepackage{amsfonts}
\usepackage{amssymb}
\usepackage{eucal}
\usepackage{yfonts}
\usepackage[all]{xy}
\usepackage{amsxtra}
\usepackage{appendix} 
\usepackage{tensor} 
\usepackage{url}
\usepackage{upgreek}
\usepackage{etoolbox}
\apptocmd{\sloppy}{\hbadness 10000\relax}{}{}

\usepackage[pdftex, breaklinks, linktocpage=true, bookmarksopen=true,bookmarksopenlevel=0,bookmarksnumbered=true]{hyperref}

\RequirePackage{color}
\definecolor{bwgreen}{rgb}{0.183,1,0.5}
\definecolor{bwmagenta}{rgb}{0.7,0.0,0.1}
\definecolor{bwblue}{rgb}{0.317,0.161,1}

\hypersetup{
pdfauthor = {\textcopyright\ Bryden Cais},
pdfcreator = {\LaTeX\ with package \flqq hyperref\frqq},
colorlinks = {true},
linkcolor={bwmagenta},	 
anchorcolor={},	 
citecolor={bwblue},	
filecolor={},	
menucolor={},	
runcolor={},	
urlcolor={bwgreen}, 
}

\CompileMatrices
\entrymodifiers={+!!<0pt,\fontdimen22\textfont2>}

\DeclareFontFamily{OT1}{rsfs}{}
\DeclareFontShape{OT1}{rsfs}{n}{it}{<-> rsfs10}{}
\DeclareMathAlphabet{\mathscr}{OT1}{rsfs}{n}{it}

\DeclareFontFamily{OT1}{pzc}{}
\DeclareFontShape{OT1}{pzc}{n}{it}{<->s*[2.2]pzc}{}
\DeclareMathAlphabet{\mathpzc}{OT1}{pzc}{b}{sl}

\makeatletter
\newcommand{\rmnum}[1]{\romannumeral #1}
\newcommand{\Rmnum}[1]{\expandafter\@slowromancap\romannumeral #1@}
\makeatother

\DeclareMathOperator{\id}{id}

\DeclareMathOperator{\Frac}{Frac}
\DeclareMathOperator{\ord}{ord} 
\DeclareMathOperator{\nil}{nil}

\newcommand*{\pr}{\rho}
\newcommand*{\ps}{\sigma}

\DeclareMathOperator{\Hom}{Hom}
\DeclareMathOperator{\End}{End}
\DeclareMathOperator{\Ext}{Ext}
\DeclareMathOperator{\Gal}{Gal}
\DeclareMathOperator{\GL}{GL}
\DeclareMathOperator{\SL}{SL}

\DeclareMathOperator{\Aut}{Aut}
\DeclareMathOperator{\Spec}{Spec}

\DeclareMathOperator{\et}{\acute{e}t}

\DeclareMathOperator{\dR}{dR}
\DeclareMathOperator{\Cris}{Cris}
\DeclareMathOperator{\cris}{cris}

\DeclareMathOperator{\tr}{tr}

\DeclareMathOperator{\Pic}{Pic}
\DeclareMathOperator{\Alb}{Alb}

\DeclareMathOperator{\Extrig}{Extrig}
\DeclareMathOperator{\Lie}{Lie}
\DeclareMathOperator{\Inf}{Inf}

\DeclareMathOperator{\Fil}{Fil}

\DeclareMathOperator{\sep}{sep}
\DeclareMathOperator{\perf}{rad}

\DeclareMathOperator{\Ig}{Ig}

\DeclareMathOperator{\im}{im}
\DeclareMathOperator{\rad}{rad}

\DeclareMathOperator{\Cot}{Cot}
\DeclareMathOperator{\pdiv}{pdiv}

\DeclareMathOperator{\mult}{m}
\DeclareMathOperator{\loc}{ll}

\DeclareMathOperator{\proj}{proj}
\DeclareMathOperator{\incl}{incl}
\DeclareMathOperator{\Hodge}{Hdg}

\DeclareMathOperator{\Null}{null}
\DeclareMathOperator{\sh}{sh}

\newcommand*{\R}{\ensuremath{\mathbf{R}}}   
\renewcommand*{\c}{\ensuremath{\mathbf{C}}}              
\newcommand*{\Z}{\ensuremath{\mathbf{Z}}}               
\newcommand*{\Q}{\ensuremath{\mathbf{Q}}}                           
             
\newcommand*{\Qbar}{\overline{\Q}}

\newcommand*{\Kbar}{\overline{K}}    
          
\newcommand*{\Gm}{\ensuremath{{\mathbf{G}_m}}}   
\newcommand*{\Ga}{\ensuremath{{\mathbf{G}_a}}}   

\newcommand*{\m}{\mathfrak{M}}

\newcommand*{\s}{\mathfrak{S}}

\newcommand*{\A}{\ensuremath{\mathcal{A}}}

\newcommand*{\B}{\mathcal{B}}

\newcommand*{\C}{\mathbf{C}}
\newcommand*{\E}{\mathscr{E}}     
\newcommand*{\F}{\mathbf{F}}
\newcommand*{\scrF}{\mathscr{F}}

\newcommand*{\G}{\mathcal{G}}
\newcommand*{\scrG}{\mathscr{G}}  
\newcommand*{\scrH}{\mathscr{H}}                           
                               
\newcommand*{\I}{\mathscr{I}}                               
\newcommand*{\J}{\mathcal{J}}

\newcommand*{\scrM}{\mathscr{M}}

\renewcommand*{\O}{\mathscr{O}}                    
 
\newcommand*{\X}{\mathcal{X}}     
\newcommand*{\Y}{\mathcal{Y}}

\newcommand*{\scrHom}{\mathscr{H}\mathit{om}}      
\newcommand*{\scrExtrig}{\mathscr{E}\mathit{xtrig}}	
\newcommand*{\scrExt}{\mathscr{E}\mathit{xt}}               
\newcommand*{\scrLie}{\mathscr{L}\mathit{ie}}

\newcommand*{\D}{\ensuremath{\mathbf{D}}}

\renewcommand*{\H}{\ensuremath{\mathfrak{H}}}

\newcommand*{\Dual}[1]{{{#1}^t}}

\renewcommand*{\int}{\ensuremath{\mathrm{int}}}

\newcommand*{\e}{\ensuremath{\mathbf{E}}}
\renewcommand*{\a}{\ensuremath{\mathbf{A}}}

\renewcommand*{\u}[1]{\underline{#1}}
\renewcommand*{\o}[1]{\overline{#1}}
\newcommand*{\wh}[1]{\widehat{#1}}
\newcommand*{\wt}[1]{\widetilde{#1}}
\newcommand*{\nor}[1]{{#1}^{\mathrm{n}}}

\newcommand*{\tens}{\mathop{\otimes}\limits}

\newcommand*{\fiber}{\mathop{\times}\limits}

\DeclareMathOperator{\BT}{BT}

\theoremstyle{plain}
  \newtheorem{theorem}{Theorem}
  \newtheorem{proposition}[theorem]{Proposition}
  \newtheorem{lemma}[theorem]{Lemma}
  \newtheorem{corollary}[theorem]{Corollary}

\theoremstyle{definition}
  \newtheorem{definition}[theorem]{Definition}

\theoremstyle{remark}
  \newtheorem{example}[theorem]{Example}
  \newtheorem{remark}[theorem]{Remark}

  \newtheorem{warning}[theorem]{Warning}

\include{header}

\numberwithin{theorem}{subsection}  
\numberwithin{equation}{subsection}

\begin{document}
\title{The Geometry of Hida Families \Rmnum{2}: $\Lambda$-adic $(\varphi,\Gamma)$-modules
and $\Lambda$-adic Hodge Theory}

\author{Bryden Cais}
\address{University of Arizona, Tucson}
\curraddr{Department of Mathematics, 617 N. Santa Rita Ave., Tucson AZ. 85721}
\email{cais@math.arizona.edu}

\thanks{
	During the writing of this paper, the author was partially supported by an NSA Young Investigator grant
	(H98230-12-1-0238) and an NSF RTG (DMS-0838218).
	}

\dedicatory{To Haruzo Hida, on the occasion of his $60^{\text{th}}$ birthday.}

\date{\today}

\begin{abstract}	
	We construct the $\Lambda$-adic crystalline and Dieudonn\'e analogues of Hida's
	ordinary $\Lambda$-adic \'etale cohomology, 
	and employ integral $p$-adic Hodge theory to prove $\Lambda$-adic
	comparison isomorphisms between these cohomologies and
	the $\Lambda$-adic de Rham cohomology studied in \cite{CaisHida1}
	as well as Hida's $\Lambda$-adic \'etale cohomology.
	As applications of our work, we provide a ``cohomological" construction of the 
	family of $(\varphi,\Gamma)$-modules attached to Hida's ordinary $\Lambda$-adic \'etale
	cohomology by \cite{Dee}, and we give a new and purely geometric 
	proof of Hida's finitenes and control theorems.  We also prove suitable 
	$\Lambda$-adic duality theorems for each of the cohomologies we construct.  
\end{abstract}

\maketitle

\section{Introduction}\label{intro}

\subsection{Motivation}

In a series of groundbreaking papers \cite{HidaGalois} and \cite{HidaIwasawa}, Hida
constructed $p$-adic analytic families of $p$-ordinary Galois representations
interpolating the Galois representations attached to
$p$-ordinary cuspidal Hecke eigenforms in integer weights $k\ge 2$ by
Deligne \cite{DeligneFormes}, \cite{CarayolReps}.
At the heart of Hida's construction is the $p$-adic \'etale
cohomology $H^1_{\et}:=\varprojlim_r H^1_{\et}(X_1(Np^r)_{\Qbar},\Z_p)$ of the tower of modular curves
over $\Q$, 
which is naturally a module for the 
``big" $p$-adic Hecke algebra $\H^*:=\varprojlim_r \H_r^*$, which is itself an algebra
over the completed group ring $\Lambda:=\Z_p[\![\Delta_1]\!]\simeq \Z_p[\![T]\!]$
via the diamond operators $\Delta_r:=1+p^r\Z_p$.  
Writing $e^*\in \H^*$ for the idempotent attached to the (adjoint) Atkin operator $U_p^*$,
Hida proves (via explicit computations in group cohomology) that
the ordinary part $e^*H^1_{\et}$ of $H^1_{\et}$ is
finite and free as a module over $\Lambda$, and that the resulting Galois representation
\begin{equation*}
	\xymatrix{
		{\rho: \scrG_{\Q}} \ar[r] & {\Aut_{\Lambda}(e^*H^1_{\et}) \simeq \GL_m(\Z_p[\![T]\!])} 
		}
\end{equation*}
$p$-adically interpolates the representations attached to ordinary cuspidal eigenforms.

By analyzing the geometry of the tower of modular curves, Mazur and Wiles \cite{MW-Hida}
showed that both the inertial invariants and convariants of the
the local (at $p$) representation $\rho_p$ are free of the finite same rank
over $\Lambda$, and hence that the 
ordinary filtration of the Galois representations attached
to ordinary cuspidal eigenforms interpolates in Hida's $p$-adic family.
As an application, they gave examples of cuspforms $f$ and primes $p$
for which the specialization of the associated Hida family of Galois representations 
to weight $k=1$ is not Hodge--Tate,
and so does not arise from a weight one cuspform via the construction of Deligne-Serre
\cite{DeligneSerre}.
Shortly thereafter, Tilouine \cite{Tilouine} clarified the geometric underpinnings
of \cite{HidaGalois} and \cite{MW-Hida}.

In \cite{OhtaEichler}, \cite{Ohta1} and \cite{Ohta2},
Ohta initiated the study of the $p$-adic Hodge theory of
Hida's ordinary $\Lambda$-adic (local) Galois representation $\rho_p$.
Using the invariant differentials on the tower of $p$-divisible groups
over $R_{\infty}:=\Z_p[\mu_{p^{\infty}}]$ 
attached to the ``good quotient" modular abelian varieties
introduced in \cite{MW-Iwasawa} and
studied in \cite{MW-Hida} and \cite{Tilouine}, Ohta constructs
a certain 
$\Lambda_{R_{\infty}}:= R_{\infty}[\![\Delta_1]\!]$-module $e^*H^1_{\Hodge}$,
which is the Hodge cohomology analogue of $e^*H^1_{\et}$.
Via an integral version of the Hodge--Tate comparison isomorphism \cite{Tate}
for ordinary $p$-divisible groups, Ohta establishes a $\Lambda$-adic
Hodge-Tate comparison isomorphism relating
$e^*H^1_{\Hodge}$ and the semisimplification
of the ``semilinear representation" $\rho_{p}\wh{\otimes} \O_{\C_p}$.
Using Hida's results, Ohta shows that $e^*H^1_{\Hodge}$
is free of finite rank over $\Lambda_{R_{\infty}}$
and specializes to finite level exactly as one expects.  As applications
of his theory, Ohta provides a construction of two-variable $p$-adic $L$-functions
attached to families of ordinary cuspforms differing from that of Kitagawa \cite{Kitagawa},
and, in a subsequent paper \cite{Ohta2}, 
provides a new and streamlined proof of the theorem of Mazur--Wiles \cite{MW-Iwasawa}
(Iwasawa's Main Conjecture for $\Q$; see also \cite{WilesTotallyReal}).
We remark that Ohta's $\Lambda$-adic Hodge-Tate isomorphism is a crucial ingredient
in the forthcoming proof of Sharifi's conjectures \cite{SharifiConj}, \cite{SharifiEisenstein} 
due to Fukaya and Kato \cite{FukayaKato}.

In \cite{CaisHida1}, we continued the trajectory begun by Ohta
by constructing the de Rham analogue of $e^*H^1_{\et}$.   
Using the canonical integral structures in de Rham cohomology
studied in \cite{CaisDualizing} and certain Katz-Mazur \cite{KM} integral
models $\X_r$ of $X_1(Np^r)$ over $R_r:=\Z_p[\mu_{p^r}]$, for each $r> 0$ we 
constructed a canonical
short exact sequences of free $R_r$-modules
\begin{equation}
	\xymatrix{
		0\ar[r] & {H^0(\X_r, \omega_{\X_r/R_r})} \ar[r] & {H^1(\X_r/R_r)} \ar[r] & {H^1(\X_r,\O_{\X_r})}
		\ar[r] & 0 
	}\label{finiteleveldRseq}
\end{equation}
whose scalar extension to $K_r:=\Frac(R_r)$ recovers the Hodge filtration of the de Rham cohomology 
of $X_1(Np^r)$ over $K_r$.
Extending scalars to $R_{\infty}$, taking projective limits, and passing to ordinary parts
gives a sequence of modules over $\Lambda_{R_{\infty}}$
with semilinear $\Gamma:=\Gal(K_{\infty}/K_0)$-action
and commuting linear $\H^*$-action 
\begin{equation}
		\xymatrix{
		0\ar[r] & {e^*H^0(\omega)} \ar[r] & {e^*H^1_{\dR}} \ar[r] & {e^*H^1(\O)} \ar[r] & 0
	}.\label{orddRseq}
\end{equation}
The main result of \cite{CaisHida1} is that (\ref{orddRseq}) 
is the correct de Rham analogue of Hida's ordinary $\Lambda$-adic \'etale cohomology
and Ohta's ordinary $\Lambda$-adic Hodge cohomology (see \cite[Theorem 3.2.3]{CaisHida1}):

\begin{theorem}\label{dRMain}
	Let $d=\sum_{k=3}^{p+1} d_k$ for $d_k:=\dim_{\F_p} S_k(\Upgamma_1(N);\F_p)^{\ord}$
	the $\F_p$-dimension of the space of mod $p$ weight $k$ ordinary cuspforms for $\Upgamma_1(N)$.
	Then $(\ref{orddRseq})$ is a short exact sequence of free $\Lambda_{R_{\infty}}$-modules
	of ranks $d$, $2d$, and $d$, respectively.  
	Applying $\otimes_{\Lambda_{R_{\infty}}} R_{\infty}[\Delta_1/\Delta_r]$ to $(\ref{orddRseq})$
	recovers the ordinary part of the scalar extension of $(\ref{finiteleveldRseq})$ to $R_{\infty}$.
\end{theorem}

The natural cup-product auto-duality of 
$(\ref{finiteleveldRseq})$ over $R_r':=R_r[\mu_N]$ induces a canonical
$\Lambda_{R_{\infty}'}$-linear and $\H^*$-equivariant auto-duality 
of (\ref{orddRseq}) which intertwines the dual semilinear action of $\Gamma \times \Gal(K_0'/K_0)\simeq \Gal(K_{\infty}'/K_0)$
with a certain $\H^*$-valued twist of its standard action; see \cite[Proposition 3.2.4]{CaisHida1}
for the precise statement.
We moreover proved that, as one would expect, the $\Lambda_{R_{\infty}}$-module $e^*H^0(\omega)$ 
is canonically isomorphic to the  module $eS(N,\Lambda_{R_{\infty}})$ of ordinary $\Lambda_{R_{\infty}}$-adic
cusp forms of tame level $N$; see \cite[Corollary 3.3.5]{CaisHida1}.

\subsection{Results}\label{resultsintro}

In this paper, we cmplete our study of the geometry and $\Lambda$-adic Hodge 
theory of Hida families begun in \cite{CaisHida1} by constructing the crystalline counterpart
to Hida's ordinary $\Lambda$-adic \'etale cohomology,  Ohta's 
$\Lambda$-adic Hodge cohomology, and our $\Lambda$-adic de Rham cohomology.
Via a careful study of the geometry of modular curves and abelian varieties
and comparison isomorphisms in integral $p$-adic cohomology, 
we prove the appropriate control and finiteness theorems,
and a suitable $\Lambda$-adic version of every integral comparison isomorphism 
one could hope for.
In particular, we are able to recover the entire family
of $p$-adic Galois representations $\rho_{p}$ (and not just its
semisimplification) from our $\Lambda$-adic crystalline cohomology.
A remarkable byproduct of our work is a {\em cohomological} construction
of 
the family of \'etale $(\varphi,\Gamma)$-modules attached to $e^*H^1_{\et}$
by Dee \cite{Dee}.
  As an application of our theory, we give a new 
and purely geometric proof of Hida's freeness and control theorems
for $e^*H^1_{\et}$. 

In order to survey our main results more precisely, we introduce
some notation.   Throughout this paper, we fix a prime $p>2$ and a positive integer
$N$ with $Np > 4$.
Fix an algebraic closure $\Qbar_p$
of $\Q_p$ 
as well as a $p$-power compatible sequence 
$\{\varepsilon^{(r)}\}_{r\ge 0}$ of primitive $p^r$-th roots of unity in $\Qbar_p$.
As above, we set $K_r:=\Q_p(\mu_{p^r})$ and $K_r':=K_r(\mu_N)$, and we
write $R_r$ and $R_r'$ for the rings of integers
in $K_r$ and $K_r'$, respectively.  
Denote by $\scrG_{\Q_p}:=\Gal(\Qbar_p/\Q_p)$ the absolute Galois group
and by $\scrH$
the kernel of the $p$-adic cyclotomic character
$\chi: \scrG_{\Q_p}\rightarrow \Z_p^{\times}$.
Using that $K_0'/\Q_p$ is unramified, we canonically identify $\Gamma=\scrG_{\Q_p}/\scrH$
with $\Gal(K_{\infty}'/K_0')$.
We will denote by $\langle u\rangle$ (respectively $\langle v\rangle_N)$
the diamond operator\footnote{Note that $\langle u^{-1}\rangle=\langle u\rangle^*$
and $\langle v^{-1}\rangle_N = \langle v\rangle_N^*$, where $\langle\cdot\rangle^*$
and $\langle \cdot\rangle_N^*$
are the adjoint diamond operators; see \cite[\S2.2]{CaisHida1}.
}
in $\H^*$ attached to $u^{-1}\in \Z_p^{\times}$
(respectively $v^{-1}\in (\Z/N\Z)^{\times}$) and write 
$\Delta_r$ for the image of the restriction of $\langle\cdot\rangle :\Z_p^{\times}\hookrightarrow \H^*$
to $1+p^r\Z_p\subseteq \Z_p^{\times}$.  For convenience, we put $\Delta:=\Delta_1$,
and for any ring $A$ we write
$\Lambda_{A}:=\varprojlim_r A[\Delta/\Delta_r]$ for the completed group ring
on $\Delta$ over $A$; if $\varphi$ is an endomorphism of $A$, we again write $\varphi$
for the induced endomorphism of $\Lambda_A$ that acts as the identity on $\Delta$.
For any ring homomorphism $A\rightarrow B$,
we will write $(\cdot)_B:=(\cdot)\otimes_A B$ and $(\cdot)_B^{\vee}:=\Hom_B((\cdot)\otimes_A B , B)$ for these
functors from $A$-modules to $B$-modules.\footnote{This convention is unfortunately
somewhat at odds with our notation $\Lambda_A$, which (as an $A$-module) is in general neither
the tensor product $\Lambda\otimes_{\Z_p} A$ nor (unless $A$ is a complete $\Z_p$-algebra)
the completed tensor product $\Lambda \wh{\otimes}_{\Z_p} A$; we hope that this small abuse 
causes no confusion.}
  If $G$ is any group of automorphisms of $A$ and $M$
is an $A$-module with a semilinear action of $G$, for any ``crossed" 
homomorphism\footnote{That is, $\psi(\sigma\tau) = \psi(\sigma)\cdot\sigma\psi(\tau)$ for all $\sigma,\tau\in\Gamma$,} 
$\psi:G\rightarrow A^{\times}$ we will write $M(\psi)$ for the 
$A$-module $M$ with ``twisted" semilinear $G$-action given by $g\cdot m:=\psi(g)g m$.
Finally, we denote by $X_r:=X_1(Np^r)$ the 
usual modular curve over $\Q$ classifying (generalized) elliptic curves
with a $[\mu_{Np^r}]$-structure,
and by $J_r:=J_1(Np^r)$ its Jacobian.

We analyze the tower of $p$-divisible groups attached to the ``good quotient" modular abelian
varieties introduced by Mazur-Wiles \cite{MW-Iwasawa}. To avoid technical complications with logarithmic
$p$-divisible groups, following \cite{MW-Hida} and \cite{OhtaEichler}, we will henceforth remove the trivial tame character by working with the sub-idempotent ${e^*}'$ of $e^*$ corresponding to projection to the part
where $\mu_{p-1}\subseteq \Z_p^{\times}$ acts {\em non}-trivially via the diamond operators.
As is well-known (e.g. \cite[\S9]{HidaGalois} and \cite[Chapter 3, \S2]{MW-Iwasawa}), the $p$-divisible group $G_r:={e^*}'J_r[p^{\infty}]$
over $\Q$ extends to a $p$-divisible group $\G_r$ over $R_r$, and we write 
$\o{\G}_r:=\G_r\times_{R_r} \F_p$ for its special fiber.  Denoting by $\D(\cdot)$
the contravariant Dieudonn\'e module functor on $p$-divisible groups over $\F_p$,
we form the projective limits
\begin{equation}
	\D_{\infty}^{\star}:=\varprojlim_r \D(\o{\G}_r^{\star})\quad\text{for}\quad \star\in \{\et,\mult,\Null\},
	\label{DlimitsDef}
\end{equation}
taken along the mappings induced by 
$\o{\G}_r\rightarrow \o{\G}_{r+1}$.  Each of these is naturally a $\Lambda$-module
equipped with linear (!) Frobenius $F$ and Verscheibung $V$ morphisms satisfying $FV=VF=p$,
as well as a linear action of $\H^*$ and a ``geometric inertia" action of $\Gamma$, which reflects
the fact that the generic fiber of $\G_r$ descends to $\Q_p$.
The $\Lambda$-modules (\ref{DlimitsDef}) have the expected structure (see Theorem \ref{MainDieudonne}):

\begin{theorem}\label{DieudonneMainThm}
	There is a canonical split short exact sequence of finite and free $\Lambda$-modules	
	\begin{equation}
		\xymatrix{
			0 \ar[r] & {\D^{\et}_{\infty}} \ar[r] & {\D_{\infty}} \ar[r] & {\D_{\infty}^{\mult}} \ar[r] & 0
			}\label{Dieudonneseq}
	\end{equation} 
	with linear $\H^*$ and $\Gamma$-actions.
	As a $\Lambda$-module, $\D_{\infty}$ is free of rank $2d'$, while $\D_{\infty}^{\et}$
	and $\D_{\infty}^{\mult}$ are free of rank $d'$, where $d':=\sum_{k=3}^p \dim_{\F_p} S_k(\Upgamma_1(N);\F_p)^{\ord}$.
	For $\star\in \{\mult,\et,\Null\}$, there are canonical isomorphisms
	\begin{equation*}
		\D_{\infty}^{\star}\tens_{\Lambda} \Z_p[\Delta/\Delta_r] \simeq \D(\o{\G}_r^{\star})
	\end{equation*}
	which are compatible with the extra structures.
	Via the canonical splitting of $(\ref{Dieudonneseq})$, $\D_{\infty}^{\star}$ for $\star=\et$
	$($respetively $\star=\mult$$)$ is
	identified with the maximal subpace of $\D_{\infty}$ on which $F$ $($respectively $V$$)$ acts
	invertibly .	
	The Hecke operator $U_p^*\in \H^*$ acts as $F$ on $\D_{\infty}^{\et}$ and as $\langle p\rangle_N V$ on 
	$\D_{\infty}^{\mult}$,
	while $\Gamma$ acts trivially on $\D_{\infty}^{\et}$ and via $\langle \chi(\cdot)\rangle^{-1}$
	on $\D_{\infty}^{\mult}$.
\end{theorem}

The short exact sequence (\ref{Dieudonneseq}) is very nearly $\Lambda$-adically
auto-dual (see Proposition \ref{DieudonneDuality}):

\begin{theorem}\label{DDuality}
		There is a canonical $\H^*$-equivariant isomorphism of exact sequences of $\Lambda_{R_0'}$-modules
		\begin{equation*}
			\xymatrix{
	0 \ar[r] & {\D_{\infty}^{\et}(\langle \chi \rangle\langle a\rangle_N)_{\Lambda_{R_0'}}} \ar[r]\ar[d]^-{\simeq} & 
	{\D_{\infty}(\langle \chi \rangle\langle a\rangle_N)_{\Lambda_{R_0'}}}\ar[r]\ar[d]^-{\simeq} & 
	{\D_{\infty}^{\mult}(\langle \chi \rangle\langle a\rangle_N)_{\Lambda_{R_0'}}}\ar[r]\ar[d]^-{\simeq} & 0 \\		
		0\ar[r] & {(\D_{\infty}^{\mult})^{\vee}_{\Lambda_{R_0'}}} \ar[r] & 
				{(\D_{\infty})^{\vee}_{\Lambda_{R_0'}}} \ar[r] & 
				{(\D_{\infty}^{\et})^{\vee}_{\Lambda_{R_0'}}}\ar[r] & 0
			}
		\end{equation*}
		that is $\Gamma\times \Gal(K_0'/K_0)$-equivariant, 
		and intertwines $F$
		$($respectively $V$$)$ on the top row with $V^{\vee}$
		$($respectively $F^{\vee}$$)$ on the bottom.\footnote{Here, $F^{\vee}$ (respectively $V^{\vee}$)
		is the map taking a linear functional $f$ to $\varphi^{-1}\circ f\circ F$ 
		(respectively $\varphi\circ f\circ V$), where $\varphi$
		is the Frobenius automorphism of $R_0'=\Z_p[\mu_N]$.}	
\end{theorem}

In \cite{MW-Hida}, Mazur and Wiles relate the ordinary-filtration of ${e^*}'H^1_{\et}$
to the \'etale cohomology of the Igusa tower studies in \cite{MW-Analogies}.
We can likewise interpret the slope filtraton (\ref{Dieudonneseq})
in terms of the crystalline cohomology of the Igusa tower as follows.
For each $r$, let $I_r^{\infty}$ and $I_r^0$ be the two ``good" irreducible components of
$\X_r\times_{R_r}\F_r$ (see the discussion preceding Proposition \ref{GisOrdinary}), each of which is isomorphic to the Igusa curve $\Ig(p^r)$
of tame level $N$ and $p$-level $p^r$. For $\star\in \{0,\infty\}$ we form the projective limit
\begin{equation*} 
	H^1_{\cris}(I^{\star}):=\varprojlim_{r} H^1_{\cris}(I_r^{\star}/\Z_p);
\end{equation*}
with respect to the trace mappings on crystalline cohmology induced by the canonical
degeneracy maps on Igusa curves.
Then $H^1_{\cris}(I^{\star})$ is naturally a $\Lambda$-module with linear Frobenius $F$ and Verscheibung $V$ endomorphisms, and we write $H^1_{\cris}(I^{\star})^{V_{\ord}}$
(respecytively $H^1_{\cris}(I^{\star})^{F_{\ord}}$) for the maximal $V$- (respectively $F$-) stable 
submodule on which $V$ (respectively $F$) acts invertibly.
Letting $f'$ be the idempotent of $\Lambda$ corresponding to projection to the part
where $\mu_{p-1}\subseteq \Z_p^{\times}$ acts nontrivially via the diamond operators, we prove
(see Theorem \ref{DieudonneCrystalIgusa}):

\begin{theorem}
 There is a canonical isomorphism of $\Lambda$-modules, compatible with $F$ and $V,$
	\begin{equation}
		\D_{\infty} =\D_{\infty}^{\mult}\oplus \D_{\infty}^{\et}\simeq 
		f'H^1_{\cris}(I^{0})^{V_{\ord}} \oplus
						f'H^1_{\cris}(I^{\infty})^{F_{\ord}}.\label{crisIgusa}
	\end{equation}
	which preserves the direct sum decompositions of source and target.
	This isomorphism is Hecke and $\Gamma$-equivariant, with $U_p^*$ and $\Gamma$
	acting as $\langle p\rangle_N V\oplus F$ and 
	$ \langle \chi(\cdot)\rangle^{-1}\oplus \id$, respectively, 
	on each direct sum.	
\end{theorem}

We note that our ``Dieudonn\'e module" analogue (\ref{crisIgusa}) is a significant
sharpening of its \'etale counterpart \cite[\S4]{MW-Hida}, which is formulated
only up to isogeny (i.e. after inverting $p$).  From $\D_{\infty}$, 
we can recover the $\Lambda$-adic Hodge filtration (\ref{orddRseq}), so
the latter is canonically split (see Theorem \ref{dRtoDieudonneInfty}):

\begin{theorem}\label{dRtoDieudonne}
	There is a canonical $\Gamma$ and $\H^*$-equivariant isomorphism of 
	exact sequences 
	\begin{equation}
	\begin{gathered}
		\xymatrix{
		0 \ar[r] & {{e^*}'H^0(\omega)} \ar[r]\ar[d]^-{\simeq} & 
		{{e^*}'H^1_{\dR}} \ar[r]\ar[d]^-{\simeq} & {{e^*}'H^1(\O)} \ar[r]\ar[d]^-{\simeq} & 0 \\
		0 \ar[r] & {\D_{\infty}^{\mult}\tens_{\Lambda} \Lambda_{R_{\infty}}} \ar[r] &
		{\D_{\infty}\tens_{\Lambda} \Lambda_{R_{\infty}}} \ar[r] &
		{\D_{\infty}^{\et}\tens_{\Lambda} \Lambda_{R_{\infty}}} \ar[r] & 0
		}\label{dRcriscomparison}
	\end{gathered}
	\end{equation}
	where the mappings on bottom row are the canonical inclusion and projection morphisms
	corresponding to the direct sum decomposition $\D_{\infty}=\D_{\infty}^{\mult}\oplus \D_{\infty}^{\et}$.
	In particular, the Hodge filtration exact sequence $(\ref{orddRseq})$ is canonically 
	split, and admits a canonical descent to $\Lambda$.
\end{theorem}

	We remark that under the identification (\ref{dRcriscomparison}),  
	the Hodge filtration (\ref{orddRseq}) and slope filtration (\ref{Dieudonneseq})
	correspond, but in the opposite directions.  As a consequence of Theorem 
	\ref{dRtoDieudonne}, we deduce (see Corollary \ref{MFIgusaDieudonne} and
	Remark \ref{MFIgusaCrystal}):

\begin{corollary}
\label{OhtaCor}
	There is a canonical isomorphism of finite free $\Lambda$ $($respectively $\Lambda_{R_0'}$$)$-modules
	\begin{equation*}
		{e}'S(N,\Lambda) \simeq \D_{\infty}^{\mult}
		\qquad\text{respectively}\qquad
		e'\H\tens_{\Lambda} \Lambda_{R_0'} \simeq \D_{\infty}^{\et}(\langle a\rangle_N)\tens_{\Lambda}{\Lambda_{R_0'}}
	\end{equation*}
	that  intertwines $T\in \H:=\varprojlim \H_r$ with $T^*\in \H^*$, where we let
	$U_p^*$ act as $\langle p\rangle_N V$ on $\D_{\infty}^{\mult}$ and as $F$ on $\D_{\infty}^{\et}$. 
	The second of these isomorphisms is in addition $\Gal(K_0'/K_0)$-equivariant.
\end{corollary}	
	
We are also able to recover the semisimplification of ${e^*}'H^1_{\et}$ from $\D_{\infty}$.
Writing $\I\subseteq \scrG_{\Q_p}$ 
for the inertia subgroup at $p$, for any $\Z_p[\scrG_{\Q_p}]$-module $M$, we denote by $M^{\I}$ (respectively $M_{\I}:=M/M^{\I}$)
the sub (respectively quotient) module of invariants (respectively covariants) under $\I$. 

\begin{theorem}\label{FiltrationRecover}
	There are canonical isomorphisms of $\Lambda_{W(\o{\F}_p)}$-modules
	with linear $\H^*$-action and semilinear actions of $F$, $V$, and $\scrG_{\Q_p}$
	\begin{subequations}
		\begin{equation}
			\D_{\infty}^{\et} \tens_{\Lambda} \Lambda_{W(\o{\F}_p)} 
			\simeq ({e^*}'H^1_{\et})^{\I}\tens_{\Lambda} \Lambda_{W(\o{\F}_p)} 
			\label{inertialinvariants}
		\end{equation}
		and
		\begin{equation}
			\D_{\infty}^{\mult}(-1) \tens_{\Lambda} \Lambda_{W(\o{\F}_p)} 
			\simeq ({e^*}'H^1_{\et})_{\I}\tens_{\Lambda} \Lambda_{W(\o{\F}_p)}.
			\label{inertialcovariants}
		\end{equation}
	\end{subequations}
	Writing $\sigma$ for the 
	Frobenius automorphism of $W(\o{\F}_p)$,
	the isomorphism $(\ref{inertialinvariants})$ intertwines $F\otimes \sigma$ with $\id\otimes\sigma$
	and $\id\otimes g$ with $g\otimes g$ for $g\in \scrG_{\Q_p}$, whereas $(\ref{inertialcovariants})$ 
	intertwines $V\otimes \sigma^{-1}$ with $\id\otimes\sigma^{-1}$ and $g\otimes g$ with $g\otimes g$,
	where $g\in\scrG_{\Q_p}$ acts on the Tate twist 
	$\D_{\infty}^{\mult}(-1):=\D_{\infty}^{\mult}\otimes_{\Z_p}\Z_p(-1)$ as
	$\langle \chi(g)^{-1}\rangle \otimes \chi(g)^{-1}$.
\end{theorem}

Theorem \ref{FiltrationRecover} gives the following ``explicit" description
of the semisimplification of ${e^*}'H^1_{\et}$:

\begin{corollary}
	For any $T\in (e^*\H^{*})^{\times}$, let 
	$\lambda(T):\scrG_{\Q_p}\rightarrow e^*\H^{*}$ be the 
	unique continuous $($for the $p$-adic topology on $e^*\H^{*}$$)$
	unramified character whose value on $($any lift of$)$
	$\mathrm{Frob}_p$ is $T$.
	Then $\scrG_{\Q_p}$ acts on $({e^*}'H^1_{\et})^{\I}$ through the
	character $\lambda({U_p^*}^{-1})$ and on $({e^*}'H^1_{\et})_{\I}$ 
	through $\chi^{-1} \cdot \langle \chi^{-1}\rangle \lambda(\langle p\rangle_N^{-1}U_p^*)$.
\end{corollary}

We remark that Corollary \ref{OhtaCor} and Theorem \ref{FiltrationRecover}
combined give a refinement of the main result of \cite{OhtaEichler}.
We are furthermore able to recover the main theorem of \cite{MW-Hida}
(that the ordinary filtration of ${e^*}'H^1_{\et}$ interpolates
$p$-adic analytically):

\begin{corollary}\label{MWmainThmCor}
	Let $d'$ be the integer of Theorem $\ref{DieudonneMainThm}$.  Then each of
	$({e^*}'H^1_{\et})^{\I}$ and $({e^*}'H^1_{\et})_{\I}$ is a free
	$\Lambda$-module of rank $d'$, and for each $r\ge 1$ there are canonical
	$\H^*$ and $\scrG_{\Q_p}$-equivariant isomorphisms of $\Z_p[\Delta/\Delta_r]$-modules
	\begin{subequations}
	\begin{equation}	
		({e^*}'H^1_{\et})^{\I} \tens_{\Lambda} \Z_p[\Delta/\Delta_r] \simeq 
		{e^*}'H^1_{\et}({X_r}_{\Qbar_p},\Z_p)^{\I}\label{HidaResultSub}
	\end{equation}
	\begin{equation}
		({e^*}'H^1_{\et})_{\I} \tens_{\Lambda} \Z_p[\Delta/\Delta_r] \simeq 
		{e^*}'H^1_{\et}({X_r}_{\Qbar_p},\Z_p)_{\I}
		\label{HidaResultQuo}
	\end{equation}
	\end{subequations}
\end{corollary}

To recover the full $\Lambda$-adic local Galois representation ${e^*}'H^1_{\et}$,
rather than just its semisimplification,
it is necessary to work with the full Dieudonn\'e {\em crystal} of $\G_r$ over $R_r$.
Following Faltings \cite{Faltings} and Breuil (e.g. \cite{Breuil}), this is equivalent
to studying the evaluation of the Dieudonn\'e crystal of $\G_r\times_{R_r} R_r/pR_r$
on the ``universal" divided power thickening $S_r\twoheadrightarrow R_r/pR_r$,
where $S_r$ is the $p$-adically completed PD-hull
of the surjection $\Z_p[\![u_r]\!]\twoheadrightarrow R_r$
sending $u_r$ to $\varepsilon^{(r)}-1$.  As the rings $S_r$ are too unwieldly
to directly construct a good crystalline analogue of Hida's
ordinary \'etale cohomology, we must functorially descend  
the ``filtered $S_r$-module" attached to $\G_r$ to the much simpler ring $\s_r:=\Z_p[\![u_r]\!]$.  
While such a descent is provided
(in rather different ways) by the work of Breuil--Kisin and Berger--Wach, neither of these
frameworks is suitable for our application: it is essential for us
that the formation of this descent to $\s_r$ commute with base change as one moves up
the cyclotomic tower, and it is not at all clear that this holds for Breuil--Kisin modules 
or for the Wach modules of Berger.  Instead, we use the theory of \cite{CaisLau}, which works with frames and 
windows \`a la Lau and Zink to provide the desired functorial descent to a ``$(\varphi,\Gamma)$-module" $\m_r(\G_r)$
over $\s_r$.  
We view $\s_r$
as a $\Z_p$-subalgebra of $\s_{r+1}$ via the map sending $u_r$ to $\varphi(u_{r+1}):=(1+u_{r+1})^p -1$,
and we write $\s_{\infty}:=\varinjlim \s_r$ for 
the rising union\footnote{As explained in Remark \ref{Slimits}, 
the $p$-adic completion of $\s_{\infty}$ is actually a very nice ring:
it is canonically and Frobenius equivariantly isomorphic to $W(\F_p[\![u_0]\!]^{\perf})$,
for $\F_p[\![u_0]\!]^{\perf}$ the perfect closure of the $\F_p$-algebra $\F_p[\![u_0]\!]$.
}
of the $\s_r$, equiped with its Frobenius {\em automorphism} $\varphi$ and commuting action of 
$\Gamma$ determined by $\gamma u_r:=(1+u_r)^{\chi(\gamma)} - 1$.
We then form the projective limits
\begin{equation*}
	\m_{\infty}^{\star}:=\varprojlim (\m_r(\G_r^{\star})\tens_{\s_r} \s_{\infty})\quad\text{for}\quad \star\in\{\et,\mult,\Null\}
\end{equation*}
taken along the mappings induced by $\G_{r}\times_{R_r} R_{r+1}\rightarrow \G_{r+1}$
via the functoriality of $\m_r(\cdot)$
and its compatibility with base change.
These are $\Lambda_{\s_{\infty}}$-modules equipped with a semilinear action of $\Gamma$,
a linear and commuting action of $\H^*$, and a $\varphi$ (respectively $\varphi^{-1}$) semilinear endomorphism $F$ (respectively $V$)
satisfying $FV=\omega$ and $VF = \varphi^{-1}(\omega)$, for 
$\omega:=\varphi(u_1)/u_1 = u_0/\varphi^{-1}(u_0)\in \s_{\infty}$,
and they provide
our crystalline analogue of Hida's ordinary \'etale cohomology (see Theorem \ref{MainThmCrystal}):

\begin{theorem}
	There is a canonical short exact sequence of finite free $\Lambda_{\s_{\infty}}$-modules
	with linear $\H^*$-action, semilinear $\Gamma$-action, and semilinear
	endomorphisms $F$, $V$ satisfying $FV=\omega$, $VF=\varphi^{-1}(\omega)$
	\begin{equation}
		\xymatrix{
			0 \ar[r] & {\m_{\infty}^{\et}} \ar[r] & {\m_{\infty}} \ar[r] & {\m_{\infty}^{\mult}} \ar[r] & 0
		}.\label{CrystallineAnalogue}
	\end{equation}
	Each of $\m_{\infty}^{\star}$ for $\star\in \{\et,\mult\}$ is free of rank $d'$ over 
	$\Lambda_{\s_{\infty}}$, while $\m_{\infty}$ is free of rank $2d'$, where
	$d'$ is as in Theorem $\ref{DieudonneMainThm}$.  Extending scalars on $(\ref{CrystallineAnalogue})$
	along the canonical surjection 
	$\Lambda_{\s_{\infty}}\twoheadrightarrow \s_{\infty}[\Delta/\Delta_r]$ yields the short exact 
	sequence
	\begin{equation*}
		\xymatrix{
			0 \ar[r] & {\m_r(\G_r^{\et})\tens_{\s_r} \s_{\infty}} \ar[r] &
			{\m_r(\G_r)\tens_{\s_r} \s_{\infty}} \ar[r] &
			{\m_r(\G_r^{\mult})\tens_{\s_r} \s_{\infty}} \ar[r] & 
			0
		}
	\end{equation*}
	compatibly with $\H^*$, $\Gamma$, $F$ and $V$.  The Frobenius endomorphism $F$
	commutes with $\H^*$ and $\Gamma$, whereas the Verscheibung $V$ commutes with $\H^*$ and satisfies
	$V\gamma = \varphi^{-1}(\omega/\gamma\omega)\cdot \gamma V$ for all $\gamma\in \Gamma$.
\end{theorem}

Again, in the spirit of Theorem \ref{DDuality} and \cite[Proposition 3.2.4]{CaisHida1}, 
there is a corresponding ``autoduality" result for $\m_{\infty}$
(see Theorem \ref{CrystalDuality}).  To state it, we must work over
$\s_{\infty}':=\varinjlim_r \Z_p[\mu_N][\![u_r]\!]$, with the inductive limit taken along the 
$\Z_p$-algebra maps sending $u_r$ to $\varphi(u_{r+1})$.

\begin{theorem}
		Let $\mu:\Gamma\rightarrow \Lambda_{\s_{\infty}}^{\times}$ be the crossed homomorphism 
		given by $\mu(\gamma):=\frac{u_1}{\gamma u_1}\chi(\gamma) \langle \chi(\gamma)\rangle$.
		There is a canonical $\H^*$ and $\Gal(K_{\infty}'/K_0)$-compatible isomorphism of 
		exact sequences
		\begin{equation*}
		\begin{gathered}
			\xymatrix{
			0\ar[r] & {\m_{\infty}^{\et}(\mu \langle a\rangle_N)_{\Lambda_{\s_{\infty}'}}} \ar[r]\ar[d]_-{\simeq} & 
			{\m_{\infty}(\mu \langle a\rangle_N)_{\Lambda_{\s_{\infty}'}}} \ar[r]\ar[d]_-{\simeq} & 
				{\m_{\infty}^{\mult}(\mu \langle a\rangle_N)_{\Lambda_{\s_{\infty}'}}} \ar[r]\ar[d]_-{\simeq} & 0\\
	0\ar[r] & {(\m_{\infty}^{\mult})_{\Lambda_{\s_{\infty}'}}^{\vee}} \ar[r] & 
				{(\m_{\infty})_{\Lambda_{\s_{\infty}'}}^{\vee}} \ar[r] & 
				{(\m_{\infty}^{\et})_{\Lambda_{\s_{\infty}'}}^{\vee}} \ar[r] & 0 
		}
		\end{gathered}
		\end{equation*}
		intertwining $F$ and $V$ on the top row with  
		$V^{\vee}$ and $F^{\vee}$, respectively,  on the bottom.  The action 
		of $\Gal(K_{\infty}'/K_0)$ on the bottom row is the standard one
		$\gamma\cdot f:=\gamma f\gamma^{-1}$ on linear duals.
\end{theorem}

The $\Lambda_{\s_{\infty}}$-modules $\m_{\infty}^{\et}$ and $\m_{\infty}^{\mult}$
have a particularly simple structure (see Theorem \ref{etmultdescent}):

\begin{theorem}
	There are canonical $\H^*$, $\Gamma$, $F$ and $V$-equivariant isomorphisms
	of $\Lambda_{\s_{\infty}}$-modules
	\begin{subequations}
	\begin{equation}
		\m_{\infty}^{\et} \simeq \D_{\infty}^{\et}\tens_{\Lambda} \Lambda_{\s_{\infty}},
	\end{equation}	
	intertwining $F$ and $V$ with 
	$F\otimes \varphi$ and $F^{-1}\otimes \varphi^{-1}(\omega)\cdot \varphi^{-1}$, respectively,
	and $\gamma\in \Gamma$
	with $\gamma\otimes\gamma$, and
	\begin{equation}
		\m_{\infty}^{\mult}\simeq \D_{\infty}^{\mult}\tens_{\Lambda} \Lambda_{\s_{\infty}},	
	\end{equation}
	intertwing $F$ and $V$ with $V^{-1} \otimes \omega \cdot\varphi$
	and $V\otimes\varphi^{-1}$, respectively,
	and $\gamma$ with $\gamma\otimes \chi(\gamma)^{-1} \gamma u_1/u_1$.
	In particular, $F$ $($respectively $V$$)$ 
	acts invertibly on $\m_{\infty}^{\et}$ $($respectively $\m_{\infty}^{\mult}$$)$.
\end{subequations}
\end{theorem}

From $\m_{\infty}$, we can recover $\D_{\infty}$ and ${e^*}'H^1_{\dR}$,
with their additional structures (see Theorem \ref{SRecovery}):
\begin{theorem}\label{MinftySpecialize}
	Viewing $\Lambda$ as a $\Lambda_{\s_{\infty}}$-algebra via the map induced by $u_r\mapsto 0$,
	there is a canonical isomorphism of short exact sequences of finite free $\Lambda$-modules
	\begin{equation*}
		\xymatrix{
			0 \ar[r] & {\m_{\infty}^{\et}\tens_{\Lambda_{\s_{\infty}}} \Lambda}\ar[d]_-{\simeq} \ar[r] &
			{\m_{\infty}\tens_{\Lambda_{\s_{\infty}}} \Lambda}\ar[r] \ar[d]_-{\simeq}&
			{\m_{\infty}^{\mult}\tens_{\Lambda_{\s_{\infty}}} \Lambda} \ar[r]\ar[d]_-{\simeq} & 0\\
			0 \ar[r] & {\D_{\infty}^{\et}} \ar[r] & {\D_{\infty}} \ar[r] &
			{\D_{\infty}^{\mult}} \ar[r] & 0
		}
	\end{equation*}
	which is $\Gamma$ and $\H^*$-equivariant and carries $F\otimes 1$ to $F$
	and $V\otimes 1$ to $V$.
	Viewing $\Lambda_{R_{\infty}}$ as a $\Lambda_{\s_{\infty}}$-algebra via the map 
	$u_r\mapsto (\varepsilon^{(r)})^p - 1$, there is a canonical
	isomorphism of short exact sequences of 
	$\Lambda_{R_{\infty}}$-modules
	\begin{equation*}
		\xymatrix{
				0 \ar[r] & {\m_{\infty}^{\et}\tens_{\Lambda_{\s_{\infty}}} \Lambda_{R_{\infty}}}
				\ar[d]_-{\simeq} \ar[r] &
			{\m_{\infty}\tens_{\Lambda_{\s_{\infty}}} \Lambda_{R_{\infty}}}\ar[r] \ar[d]_-{\simeq}&
			{\m_{\infty}^{\mult}\tens_{\Lambda_{\s_{\infty}}} \Lambda_{R_{\infty}}} \ar[r]\ar[d]_-{\simeq} & 0\\
		0 \ar[r] & {{e^*}'H^1(\O)} \ar[r]_{i} & 
		{{e^*}'H^1_{\dR}} \ar[r]_-{j} & {{e^*}'H^0(\omega)} \ar[r] & 0 
		}
	\end{equation*}
	that is $\Gamma$ and $\H^*$-equivariant, 
	where $i$ and $j$ the splittings 
	given by Theorem $\ref{dRtoDieudonne}$.
\end{theorem}

To recover Hida's ordinary \'etale cohomology from $\m_{\infty}$,
we introduce the ``period" ring of Fontaine\footnote{Though we use the notation introduced by Berger and Colmez.} 
$\wt{\e}^+:=\varprojlim \O_{\C_p}/(p)$, with the projective limit
taken along the $p$-power mapping; this is a perfect valuation ring of characteristic $p$
equipped with a canonical action of $\scrG_{\Q_p}$ via ``coordinates".  We write $\wt{\e}$
for the fraction field of $\wt{\e}^+$ and
$\wt{\a}:=W(\wt{\e})$ for its ring of Witt vectors, equipped 
with its canonical Frobenius automorphism $\varphi$ and $\scrG_{\Q_p}$-action induced by Witt functoriality.
Our fixed choice of $p$-power compatible sequence $\{\varepsilon^{(r)}\}$
determines an element $\u{\varepsilon}:=(\varepsilon^{(r)}\bmod p)_{r\ge 0}$
of $\wt{\e}^+$, and we $\Z_p$-linearly embed $\s_{\infty}$ in $\wt{\a}$ via 
$u_r\mapsto \varphi^{-r}([\u{\varepsilon}]-1)$ where $[\cdot]$ is the Teichm\"uller
section.  This embedding is $\varphi$ and $\scrG_{\Q_p}$-compatible, with $\scrG_{\Q_p}$
acting on $\s_{\infty}$ through the quotient $\scrG_{\Q_p}\twoheadrightarrow \Gamma$.

\begin{theorem}
\label{RecoverEtale}
	Twisting the structure map $\s_{\infty}\rightarrow \wt{\a}$ by the Frobenius automorphism $\varphi$,
	there is a canonical isomorphism of short exact sequences of $\Lambda_{\wt{\a}}$-modules
	with $\H^*$-action
	\begin{equation}
	\begin{gathered}
		\xymatrix{
				0 \ar[r] & {\m_{\infty}^{\et}\tens_{\Lambda_{\s_{\infty}},\varphi} \Lambda_{\wt{\a}}}
				\ar[d]_-{\simeq} \ar[r] &
			{\m_{\infty}\tens_{\Lambda_{\s_{\infty}},\varphi} \Lambda_{\wt{\a}}}\ar[r] \ar[d]_-{\simeq}&
			{\m_{\infty}^{\mult}\tens_{\Lambda_{\s_{\infty}},\varphi} \Lambda_{\wt{\a}}} \ar[r]\ar[d]_-{\simeq} & 0\\
			0 \ar[r] & {({e^*}'H^1_{\et})^{\I}\tens_{\Lambda} \Lambda_{\wt{\a}}} \ar[r] & 
			{{e^*}'H^1_{\et}\tens_{\Lambda} \Lambda_{\wt{\a}}} \ar[r] &
			({e^*}'H^1_{\et})_{\I}\tens_{\Lambda} \Lambda_{\wt{\a}}\ar[r] & 0
		}\label{FinalComparisonIsom}
	\end{gathered}
	\end{equation}
	that is $\scrG_{\Q_p}$-equivariant for the ``diagonal" action of $\scrG_{\Q_p}$
	$($with $\scrG_{\Q_p}$ acting on $\m_{\infty}$ through $\Gamma$$)$
	and intertwines
	$F\otimes \varphi$ with $\id\otimes\varphi$ and $V\otimes\varphi^{-1}$ with $\id\otimes \omega\cdot\varphi^{-1}$.
	In particular, there is a canonical isomorphism of $\Lambda$-modules, compatible
	with the actions of $\H^*$ and $\scrG_{\Q_p}$,
	\begin{equation}
		{e^*}'H^1_{\et} \simeq \left( \m_{\infty}\tens_{\Lambda_{\s_{\infty}},\varphi} 
		\Lambda_{\wt{\a}}\right )^{F\otimes\varphi = 1}.\label{RecoverEtaleIsom}
	\end{equation}
\end{theorem}

Theorem \ref{RecoverEtale} allows us to give a new proof of Hida's finiteness and control theorems
for ${e^*}'H^1_{\et}$:

\begin{corollary}[Hida]\label{HidasThm}	
	Let $d'$ be as in Theorem $\ref{DieudonneMainThm}$.  Then
	${e^*}'H^1_{\et}$ is free $\Lambda$-module of rank $2d'$. 
	For each $r\ge 1$ there is a canonical isomorphism of $\Z_p[\Delta/\Delta_r]$-modules
	with linear $\H^*$ and $\scrG_{\Q_p}$-actions
	\begin{equation*}
		{e^*}'H^1_{\et} \tens_{\Lambda} \Z_p[\Delta/\Delta_r] \simeq {e^*}'H^1_{\et}({X_r}_{\Qbar_p},\Z_p)
	\end{equation*}
	which is moreover compatible with the isomorphisms $(\ref{HidaResultSub})$ and $(\ref{HidaResultQuo})$ 
	in the evident manner.
\end{corollary}

We also deduce a new proof of the following duality result
\cite[Theorem 4.3.1]{OhtaEichler} ({\em cf.} \cite[\S6]{MW-Hida}):

\begin{corollary}[Ohta]\label{OhtaDuality}
	Let $\nu:\scrG_{\Q_p}\rightarrow \H^*$ be the 
	character $\nu:=\chi\langle \chi\rangle \lambda(\langle p\rangle_N)$.
	There is a canonical $\H^*$ and $\scrG_{\Q_p}$-equivariant isomorphism
	of short exact sequences of $\Lambda$-modules 
	\begin{equation*}
		\xymatrix{
			0 \ar[r] & {({e^*}'H^1_{\et})^{\I}(\nu)}
			\ar[d]^-{\simeq} \ar[r] & 
			{{e^*}'H^1_{\et}(\nu)}\ar[d]^-{\simeq} \ar[r] &
			{({e^*}'H^1_{\et})_{\I}(\nu)}
			\ar[d]^-{\simeq}\ar[r] & 0 \\
			0 \ar[r] & {\Hom_{\Lambda}(({e^*}'H^1_{\et})_{\I},\Lambda)} \ar[r] & 
			{\Hom_{\Lambda}({e^*}'H^1_{\et},\Lambda)} \ar[r] &
			{\Hom_{\Lambda}(({e^*}'H^1_{\et})^{\I},\Lambda)}\ar[r] & 0
		}
	\end{equation*}
\end{corollary}

The $\Lambda$-adic splitting 
of the ordinary filtration of $e^*H^1_{\et}$ was considered by 
Ghate and Vatsal \cite{GhateVatsal}, who prove (under certain
technical hypotheses of ``deformation-theoretic nature")
that if the $\Lambda$-adic family $\scrF$ associated to 
a cuspidal eigenform $f$
is primitive and $p$-distinguished, then the associated
$\Lambda$-adic local Galois representation $\rho_{\scrF,p}$
is split if and only if 
some arithmetic specialization of $\scrF$
has CM \cite[Theorem 13]{GhateVatsal}.
We interpret the $\Lambda$-adic splitting of the ordinary filtration
as follows:

\begin{theorem}\label{SplittingCriterion}
	The short exact sequence $(\ref{CrystallineAnalogue})$ admits
	a $\Lambda_{\s_{\infty}}$-linear splitting which is compatible with $F$, $V$,
	and $\Gamma$ if and only if the ordinary filtration of ${e^*}'H^1_{\et}$ 
	admits a $\Lambda$-linear spitting which is compatible with the action of $\scrG_{\Q_p}$.
\end{theorem}

\subsection{Overview of the article}\label{Overview}

Section \ref{Prelim} is preliminary: we first review in \S\ref{DDR}--\ref{Universal} some background
material on Dieudonn\'e modules and crystals, as well as the integral $p$-adic
cohomology theories of \cite{CaisDualizing} and \cite{CaisNeron}.
In \S\ref{PhiGammaCrystals}, we summarize the theory developed in \cite{CaisLau},
which uses Dieudonn\'e crystals of $p$-divisible groups to
provide a ``cohomological" construction of the $(\varphi,\Gamma)$-modules attached 
to potentially Barsotti--Tate representations.  We then specialize these results
to the case of ordinary $p$-divisible groups in \S\ref{pDivOrdSection}; 
it is precisely this theory
which allows us to construct our crystalline analogue of Hida's ordinary $\Lambda$-adic
\'etale cohomology.  


Section \ref{results} constitutes the main body of this paper, and the reader who 
is content to refer back to \S\ref{DDR}--\ref{pDivOrdSection} as needed should skip directly there.
In section \ref{BTfamily}, we study the tower of $p$-divisible groups whose cohomology 
allows us to construct our $\Lambda$-adic Dieudonn\'e and crystalline analogues of
Hida's \'etale cohomlogy in \S\ref{OrdDieuSection} and \S\ref{OrdSigmaSection}, respectively. 
We establish $\Lambda$-adic comparison isomorphisms between each of these cohomologies
using the integral comparison isomorphisms of \cite{CaisNeron} and \cite{CaisLau}, recalled in \S\ref{Universal}
and \S\ref{PhiGammaCrystals}--\ref{pDivOrdSection}, respectively.  
This enables us to give a new proof of Hida's freeness and 
control theorems and of Ohta's duality theorem in \S\ref{OrdSigmaSection}.
A key technical ingredient in our proofs is the commutative algebra formalism 
developed in \cite[\S3.1]{CaisHida1} for dealing with projective limits of
cohomology and establishing appropriate ``freeness and control" theorems
by reduction to characteristic $p$.

As remarked in \S\ref{resultsintro}, and following \cite{OhtaEichler} and \cite{MW-Hida}, our construction
of the $\Lambda$-adic Dieudonn\'e and crystalline counterparts to Hida's \'etale cohomology
excludes the trivial eigenspace for the action of $\mu_{p-1}\subseteq \Z_p^{\times}$
so as to avoid technical complications with logarithmic $p$-divisible groups.  
In \cite{Ohta2}, Ohta uses the ``fixed part" (in the sense of Grothendieck \cite[2.2.3]{GroModeles})
of N\'eron models with semiabelian reduction to extend his results on 
$\Lambda$-adic Hodge cohomology to allow trivial tame nebentype character.  
We are confident that by using Kato's logarithmic Dieudonn\'e theory \cite{KatoDegen} 
one can appropriately generalize our results in \S\ref{OrdDieuSection} and \S\ref{OrdSigmaSection} to include
the missing eigenspace for the action of $\mu_{p-1}$.

\subsection{Notation}\label{Notation}

If $\varphi:A\rightarrow B$ is any map of rings, we will often write $M_B:=M\otimes_{A} B$ 
for the $B$-module induced from an $A$-module $M$ by extension of scalars.
When we wish to specify $\varphi$, we will write $M\otimes_{A,\varphi} B$.
Likewise, if $\varphi:T'\rightarrow T$ is any morphism of schemes, for any $T$-scheme $X$
we denote by $X_{T'}$ the base change of $X$ along $\varphi$.
If $f:X\rightarrow Y$ is any morphism of $T$-schemes,
we will write $f_{T'}: X_{T'}\rightarrow Y_{T'}$
for the morphism of $T'$-schemes obtained from $f$ by base change along $\varphi$.
When $T=\Spec(R)$ and $T'=\Spec(R')$ are affine, we abuse notation and write
$X_{R'}$ or $X\times_{R} R'$ for $X_{T'}$.
We frequently work with schemes over a discrete valuation ring $R$, and
will write $\X,\Y,\ldots$ for schemes over $\Spec(R)$, reserving
$X,Y,\ldots$ (respectively $\o{\X},\o{\Y},\ldots$) 
for their generic (respectively special) fibers.
As this article is a continuation of \cite{CaisHida1}, we will freely use the notation
and conventions therein.

\subsection{Acknowledgements}

It is a pleasure to thank Laurent Berger, Brian Conrad, Adrian Iovita, Tong Liu,
and Jacques Tilouine for many enlightening conversations and correspondence.
I am especially grateful to Haruzo Hida---to whom this paper is dedicated---for
his willingness to answer many questions concerning his unpublished
notes \cite{HidaNotes} and \cite{HidaNotes2}.\footnote{Hida has recently revised and expanded 
these very interesting notes
in his preprint \cite{HidaLambdaBT}, which includes new proofs
and refinements of some of the background material in \S\ref{BTfamily}
of the present paper.
} 

After a preliminary version of this paper was submitted for publication,
we learned of the preprint \cite{Wake}, in which Wake obtains
a new proof of Ohta's Eichler--Shimura isomorphism as well as
some refinements of Hida's finiteness and control theorems.
Wake's methods are largely different from ours, as he works with the $p$-adic \'etale
cohomology of the special fibers of modular curves, while we instead
use their mod $p$ de Rham cohomology in \cite{CaisHida1} and 
the theory of \cite{CaisLau} in the present article to compare
our $\Lambda$-adic de Rham cohomology with Hida's $\Lambda$-adic \'etale cohomology.
Our cohomological construction via Dieudonn\'e crystals of the family of \'etale $(\varphi,\Gamma)$-modules
attached to Hida's ordinary $\Lambda$-adic \'etale cohomology by \cite{Dee}
appears to be entirely new.

\tableofcontents

\section{Dieudonn\'e Crystals and Dieudonn\'{e} Modules}\label{Prelim}

This section is devoted to recalling the geometric 
background we will need in our constructions.  Much (though not all)
of this material is contained in \cite{CaisDualizing}, \cite{CaisNeron},
and \cite{CaisLau}. 

\subsection{Dieudonn\'e modules and {d}e~Rham cohomology}\label{DDR}

Let $k$ be a perfect field of characteristic $p$ and $X$ a smooth and proper curve over $k$.
We begin by recalling the relation between the de Rham cohomology of $X$ over $k$
and the Dieudonn\'e module of the $p$-divisible group of the Jacobian of $X$.  
Let us write $H(X/k)$ for the three-term ``Hodge filtration" exact sequence
\begin{equation*}
	\xymatrix{
	0\ar[r] & {H^0(X,\Omega^1_{X/k})} \ar[r] & {H^1_{\dR}(X/k)} \ar[r] & {H^1(X,\O_X)}\ar[r] & 0.
	}
\end{equation*}
Pullback by the absolute Frobenius gives an endomorphism
of $F:H(X/k)\rightarrow H(X/k)$ that is semilinear over the $p$-power Frobenius
automorphism $\varphi$ of $k$. 
Under the canonical cup-product autoduality of $H(X/k)$,
we obtain $\varphi^{-1}$-semilinear endomorphism
\begin{equation}
	\xymatrix{
		{V:={F}_*: H^1_{\dR}(X/k)} \ar[r] & {H^1_{\dR}(X/k)}
		}\label{CartierOndR}
\end{equation}
whose restriction to $H^0(X,\Omega^1_{X/k})$ coincides with the Cartier operator
\cite[\S2.3]{CaisHida1}.  
Let $A$ be the ``Dieudonn\'e ring", {\em i.e.}~the (noncommutative if $k\neq \F_p$)
ring $A:=W(k)[F,V]$, where  
 $F$, $V$ satisfy $FV=VF=p$, 
$F\alpha=\varphi(\alpha)F$, and $V\alpha=\varphi^{-1}(\alpha)V$ for all $\alpha\in W(k)$.
We view $H^1_{\dR}(X/k)$ as a left $A$-module 
in the obvious way.  
By Fitting's Lemma \cite[Lemma 2.3.3]{CaisHida1}, 
for $f=F$ or $V$,
any finite left $A$-module $M$
admits a canonical direct sum decomposition
\begin{equation}
		M=M^{f_{\ord}}\oplus M^{f_{\nil}}
		\label{ordnotation}
\end{equation}
where $M^{f_{\ord}}$ (respectively $M^{f_{\nil}}$) is the maximal $A$-submodule of $M$
on which $f$ is bijective (respectively $p$-adically topologically nilpotent).

\begin{proposition}[Oda]\label{OdaDieudonne}
	Let $J:=\Pic^0_{X/k}$ be the Jacobian of $X$ over $k$ and $G:=J[p^{\infty}]$
	its $p$-divisible group.  Denote by $\D(\cdot)$ the contravariant
	Dieudonn\'e crystal $($see $(\ref{DieudonneDef})$ below$)$, so the Dieudonn\'e module $\D(G)_W$ is naturally a left $A$-module, finite and
	free over $W:=W(k)$.  
	\begin{enumerate}
		\item \label{OdaIsom} There are canonical isomorphisms of left $A$-modules
		\begin{equation*}
			H^1_{\dR}(X/k)\simeq \D(J[p])_{k}\simeq \D(G)_k.
		\end{equation*}
		\item For any finite morphism $\rho:Y\rightarrow X$ of smooth and proper curves
		over $k$, the identification of $(\ref{OdaIsom})$ intertwines $\rho_*$
		with $\D(\Pic^0(\rho))$ and $\rho^*$ with $\D(\Alb(\rho))$.\label{OdaIsomFunctoriality}
		\item \label{GetaleGmult} Let $G=G^{\et}\times G^{\mult}\times G^{\loc}$
		be the canonical direct product decomposition of $G$ into its maximal \'etale,
		multiplicative, and local-local subgroups. 
		Via the identification of $(\ref{OdaIsom})$, the canonical mappings
		in the exact sequence $H(X/k)$ induce natural isomorphisms of left $A$-modules
		\begin{equation*}
			H^0(X,\Omega^1_{X/k})^{V_{\ord}} \simeq \D(G^{\mult})_k
			\quad\text{and}\quad
			H^1(X,\O_X)^{F_{\ord}} \simeq \D(G^{\et})_k
		\end{equation*}
		\item The isomorphisms of $(\ref{GetaleGmult})$ are dual to each other, using the 
		perfect duality on cohomology induced by the cup-product pairing \cite[Remark 2.3.3]{CaisHida1}
		and the identification $\D(G)_k^t\simeq  \D(G)_k$
		resulting from the compatibility of $\D(\cdot)_k$ with duality and
		the autoduality of $J$.\label{BBMDuality}		
	\end{enumerate}
\end{proposition}

\begin{proof}
	Using the characterizing properties of the Cartier operator
	defined by Oda \cite[Definition 5.5]{Oda} and the explicit
	description of the autoduality of $H^1_{\dR}(X/k)$ in terms
	of cup-product and residues, one checks that the endomorphism
	of $H^1_{\dR}(X/k)$ in \cite[Definition 5.6]{Oda} is adjoint
	to $F^*$, and therefore coincides with the endomorphism
	$V:={F}_*$ in (\ref{CartierOndR}); {\em cf.}
	the proof of \cite[Proposition 9]{SerreTopology}.

	We recall that one has a canonical isomorphism 
	\begin{equation}
		H^1_{\dR}(X/k)\simeq H^1_{\dR}(J/k)\label{dRIdenJac}
	\end{equation}
	which is compatible with Hodge filtrations and duality (using the canonical
	principal polarization to identify $J$ with its dual) and which, for	
	any finite morphism of smooth curves $\rho:Y\rightarrow X$ over $k$,
	intertwines $\rho_*$ with $\Pic^0(\rho)^*$ and $\rho^*$ with $\Alb(\rho)^*$; see  
	\cite[Proposition 5.4]{CaisNeron}, noting that the proof given there works over any field $k$,
	and {\em cf.}~Proposition \ref{intcompare}.  It follows from these compatibilities
	and the fact that the Cartier operator as defined in \cite[Definition 5.5]{Oda} is functorial
	that the identification (\ref{dRIdenJac}) is moreover an isomorphism of left $A$-modules,
	with the $A$-structure on $H^1_{\dR}(J/k)$ defined as in \cite[Definition 5.8]{Oda}.
	
	Now by \cite[Corollary 5.11]{Oda} and \cite[Theorem 4.2.14]{BBM}, 
	for any abelian variety $B$ over $k$, 
	there is a canonical isomorphism of left $A$-modules
	\begin{equation}
		H^1_{\dR}(B/k)\simeq \D(B)_k\label{AbVarDieuMod}
	\end{equation} 
	Using the definition of this isomorphism in Proposition 4.2 and Theorem 5.10 of \cite{Oda}, 
	it is straightforward (albeit tedious\footnote{Alternately, 
	one could appeal to 
	\cite{MM}, specifically to Chapter I, 4.1.7, 4.2.1, 3.2.3, 2.6.7
	and to Chapter II, \S13 and \S15 (see especially Chapter II, 13.4 and 1.6).  
	See also \S2.5 and \S4 of \cite{BBM}.
	}) 
	to check that for any homomorphism $h:B'\rightarrow B$
	of abelian varieties over $k$, the identification (\ref{AbVarDieuMod}) intertwines $h^*$ and $\D(h)$.  
	Combining (\ref{dRIdenJac}) and (\ref{AbVarDieuMod})
	yields (\ref{OdaIsom}) and (\ref{OdaIsomFunctoriality}).

	Now since $V={F}_*$ (respectively $F=F^*$) is the zero endomorphism of $H^1(X,\O_X)$ 
	(respectively $H^0(X,\O_X)$), the canonical mapping
	\begin{equation*}
		\xymatrix{
			{H^0(X,\Omega^1_{X/k})} \ar@{^{(}->}[r] & {H^1_{\dR}(X/k)\simeq \D(G)_k}
		}
		\quad\text{respectively}\quad
		\xymatrix{
			{\D(G)_k\simeq H^1_{\dR}(X/k)} \ar@{->>}[r] & {H^1(X,\O_X)}
		}
	\end{equation*}
	induces an isomorphism on $V$-ordinary (respectively $F$-ordinary) subspaces.  
	On the other hand, by Dieudonn\'e theory one knows that
	for {\em any} $p$-divisible group $H$, the semilinear endomorphism $V$ 
	(respectively $F$) of $\D(H)_W$
	is bijective if and only if $H$ is of multiplicative type (respectively \'etale).
	The (functorial) decomposition $G=G^{\et}\times G^{\mult}\times G^{\loc}$
	yields a natural isomorphism of left $A$-modules
	\begin{equation*}
		\D(G)_W\simeq \D(G^{\et})_W\oplus \D(G^{\mult})_W\oplus \D(G^{\loc})_W,
	\end{equation*}
	and it follows that the natural maps $\D(G^{\mult})_W\rightarrow \D(G)_W$,
	$\D(G)_W\rightarrow \D(G^{\et})_W$ induce isomorphisms 
	\begin{equation}
		\D(G^{\mult})_W \simeq \D(G)^{V_{\ord}}_W\quad\text{and}\quad
		\D(G)^{F_{\ord}}_W\simeq \D(G^{\et})_W,\label{VordMultFordEt}
	\end{equation}
	respectively,  which gives (\ref{GetaleGmult}).  Finally, (\ref{BBMDuality})
	follows from Proposition 5.3.13 and the proof of Theorem 5.1.8 in \cite{BBM},
	using Proposition 2.5.8 of {\em op.~cit.}~and the compatibility of the isomorphism
	(\ref{dRIdenJac}) with duality (for which see 
	\cite[Theorem 5.1]{colemanduality} and {\em cf.} \cite[Lemma 5.5]{CaisNeron}). 	
\end{proof}

\subsection{Universal vectorial extensions}\label{Universal}

We now study the mixed characteristic analogue of the situation
considered in \S\ref{DDR}.
Fix a discrete valuation ring $R$ with field of fractions $K$
of characteristic zero and perfect residue field $k$ of characteristic $p$.
Recall \cite[\S2.1]{CaisHida1} that by a {\em curve} over $S:=\Spec R$ we mean a
flat finitly presented local complete intersection $f:X\rightarrow S$
of relative dimension one with geometrically reduced fibers.
Let $f:X\rightarrow S$ be a normal and proper curve over $S$
with smooth and geometrically connected generic fiber $X_K$, and write
$\omega_{X/S}$
for the relative dualizing sheaf of $f$. 
The hypercohomology $H^i(X/R)$ of the 
two-term complex $\O_X\rightarrow \omega_{X/S}$
provides a canonical integral structure on the algebraic de Rham cohomology of 
the generic fiber $X_K$:

\begin{proposition}[{\cite[2.1.11]{CaisHida1}}]\label{HodgeIntEx}
	Let $f:X\rightarrow S$ be a normal curve that is proper over $S=\Spec(R)$.  
	There is a canonical short exact sequence of finite free $R$-modules, 
	which we denote $H(X/R)$,
		\begin{equation*}
			\xymatrix{
					0\ar[r] & {H^0(X,\omega_{X/S})} \ar[r] & {H^1(X/R)} \ar[r] & {H^1(X,\O_X)} \ar[r] & 0
			}
		\end{equation*}		
	that recovers the Hodge filtration of $H^1_{\dR}(X_K/K)$ after
	extending scalars to $K$ and is canonically $R$-linearly self-dual
	via the cup-product pairing on $H^1_{\dR}(X_K/K)$. 
	The exact sequence $H(X/R)$ is functorial in finite morphisms $\rho:Y\rightarrow X$
	of normal and proper $S$-curves via pullback $\rho^*$ and trace $\rho_*$; these 
	morphisms recover the usual pullback and trace mappings on Hodge filtrations after extending scalars 
	to $K$ and are adjoint with respect to the cup-product autoduality of $H(X/R)$.	
\end{proposition}

There is an alternate description of the short exact sequence $H(X/R)$ of Proposition
\ref{HodgeIntEx} in terms of 
Lie algebras and N\'eron models of Jacobians that will allow us to relate
this cohomology to Dieudonn\'e modules.  To explain this description and its connection
with crystals, we first recall some facts from \cite{MM} and \cite{CaisNeron}.

Fix a base scheme $T$, and let $G$ be an fppf sheaf of abelian groups over $T$.
A {\em vectorial extension} of $G$ is a short exact sequence (of fppf sheaves of abelian
groups)
\begin{equation}
	\xymatrix{
		0 \ar[r] & {V} \ar[r] & {E} \ar[r] & {G} \ar[r] & 0.
		}\label{extension}
\end{equation}
with $V$ a vector group (i.e. an fppf abelian sheaf which is locally represented by a product of $\Ga$'s).  
Assuming that $\Hom(G,V)=0$ for all vector groups $V$, we say that a vectorial extension 
(\ref{extension}) is {\em universal} if, for any vector group $V'$ over $T$, 
the pushout map $\Hom_T(V,V')\rightarrow \Ext^1_T(G,V')$
is an isomorphism.  When a universal vectorial extension of $G$ exists, it is
unique up to canonical isomorphism and covariantly functorial in morphisms $G'\rightarrow G$
with $G'$ admitting a universal extension.  

\begin{theorem}\label{UniExtCompat}
	Let $T$ be an arbitrary base scheme.
	\begin{enumerate}
		\item If $A$ is an abelian scheme over $T$, then a universal vectorial
		extension $\E(A)$ of $A$ exists, with $V=\omega_{\Dual{A}}$,
		and is compatible with arbitrary base change on $T$.
		\label{UniExtCompat1}
		
		\item If $p$ is locally nilpotent on $T$ and $G$ is a $p$-divisible group over
		$T$, then a universal vectorial extension $\E(G)$ of $G$ extsis, with $V=\omega_{\Dual{G}}$,
		and is compatible with arbitrary base change on $T$.\label{UniExtCompat2}
		
		\item If $p$ is locally nilpotent on $T$ and $A$ is an abelian scheme over $T$ with
		associated $p$-divisible group $G:=A[p^{\infty}]$, then the canonical map of fppf sheaves
		$G\rightarrow A$ extends to a natural map
		\begin{equation*}
			\xymatrix{
				0 \ar[r] & {\omega_{\Dual{G}}} \ar[r]\ar[d] & {\E(G)} \ar[r]\ar[d] & {G}\ar[d] \ar[r] & 0\\
				0 \ar[r] & {\omega_{\Dual{A}}} \ar[r] & {\E(A)} \ar[r] & {A} \ar[r] & 0
			}
		\end{equation*}
		which induces an isomorphism of the corresponding short exact sequences of Lie algebras.
		\label{UniExtCompat3}
	\end{enumerate} 
\end{theorem}

\begin{proof}
	For the proofs of (\ref{UniExtCompat1}) and (\ref{UniExtCompat2}), see
	 \cite[\Rmnum{1}, \S1.8 and \S1.9]{MM}.  To prove (\ref{UniExtCompat3}), note that
	 pulling back the universal vectorial extension of $A$ along $G\rightarrow A$
	 gives a vectorial extension $\E'$ of $G$ by $\omega_{\Dual{A}}$.  By universality, there then exists
	 a unique map $\psi:\omega_{\Dual{G}}\rightarrow \omega_{\Dual{A}}$ with the property
	 that the pushout of $\E(G)$ along $\psi$ is $\E'$, and this gives the map on universal extensions.  
	 That the induced map on Lie algebras is an isomorphism follows from \cite[\Rmnum{2}, \S 13]{MM}.
\end{proof}

For our applications, we will need a generalization of the universal extension 
of an abelian scheme to the setting of N\'eron models; in order to describe this
generalization, we first recall the explicit description of the universal
extension of an abelian scheme in terms of rigidified extensions.

For any commutative $T$-group scheme $F$, 
a {\em rigidified extension of $F$ by $\Gm$ over $T$} is a pair $(E,\sigma)$
consisting of an extension (of fppf abelian sheaves)
\begin{equation}
	\xymatrix{
		0 \ar[r] & {\Gm} \ar[r] & {E} \ar[r] & {F} \ar[r] & 0
		}\label{ExtRigDef}
\end{equation}
and a splitting $\sigma: \Inf^1(F)\rightarrow E$ of the pullback of (\ref{ExtRigDef})
along the canonical closed immersion $\Inf^1(F)\rightarrow F$.  Two rigidified 
extensions $(E,\sigma)$ and $(E',\sigma')$ are equivalent if there 
is a group homomorphism $E\rightarrow E'$ carrying $\sigma$ to $\sigma'$
and inducing the identity on $\Gm$ and on $F$.
The set $\Extrig_T(F,\Gm)$ of equivalence classes of rigidified extensions over $T$ is naturally a group
via Baer sum of rigidified extensions\cite[\Rmnum{1}, \S2.1]{MM}, so the functor on $T$-schemes
$T'\rightsquigarrow \Extrig_{T'}(F_{T'},\Gm)$ is naturally a group functor that is
contravariant in $F$ via pullback (fibered product).
We write $\scrExtrig_T(F,\Gm)$ for the fppf sheaf of abelian groups associated to this functor.

\begin{proposition}[Mazur-Messing]\label{MMrep}
	Let $A$ be an abelian scheme over an arbitrary base scheme $T$. 
	The fppf sheaf $\scrExtrig_T(A,\Gm)$ is represented by a smooth and separated $T$-group scheme, 
	and there is a canonical short exact sequence of smooth group schemes over $T$
	\begin{equation}
		\xymatrix{
			0\ar[r] & {\omega_A} \ar[r] & {\scrExtrig_T(A,\Gm)} \ar[r] & {\Dual{A}} \ar[r] & 0
		}.\label{univextabelian}
	\end{equation}
	Furthermore, $(\ref{univextabelian})$ is naturally isomorphic to the universal extension of $\Dual{A}$ by a vector
	group.
\end{proposition}

\begin{proof}
	See \cite{MM}, $\Rmnum{1}, \S2.6$ and Proposition 2.6.7.
\end{proof}

In the case that $T=\Spec R$ for $R$ a discrete valuation ring of mixed
characteristic $(0,p)$ with fraction field $K$,
we have the following genaralization of Proposition \ref{MMrep}:

\begin{proposition}	
	Let $A$ be an abelian variety over $K$, with dual abelian variety $\Dual{A}$, and
	write $\A$ and $\Dual{\A}$ for the N\'eron models of $A$ and $\Dual{A}$ over $T=\Spec(R)$.
	Then the fppf abelian sheaf $\scrExtrig_T(\A,\Gm)$ on the category of smooth $T$-schemes
	is represented by a smooth and separated $T$-group scheme.  Moreover, there
	is a canonical short exact sequence of smooth group schemes over $T$
	\begin{equation}
		\xymatrix{
			0\ar[r] & {\omega_{\A}} \ar[r] & {\scrExtrig_T(\A,\Gm)} \ar[r] & {\Dual{\A}^0} \ar[r] & 0
		}\label{NeronCanExt}
	\end{equation}
	which is contravariantly functorial in $A$ via homomorphisms of abelian varieties over $K$.
	The formation of $(\ref{NeronCanExt})$ is compatible with smooth base change on $T$; in particular,
	the generic fiber of $(\ref{NeronCanExt})$ is the universal extension of $\Dual{A}$ by a vector group.
\end{proposition}

\begin{proof}
	Since $R$ is of mixed characteristic $(0,p)$ with perfect residue field,
    this follows from Proposition 2.6 and the discussion following Remark 2.9 in \cite{CaisNeron}.
\end{proof}

In the particular case that $A$ is the Jacobian of a smooth, proper and geometrically
connected curve $X$ over $K$ which is the generic fiber of a normal proper curve $\X$
over $R$, we can relate the exact sequence of Lie algebras attached to (\ref{NeronCanExt})
to the exact sequence $H(X/R)$ of Proposition \ref{HodgeIntEx}:

\begin{proposition}	\label{intcompare}
	Let $\X$ be a proper relative curve over $T=\Spec(R)$ with smooth generic fiber $X$ over $K$.
	Write $J:=\Pic^0_{X/K}$ for the Jacobian of $X$ and $\Dual{J}$ for its dual, 
	and let $\J$, $\Dual{\J}$ be the corresponding N\'eron models over $R$.
		There is a canonical homomorphism of exact sequences of finite free $R$-modules
			\begin{equation}
			\begin{gathered}
				\xymatrix{
					0 \ar[r] & {\Lie\omega_{\J}} \ar[r]\ar[d] & {\Lie\scrExtrig_T(\J,\Gm)} \ar[r]\ar[d]
					& {\Lie \Dual{\J}^0}  \ar[r]\ar[d] & 0\\
					0 \ar[r] & {H^0(\X,\omega_{\X/T})} \ar[r] & {H^1(\X/R)} \ar[r] & {H^1(\X,\O_{\X})} 
					\ar[r] & 0
			}
			\end{gathered}\label{IntegralComparisonMap}
			\end{equation} 
		that is an isomorphism when $\X$ has rational singularities.\footnote{Recall that $\X$ is 
		said to have {\em rational singularities} if it admits a resolution of singularities 
		$\rho:\X'\rightarrow \X$ with the natural map $R^1\rho_*\O_{{\X'}}=0$.  Trivially, any 
		regular $\X$ has rational singularities.}	
		For any finite morphism $\rho:\Y \rightarrow \X$ of $S$-curves satisfying the above hypotheses,
		the map $(\ref{IntegralComparisonMap})$ intertwines $\rho_*$ 
		$($respectively $\rho^*$$)$ on the bottom row with $\Pic(\rho)^*$ 
		$($respectively $\Alb(\rho)^*$$)$ on the top.
\end{proposition}	
	
\begin{proof}
	 	See Theorem 1.2 and (the proof of) Corollary 5.6 in \cite{CaisNeron}.
\end{proof}

\begin{remark}\label{canonicalproperty}
	Let $X$ be a smooth and geometrically connected curve over $K$
	admitting a normal proper model $\X$ over $R$ that is a curve
	having rational singularities.
	It follows from Proposition \ref{intcompare}
	and the N\'eron mapping property
	that $H(\X/R)$ is a {\em canonical integral structure}
	on the Hodge filtration of $H^1_{\dR}(X/K)$: it is
	independent of the choice of proper model $\X$ that is normal with rational singularities, 
	and
	is contravariantly (respectively covariantly) functorial by pullback
	(respectively trace) in finite morphisms $\rho:Y\rightarrow X$
	of proper smooth curves over $K$ which admit models over $R$ satisfying these hypotheses.
	These facts can be proved in greater generality by appealing to resolution of singularities
	for excellent surfaces and the flattening techniques of Raynaud--Gruson \cite{RayGrus};
	see \cite[Theorem 5.11]{CaisDualizing} for details.
\end{remark}

Finally, we will need to relate the universal extension of a $p$-divisible group
as in Theorem \ref{UniExtCompat} (\ref{UniExtCompat2}) to 
its Dieudonn\'e crystal.
In order to explain how this goes, we begin by recalling some basic facts
from crystalline Dieudonn\'e theory, as discussed in \cite{BBM}.

Fix a perfect field $k$ and set $\Sigma:=\Spec(W(k))$, considered as a PD-scheme via the canonical divided powers on the ideal $pW(k)$. Let $T$ be 
a $\Sigma$-scheme on which $p$ is locally nilpotent (so $T$ is naturally a PD-scheme over $\Sigma$), and 
denote by $\Cris(T/\Sigma)$ the big crystalline site of $T$ over $\Sigma$,
endowed with the {\em fppf} topology (see \cite[\S 2.2]{BBM1}). 
If $\scrF$ is a sheaf on $\Cris(T/\Sigma)$ and $T'$ is any PD-thickening of $T$, 
we write $\scrF_{T'}$ for the associated {\em fppf} sheaf on $T'$.
As usual, we denote by $i_{T/\Sigma}:T_{fppf}\rightarrow (T/\Sigma)_{\Cris}$
the canonical morphism of topoi, and 
we abbreviate $\underline{G}:={i_{T/\Sigma}}_{*}G$ for any fppf sheaf $G$ on $T$.

Let $G$ be a $p$-divisible group over $T$, considered as an fppf abelian sheaf on $T$.
As in \cite{BBM}, we define the (contravariant) {\em Dieudonn\'e crystal of $G$ over $T$} to be
\begin{equation}
	\D(G) := \scrExt^1_{T/\Sigma}(\underline{G},\O_{T/\Sigma}).\label{DieudonneDef}
\end{equation}
It is a locally free crystal in $\O_{T/\Sigma}$-modules, which is contravariantly functorial
in $G$ and of formation compatible with base change along PD-morphisms $T'\rightarrow T$ of $\Sigma$-schemes
thanks to 2.3.6.2 and Proposition 2.4.5 $(\rmnum{2})$ of \cite{BBM}.
If $T'=\Spec(A)$ is affine, we will simply write $\D(G)_A$ for the finite locally free $A$-module 
associated to $\D(G)_{T'}$.

The structure sheaf $\O_{T/\Sigma}$ is canonically an extension of $\u{\mathbf{G}}_a$
by the PD-ideal $\J_{T/\Sigma}\subseteq \O_{T/\Sigma}$, and by applying $\scrHom_{T/\Sigma}(\underline{G},\cdot)$
to this extension one obtains (see Propositions 3.3.2 and 3.3.4 as well as
Corollaire 3.3.5 of \cite{BBM})
a short exact sequence (the {\em Hodge filtration})
\begin{equation}
	\xymatrix{
		0\ar[r] & {\scrExt^1_{T/\Sigma}(\underline{G},\J_{T/\Sigma})}\ar[r] &		
		{\D(G)}\ar[r] &		
		{\scrExt^1_{T/\Sigma}(\underline{G},\u{\mathbf{G}}_a)}\ar[r] & 0		
	}\label{HodgeFilCrys}
\end{equation}
that is contravariantly functorial
in $G$ and of formation compatible with base change along PD-morphisms
$T'\rightarrow T$ of $\Sigma$-schemes.
The following ``geometric" description of the value of (\ref{HodgeFilCrys}) on a PD-thickening of the
base will be essential for our purposes:

\begin{proposition}\label{BTgroupUnivExt}
	Let $G$ be a fixed $p$-divisible group over $T$ and let $T'$ be any 
	$\Sigma$-PD thickening of $T$. If $G'$ is any lifting of $G$ to a $p$-divisible
	group on $T'$, then there is a natural isomorphism 
	\begin{equation*}
		\xymatrix{
			0 \ar[r] & {\omega_{G'}} \ar[r]\ar[d]^-{\simeq} & {\scrLie(\E(\Dual{G'}))} \ar[r]\ar[d]^-{\simeq} & 
			{\scrLie (\Dual{G'})}\ar[r]\ar[d]^-{\simeq} & 0\\		
			0\ar[r] & {\scrExt^1_{T/\Sigma}(\underline{G},\J_{T/\Sigma})_{T'}}\ar[r] & {\D(G)_{T'}}\ar[r] &		
			{\scrExt^1_{T/\Sigma}(\underline{G},\underline{\mathbf{G}}_a)_{T'}}\ar[r] & 0		
		}
	\end{equation*} 
	that is moreover compatible with base change in the evident manner.
\end{proposition}

\begin{proof}
	See \cite[Corollaire 3.3.5]{BBM} and \cite[\Rmnum{2}, Corollary 7.13]{MM}.
\end{proof}

\begin{remark}\label{MessingRem}
	In his thesis \cite{Messing}, Messing showed that the Lie algebra of the universal extension
	of $\Dual{G}$ is ``crystalline in nature" and used this as the {\em definition}\footnote{Noting 
	that it suffices to define the crystal $\D(G)$ on $\Sigma$-PD thickenings $T'$ 
	of $T$ to which $G$ admits a lift.} of $\D(G)$.
(See chapter $\Rmnum{4}$, \S2.5 of \cite{Messing} and especially 2.5.2).  Although we
	prefer the more intrinsic description (\ref{DieudonneDef}) of 
	\cite{MM} and \cite{BBM}, it is ultimately Messing's original
	definition that will be important for us.  
\end{remark}

\subsection{Dieudonn\'e crystals and \texorpdfstring{$(\varphi,\Gamma)$}{(phi,Gamma)}-modules}\label{PhiGammaCrystals}

In this section, we summarize the main results of \cite{CaisLau},
which provides a classification of $p$-divisible groups 
by certain semi-linear algebra structures.
These structures---which arise naturally via the Dieudonn\'e crystal functor---
are cyclotomic analogues of Breuil and Kisin modules, and are closely
related to Wach modules.\footnote{See \cite{CaisLau} for the precise relationship.}


Fix a perfect field $k$ of characteristic $p$.
Write $W:=W(k)$ for the Witt vectors of $k$ and $K$ for its fraction field,
and denote by $\varphi$ the unique automorphism of $W(k)$ lifting the $p$-power map
on $k$.  Fix an algebraic closure $\overline{K}$ of $K$, as well
as a compatible sequence $\{\varepsilon^{(r)}\}_{r\ge 1}$ of primitive $p$-power roots of unity in $\o{K}$,
and set $\scrG_K:=\Gal(\o{K}/K)$.
For $r\ge 0$, we put $K_r:=K(\mu_{p^r})$ and $R_r:=W[\mu_{p^r}]$,
and we set $\Gamma_r:=\Gal(K_{\infty}/K_r)$, and $\Gamma:=\Gamma_0$.

Let $\s_r:=W[\![u_r]\!]$ be the power series ring in one variable $u_r$ over $W$,
viewed as a topological ring via the $(p,u_r)$-adic topology. 
We equip $\s_r$ with the unique continuous action of $\Gamma$ and extension of $\varphi$ 
determined by
\begin{align}
	&\gamma u_r := (1+u_r)^{\chi(\gamma)} -1\quad \text{for $\gamma\in \Gamma$} && \text{and} &&
	\varphi(u_r) := (1+u_r)^p -1.\label{gamphiact}
\end{align}
We denote by $\O_{\E_r}:=\widehat{\s_r[\frac{1}{u_r}]}$ the $p$-adic completion of the localization
${\s_r}_{(p)}$, which is a complete discrete valuation ring with uniformizer $p$ and 
residue field $k(\!(u_r)\!)$.  One checks that the actions of $\varphi$ and $\Gamma$
on $\s_r$ uniquely extend to $\O_{\E_r}$.

For $r>0$, we write $\theta: \s_r\twoheadrightarrow R_r$ for the continuous and $\Gamma$-equivariant $W$-algebra 
surjection
sending $u_r$ to $\varepsilon^{(r)}-1$, whose kernel is the principal ideal generated by
the Eisenstein polynomial $E_r:=\varphi^r(u_r)/\varphi^{r-1}(u_r)$,
and we denote by $\tau:\s_r\twoheadrightarrow W$ the continuous and $\varphi$-equivariant surjection of $W$-algebras 
determined by $\tau(u_r)=0$.
We lift the canonical inclusion $R_r\hookrightarrow R_{r+1}$
to a $\Gamma$- and $\varphi$-equivariant $W$-algebra injection
${\s_r} \hookrightarrow {\s_{r+1}}$
determined by $u_r\mapsto \varphi(u_{r+1})$;
this map uniquely extends to a continuous injection
$\O_{\E_r}\hookrightarrow \O_{\E_{r+1}}$, compatibly with $\varphi$ and $\Gamma$.
We will frequently identify $\s_r$ (respectively $\O_{\E_r}$) with its image in $\s_{r+1}$ 
(respectively $\O_{\E_{r+1}}$),
which coincides with the image of $\varphi$ on $\s_{r+1}$ (respectively $\O_{\E_{r+1}})$.
Under this convention, we have 
$E_{r}(u_r) = E_1(u_1) = u_0/u_1$ for all $r>0$, so we will simply write $\omega:=E_r(u_r)$
for this common element of $\s_r$ for $r>0$.

\begin{definition}	
	We write $\BT_{\s_r}^{\varphi}$ for the category of {\em Barsotti-Tate modules over $\s_r$},
	{\em i.e.} the category whose objects are pairs $(\m,\varphi_{\m})$ where
	\begin{itemize}
		\item $\m$ is a free $\s_r$-module of finite rank.
		\item $\varphi_{\m}:\m\rightarrow \m$ is a $\varphi$-semilinear
		map whose linearization has cokernel killed by $\omega$,
	\end{itemize}
	and whose morphisms are $\varphi$-equivariant $\s_r$-module homomorphisms. 
	We write $\BT_{\s_r}^{\varphi,\Gamma}$ for the subcategory of $\BT_{\s_r}^{\varphi}$
	consisting of objects $(\m,\varphi_{\m})$ which admit a semilinear $\Gamma$-action 
	(in the category $\BT_{\s_r}^{\varphi}$) with the property that $\Gamma_r$ acts trivially on
	$\m/u_r\m$.  Morphisms in $\BT_{\s_r}^{\varphi,\Gamma}$ are $\varphi$ and $\Gamma$-equivariant
	morphisms of $\s_r$-modules.
	We often abuse notation by writing $\m$ for the pair 
	$(\m,\varphi_{\m})$ and $\varphi$ for $\varphi_{\m}$. 
\end{definition}

If $(\m,\varphi_{\m})$ is any object of $\BT_{\s_r}^{\varphi,\Gamma}$, then 
$1\otimes\varphi_{\m}:\varphi^*\m\rightarrow \m$
is injective with cokernel killed by $\omega$, so there is a unique
$\s_r$-linear homomorphism $\psi_{\m}:\m\rightarrow \varphi^*\m$
with the property that the composition of $1\otimes\varphi_{\m}$ and $\psi_{\m}$ (in either order)
is multiplication by $\omega$.  Clearly, $\varphi_{\m}$ and $\psi_{\m}$ determine eachother.
We warn the reader that 
the action of $\Gamma$ 
does {\em not}
commute with $\psi_{\m}$: instead, for any $\gamma\in \Gamma$, one has 
\begin{equation}
	(\gamma\otimes\gamma)\circ\psi_{\m}=(\gamma\omega/\omega)\cdot\psi_{\m}\circ\gamma.
	\label{psiGammarel}
\end{equation}

\begin{definition}\label{DualBTDef}
	Let $\m$ be an object of $\BT_{\s_r}^{\varphi,\Gamma}$.  The {\em dual
	of $\m$} is the object $(\m^{t},\varphi_{\m^{t}})$ of $\BT_{\s_r}^{\varphi,\Gamma}$
	whose underlying $\s_r$-module is $\m^{t}:=\Hom_{\s_r}(\m,\s_r)$, equipped with
	the $\varphi$-semilinear endomorphism
	\begin{equation*}
		\xymatrix@C=32pt{
			{\varphi_{\m^{t}}: \m^{t}} \ar[r]^-{1\otimes \id_{\m^{t}}} & {\varphi^*\m^{t} \simeq (\varphi^*\m)^{t}}
			\ar[r]^-{\psi_{\m}^{t}} & {\m^{t}}
		}
	\end{equation*}
	and the commuting\footnote{As one checks
	using the intertwining relation (\ref{psiGammarel}).} action of $\Gamma$ given for $\gamma\in \Gamma$ by
	\begin{equation*}
		(\gamma f)(m) := \chi(\gamma)^{-1}\varphi^{r-1}(\gamma u_r/u_r)\cdot\gamma (f(\gamma^{-1} m )).
	\end{equation*}
\end{definition}

There is a natural notion of base change for Barsotti--Tate modules.
Let $k'/k$ be an algebraic extension (so $k'$ is automatically perfect), and write $W':=W(k')$,
$R_r':=W'[\mu_{p^r}]$, $\s_r':=W'[\![u_r]\!]$, and so on.  
The canonical
inclusion $W\hookrightarrow W'$ extends to a $\varphi$ and $\Gamma$-compatible
$W$-algebra injection $\iota_r:\s_r\hookrightarrow \s_{r+1}'$, and extension
of scalars along $\iota_r$ yields a canonical 
canonical base change functor ${\iota_r}_*: \BT_{\s_r}^{\varphi,\Gamma}\rightarrow \BT_{\s_{r+1}}^{\varphi,\Gamma}$
which one checks is compatible with duality.

Let us write $\pdiv_{R_r}^{\Gamma}$ for the subcategory of $p$-divisible groups over $R_r$
consisting of those objects and morphisms which descend (necessarily uniquely) to $K=K_0$ on generic fibers.
By Tate's Theorem, this is of course equivalent to the full subcategory of $p$-divisible
groups over $K_0$ which have good reduction over $K_r$.   Note that for any algebraic extension $k'/k$, 
base change along the inclusion $\iota_r:R_r\hookrightarrow R_{r+1}'$ gives a covariant functor
${\iota_r}_*:\pdiv_{R_r}^{\Gamma}\rightarrow \pdiv_{R_{r+1}'}^{\Gamma}$.

The main result of \cite{CaisLau} is the following:

\begin{theorem}\label{CaisLauMain}
	For each $r>0$, there is a contravariant functor 
	$\m_r:\pdiv_{R_r}^{\Gamma}\rightarrow \BT_{\s_r}^{\varphi,\Gamma}$ such that:
	\begin{enumerate}
		\item The functor $\m_r$ is an exact antiequivalence of categories, compatible with duality.
		\label{exequiv}
		\item The functor $\m_r$ is of formation compatible with base change: 
		for any algebraic extension $k'/k$, there is a natural isomorphism
		of composite functors ${\iota_r}_*\circ \m_r \simeq \m_{r+1}\circ {\iota_{r}}_*$ on $\pdiv_{R_r}^{\Gamma}$.
		\label{BaseChangeIsom}
		\item For $G\in \pdiv_{R_r}^{\Gamma}$, put $\o{G}:=G\times_{R_r} k$ and $G_0:=G\times_{R_r} R_r/pR_r$.
		\begin{enumerate}		
			\item There is a functorial and $\Gamma$-equivariant isomorphism of $W$-modules
			\begin{equation*}
				\m_r(G)\tens_{\s_r,\varphi\circ \tau} W \simeq \D(\o{G})_W,
			\end{equation*}		 
			carrying 
			$\varphi_{\m}\otimes \varphi$ to $F:\D(\o{G})_W\rightarrow \D(\o{G})_W$ 
			and $\psi_{\m}\otimes 1$ to $V\otimes 1: \D(\o{G})_W \rightarrow  \varphi^*\D(\o{G})_W$.
			\label{EvaluationONW}
		\item There is a functorial and $\Gamma$-equivariant isomorphism of $R_r$-modules
			\begin{equation*}
				\m_r(G)\tens_{\s_r,\theta\circ\varphi} R_{r} \simeq \D(G_0)_{R_r}.
			\end{equation*}\label{EvaluationONR}
		\end{enumerate}
	\end{enumerate}
\end{theorem}

We wish to explain how to functorially recover the $\scrG_K$-representation afforded 
by the $p$-adic Tate module $T_pG_K$ from $\m_r(G)$.  In order to do so, we must first recall 
the necessary period rings; for a more detailed synopsis of these rings and their properties,
we refer the reader to \cite[\S6--\S8]{Colmez}.

As usual, we put\footnote{Here we use the notation introduced by Berger and Colmez; in Fontaine's
original notation, this ring is denoted $\R$.} 
$$\wt{\e}^+:=\varprojlim_{x\mapsto x^p} \O_{\c_K}/(p),$$
equipped with its canonical $\scrG_K$-action via ``coordinates"
and $p$-power Frobenius map $\varphi$.
This is a perfect ({\em i.e.} $\varphi$ is an automorphism) valuation ring  
of charteristic $p$ with residue field $\overline{k}$
and fraction field $\wt{\e}:=\Frac(\wt{\e}^+)$ that is algebraically closed.
We view $\wt{\e}$ as a topological field via its valuation topology, with respect to which 
it is complete.
Our fixed choice of $p$-power compatible
sequence $\{\varepsilon^{(r)}\}_{r\ge 0}$ 
induces an element $\u{\varepsilon}:=(\varepsilon^{(r)}\bmod p)_{r\ge 0}$ of $\wt{\e}^+$
and we set $\e_{K}:=k(\!(\u{\varepsilon} - 1)\!)$, viewed as a 
topological\footnote{The valuation $v_{\e}$ on $\wt{\e}$ induces the usual discrete valuation 
on $\e_{K,r}$, with the unusual 
normalization $1/p^{r-1}(p-1)$.} subring of $\wt{\e}$; note that this is 
a $\varphi$- and $\scrG_K$-stable subfield of $\wt{\e}$ that is independent
of our choice of $\u{\varepsilon}$.  We write $\e:=\e_K^{\sep}$
for the separable closure of $\e_K$ in the algebraically closed field $\wt{\e}$.
The natural $\scrG_K$-action on $\wt{\e}$ induces a canonical identification
$\Gal(\e/\e_{K}) = \scrH:=\ker(\chi)\subseteq \scrG_K$, so
$\e^{\scrH}=\e_{K}$.  
If $E$ is any subring of $\wt{\e}$, we write $E^+:=E\cap \wt{\e}^+$
for the intersection (taken inside $\wt{\e}$).

We now construct Cohen rings for each of the above subrings of $\wt{\e}$.
To begin with, we put
\begin{equation*}
 \wt{\a}^+:=W(\wt{\e}^+),\qquad\text{and}\qquad \wt{\a}:=W(\wt{\e});
\end{equation*} 
each of these rings is equipped with a canonical Frobenius automorphism $\varphi$
and action of $\scrG_K$ via Witt functoriality.  
Set-theoretically identifying $W(\wt{\e})$ with $\prod_{m=0}^{\infty} \wt{\e}$ in the usual way, we 
endow each factor with its valuation topology and give $\wt{\a}$ the product topology.\footnote{This 
is what is called the {\em weak topology} on $\wt{\a}$.
If each factor of $\wt{\e}$ is instead given the discrete topology, then the
product topology on $\wt{\a}=W(\wt{\e})$ is the familiar $p$-adic 
topology, called the {\em strong} topology.}
The $\scrG_K$ action on $\wt{\a}$ is then continuous
and the canonical $\scrG_K$-equivariant $W$-algebra surjection 
$\theta:\wt{\a}^+\rightarrow \O_{\c_K}$ is continuous when
$\O_{\c_K}$ is given its usual $p$-adic topology.
For each $r\ge 0$, there is a unique continuous $W$-algebra map $j_r:\O_{\E_r}\hookrightarrow \wt{\a}$ 
determined by $j_r(u_r):=\varphi^{-r}([\u{\varepsilon}] - 1)$.  These maps are 
moreover $\varphi$ and $\scrG_K$-equivariant, with $\scrG_K$ acting on $\O_{\E_r}$ through the quotient
$\scrG_K\twoheadrightarrow \Gamma$, and compatible with change in $r$.
We define $\a_{K,r}:=\im(j_r:\O_{\E_r}\rightarrow \wt{\a}),$
which is
naturally a $\varphi$ and $\scrG_K$-stable subring of $\wt{\a}$ that is independent of our choice
of $\u{\varepsilon}$.  
We again omit the subscript when $r=0$. 
Note that $\a_{K,r}=\varphi^{-r}(\a_K)$ inside $\wt{\a}$, and that $\a_{K,r}$
is a discrete valuation ring with uniformizer $p$ and residue field $\varphi^{-r}(\e_K)$
that is purely inseparable over $\e_K$.
We define $\a_{K,\infty}:=\bigcup_{r\ge 0} \a_{K,r}$
and write $\wt{\a}_K$ (respectively $\wh{\a}_K$) 
for the closure of $\a_{K,\infty}$ in $\wt{\a}$ with respect to the weak
(respectively strong) topology.  

Let $\a_{K,r}^{\sh}$ be the strict Henselization of $\a_{K,r}$
with respect to the separable closure of its residue field inside $\wt{\e}$.
Since $\wt{\a}$ is strictly Henselian, there is a unique local morphism 
$\a_{K,r}^{\sh}\rightarrow \wt{\a}$ recovering the given inclusion on residue fields,
and we henceforth view $\a_{K,r}^{\sh}$ as a subring of $\wt{\a}$.
We denote by $\a_r$ the topological closure
of $\a_{K,r}^{\sh}$ inside $\wt{\a}$ with respect to the strong topology,
which
is a $\varphi$ and $\scrG_K$-stable subring of $\wt{\a}$,
and we note that $\a_r = \varphi^{-r}(\a)$ and $\a_r^{\scrH}= \a_{K,r}$
inside $\wt{\a}$. 
We note also that the canonical map $\Z_p\hookrightarrow \wt{\a}^{\varphi=1}$
is an isomorphism, from which it immediately follows that the same is true
if we replace $\wt{\a}$ by any of its subrings constructed above.
If $A$ is any subring of $\wt{\a}$, we define $A^+:=A\cap \wt{\a}^+$,
with the intersection taken inside $\wt{\a}$.

\begin{remark}\label{Slimits}
	We will identify $\s_r$ and $\O_{\E_r}$ with their respective images
	$\a_{K,r}^+$ and $\a_{K,r}$ in $\wt{\a}$ under $j_r$.
	Writing $\s_{\infty}:=\varinjlim \s_r$
	and $\O_{\E_{\infty}}:=\varinjlim \s_r$, we likewise 
	identify $\s_{\infty}$ with $\a_{K,\infty}^+$ and $\O_{\E_{\infty}}$
	with $\a_{K,\infty}$.  
	Denoting by $\wh{\s}_{\infty}$ (respectively $\wt{\s}_{\infty}$) the $p$-adic (respectively 
	$(p,u_0)$-adic) completions, one has
	\begin{equation*}
		\wh{\s}_{\infty} = \wh{\a}_K^+ = W(\e_K^{\rad,+})\quad\text{and}\quad
		\wt{\s}_{\infty} = \wt{\a}_K^+ = W(\wt{\e}_K^{+}),
	\end{equation*}
	for $\e_K^{\rad}:=\cup_{r\ge 0} \varphi^{-r}(\e_K)$ the radiciel ($=$perfect) closure of 
	$\e_K$ in $\wt{\e}$ and $\wt{\e}_K$ its topological completion.
	Via these identifications, $\omega :=u_0/u_1\in \a_{K,1}^+$ is
	a principal generator
	of $\ker(\theta:\wt{\a}^+\twoheadrightarrow \O_{\C_K})$.
\end{remark}

We can now explain the functorial relation between $\m_r(G)$ and $T_pG_K$:

\begin{theorem}\label{comparison}
	Let $G\in \pdiv_{R_r}^{\Gamma}$, and write $H^1_{\et}(G_K):=(T_pG_K)^{\vee}$
	for the $\Z_p$-linear dual of $T_pG_K$.
	There is a canonical mapping
	of finite free $\a_r^+$-modules with semilinear Frobenius and $\scrG_K$-actions
	\begin{equation}
		\xymatrix{
			{\m_r(G)\tens_{\s_r,\varphi} \a_r^+} \ar[r] & {H^1_{\et}(G_K)\otimes_{\Z_p} \a_r^+}
		}
	\end{equation}
	that is injective with cokernel killed by $u_1$.  
	Here, $\varphi$ acts as $\varphi_{\m_r(G)}\otimes \varphi$ on source
	and as $1\otimes\varphi$ on target, while $\scrG_K$ acts diagonally 
	on source and target through the quotient $\scrG_K\twoheadrightarrow \Gamma$
	on $\m_r(G)$.
	  In particular, there is a natural $\varphi$ and $\scrG_K$-equivariant
	isomorphism
	\begin{equation}
				{\m_r(G)\tens_{\s_r,\varphi} \a_r} \simeq  {H^1_{\et}(G_K)\otimes_{\Z_p} \a_r}.
				\label{comparisonb}
	\end{equation}
	These mappings are compatible with duality and with change in $r$ in the obvious manner.
\end{theorem}

\begin{corollary}\label{GaloisComparison}
	For $G\in \pdiv_{R_r}^{\Gamma}$, there are functorial isomorphisms of $\Z_p[\scrG_K]$-modules
	\begin{subequations}
		\begin{align}
			T_pG_K &\simeq \Hom_{\s_r,\varphi}(\m_r(G),\a_r^+)\\
			H^1_{\et}(G_K) &\simeq (\m_r(G) \tens_{\s_r,\varphi} \a_r)^{\varphi_{\m_r(G)}\otimes \varphi=1}.
			\label{FontaineModule}
		\end{align}
	\end{subequations}
	which are compatible with duality and change in $r$.  In the first isomorphism,
	we view $\a_r^+$ as a $\s_r$-algebra via the composite of the usual structure map with 
	$\varphi$. 
\end{corollary}

\begin{remark}
	By definition, the map $\varphi^r$ on $\O_{\E_r}$ is injective with image $\O_{\E}:=\O_{\E_0}$,
	and so induces a $\varphi$-semilinear isomorphism of $W$-algebras 
	$\xymatrix@C=15pt{{\varphi^{r}:\O_{\E_r}} \ar[r]^-{\simeq}&{\O_{\E}} }$.
	It follows from (\ref{FontaineModule}) of Corollary \ref{GaloisComparison} and Fontaine's theory of 
	$(\varphi,\Gamma)$-modules
	over $\O_{\E}$ that $\m_r(G)\otimes_{\s_r,\varphi^r} \O_{\E}$ {\em is} the $(\varphi,\Gamma)$-module
	functorially associated to the $\Z_p[\scrG_K]$-module $H^1_{\et}(G_K)$.
\end{remark}

For the remainder of this section, we recall the construction of the functor $\m_r$, both because we shall need
to reference it in what follows, and because we feel it is enlightening.  
For details, including the proofs of Theorems \ref{CaisLauMain}--\ref{comparison}
and Corollary \ref{GaloisComparison}, we refer the reader to \cite{CaisLau}.

Fix $G\in \pdiv_{R_r}^{\Gamma}$ and set $G_0:=G\times_{R_r}{R_r/pR_r}.$  
The $\s_r$-module $\m_r(G)$ is a functorial descent of the evaluation of
the Dieudonn\'e crystal $\D(G_0)$ on a certain ``universal" PD-thickening of $R_r/pR_r$, 
which we now describe.
Let $S_r$ be the $p$-adic completion of the PD-envelope of $\s_r$ with respect to the 
ideal $\ker\theta$, viewed as a (separated and complete) topological ring via the $p$-adic topology.
We give $S_r$ its PD-filtration: for $q\in \Z$ the ideal $\Fil^q S_r$ is the 
topological closure of the ideal generated by the divided powers
$\{\alpha^{[n]}\}$ for $\alpha\in \ker\theta$ and $n\ge q$.
By construction, the map $\theta:\s_r\twoheadrightarrow R_r$
uniquely extends to a continuous surjection of $\s_r$-algebras $S_r\twoheadrightarrow R_r$ 
(which we again denote by $\theta$) whose kernel $\Fil^1 S_r$ is equipped with topologically 
PD-nilpotent\footnote{Here we use our assumption that $p>2$.} divided powers.
Similarly, the continuous $W$-algebra map $\tau:\s_r\twoheadrightarrow W$ determined by
$\tau(u_r)=0$ uniquely extends to a continuous, PD-compatible $W$-algebra surjection $\tau:S_r\twoheadrightarrow W$
whose kernel we denote by $I:=\ker(\tau)$.
One shows that there is a unique continuous endomorphism $\varphi$ of $S_r$ extending $\varphi$ on $\s_r$,
and that $\varphi(\Fil^1 S_r)\subseteq pS_r$; in particular, we may define
$\varphi_1: \Fil^1 S_r\rightarrow S_r$ by $\varphi_1:=\varphi/p$,
which is a $\varphi$-semilinear homomorphism of $S_r$-modules. Note that
$v_r:=\varphi_1(E_r)$ is a unit of $S_r$, so the image of $\varphi_1$ generates
$S_r$ as an $S_r$-module.

Since the action of $\Gamma$ on $\s_r$ preserves $\ker\theta$, it follows from the universal mapping property of
divided power envelopes and $p$-adic continuity considerations that this action uniquely extends to
a continuous and $\varphi$-equivariant action of $\Gamma$ on $S_r$ which
is compatible with the PD-structure and the filtration.
Similarly, the transition map $\s_r\hookrightarrow \s_{r+1}$ uniquely extends to a 
continuous $\s_r$-algebra homomorphism $S_r\rightarrow S_{r+1}$ which is moreover compatible with filtrations
(because $E_r(u_r)=E_{r+1}(u_{r+1})$ under our identifications),
and for nonnegative integers $s < r$ we view $S_r$ as an $S_s$-algebra  
via these maps.  

Put $\lambda := \log(1+u_0)/{u_0},$
where $\log(1+X):\Fil^1 S_r\rightarrow S_r$ is the usual (convergent for the $p$-adic topology) power series
and $u_0:=\varphi^r(u_r)\in S_r$.
One checks that $\lambda$ admits the convergent product expansion 
$\lambda=\prod_{i\ge 0} \varphi^i(v_r)$, so $\lambda\in S_r^{\times}$ and
	\begin{equation}
		\frac{\lambda}{\varphi(\lambda)}  = \varphi(E_r)/p= v_r\qquad\text{and}\qquad
		\frac{\lambda}{\gamma\lambda} = \chi(\gamma)^{-1}\varphi^r(\gamma u_r/u_r) \quad\text{for}\ 
		\gamma\in 		\Gamma.\label{lambdaTransformation}
	\end{equation}

\begin{definition}
	Let $\BT_{S_r}^{\varphi}$ be the category of triples $(\scrM,\Fil^1\scrM, \varphi_{\scrM,1})$ where
	\begin{itemize}
		\item $\scrM$ is a finite free $S_r$-module and $\Fil^1\scrM\subseteq \scrM$ is an $S_r$-submodule.
		\item $\Fil^1\scrM$ contains $(\Fil^1 S_r)\scrM$ and the quotient $\scrM/\Fil^1\scrM$ is a free 
		$S_r/\Fil^1S_r=R_r$-module.
		\item $\varphi_{\scrM,1}:\Fil^1\scrM_r\rightarrow \scrM$ is a $\varphi$-semilinear map whose image
		generates $\scrM$ as an $S_r$-module.
	\end{itemize}
	Morphisms in $\BT_{S_r}^{\varphi}$ are $S_r$-module homomorphisms which are compatible with the
	extra structures.  As per our convention, we will often write $\scrM$ for a triple 
	$(\scrM,\Fil^1\scrM,\varphi_{\scrM,1})$, and $\varphi_1$ for $\varphi_{\scrM,1}$ when it can cause
	no confusion.  We denote by $\BT_{S_r}^{\varphi,\Gamma}$ the subcategory of $\BT_{S_r}^{\varphi}$
	consisting of objects $\scrM$ that are equipped
	with a semilinear action of $\Gamma$ which preserves $\Fil^1\scrM$, commutes with $\varphi_{\scrM,1}$,
	and whose restriction to $\Gamma_r$ is trivial on $\scrM/I\scrM$; morphisms in $\BT_{S_r}^{\varphi,\Gamma}$
	are $\Gamma$-equivariant morphisms in $\BT_{S_r}^{\varphi}$.
\end{definition}

The kernel of the surjection $S_r/p^nS_r\twoheadrightarrow R_r/pR_r$ is the image of the ideal 
$\Fil^1 S_r + pS_r$, which by construction is equipped topologically PD-nilpotent divided powers.
We may therefore define
\begin{equation}
	\scrM_r(G)=\D(G_0)_{S_r}:=\varprojlim_n \D(G_0)_{S_r/p^nS_r},
\end{equation}
which is a finite free $S_r$-module that depends contravariantly functorially on $G_0$.
We promote $\scrM_r(G)$ to an object of $\BT_{S_r}^{\varphi,\Gamma}$ as follows.
As the quotient map $S_r\twoheadrightarrow R_r$ induces a PD-morphism of PD-theckenings
of $R_r/pR_r$, there is a natural isomorphism of free $R_r$-modules
\begin{equation}
	\scrM_r(G)\otimes_{S_r} R_r \simeq \D(G_0)_{R_r}.\label{surjR}
\end{equation}
By Proposition \ref{BTgroupUnivExt}, there is a canonical ``Hodge" filtration $\omega_G \subseteq \D(G_0)_{R_r}$,
which reflects the fact that $G$ is a $p$-divisible group over $R_r$ lifting $G_0$,
and we define $\Fil^1\scrM_r(G)$ to be the preimage of $\omega_G$ under the 
composite of the isomorphism (\ref{surjR}) with the natural surjection 
$\scrM_r(G)\twoheadrightarrow \scrM_r(G)\otimes_{S_r} R_r$; note that this depends on $G$ and not just on 
$G_0$.  The Dieudonn\'e crystal is compatible with base change (see, {\em e.g.} \cite[2.4]{BBM1}), so the 
relative Frobenius $F_{G_0}:G_0\rightarrow G_0^{(p)}$ induces an canonical morphism of $S_r$-modules
\begin{equation*}
	\xymatrix{
		{\varphi^*(\D(G_0)_{S_r}) \simeq \D(G_0^{(p)})_{S_r}} \ar[r]^-{\D(F_{G_0})} & {\D(G_0)_{S_r}}
	},
\end{equation*}
which we may view as a $\varphi$-semilinear map $\varphi_{\scrM_r(G)}:\scrM_r(G)\rightarrow \scrM_r(G)$.
As the relative Frobenius map $\omega_{G_0^{(p)}}\rightarrow \omega_{G_0}$ is zero, 
it follows that the restriction of $\varphi_{\scrM_r(G)}$ to $\Fil^1 \scrM_r(G)$ has image contained in
$p\scrM_r(G)$, so we may define $\varphi_{\scrM_r(G),1}:=\varphi_{\scrM_r(G)}/p$, and one proves as in 
\cite[Lemma A.2]{KisinFCrystal}
that the image of $\varphi_{\scrM_r(G),1}$ generates $\scrM_r(G)$ as an $S_r$-module.

It remains to equip $\scrM_r(G)$ with a canonical semilinear action of $\Gamma$.
Let us write $G_{K_r}$ for the generic fiber of $G$ and $G_{K}$ for its unique
descent to $K=K_0$.  The existence of this descent is reflected by the 
existence of a commutative diagram with cartesian square
\begin{equation}
\begin{gathered}
	\xymatrix{
{G_{K}\fiber_K K_r} \ar@/^/[rrd]^-{1\times \gamma} \ar@/_/[ddr] \ar@{.>}[dr]|-{\gamma} 	&         &					\\
	&{\big(G_{K}\fiber_K K_r\big)_{\gamma}} \ar[r]_-{\pr_1} \ar[d]^-{\pr_2}\ar@{} [dr] |{\square} & 
	{G_{K}\fiber_K K_r} \ar[d]\\
	&{\Spec(K_r)} \ar[r]_-{\gamma} &{\Spec(K_r)}
	}
\end{gathered}
\label{GammaAction}
\end{equation}
for each $\gamma\in \Gamma$, compatibly with change in $\gamma$; here, the subscript of $\gamma$ denotes base change
along the map of schemes induced by $\gamma$.
Since $G$ has generic fiber $G_{K_r}=G_K\times_K K_r$, Tate's Theorem ensures that the
dotted arrow above uniquely extends to an isomorphism
of $p$-divisible groups over $R_r$
\begin{equation}
	\xymatrix{
		{G}\ar[r]^-{\gamma} & {G_{\gamma}} 
	},\label{TateExt}
\end{equation}
compatibly with change in $\gamma$.

By assumption, the action of $\Gamma$ on $S_r$ commutes with the divided powers
on $\Fil^1 S_r$ and induces the given action on the quotient $S_r\twoheadrightarrow R_r$;
in other words, $\Gamma$ acts by automorphisms on the object
$(\Spec(R_r/pR_r)\hookrightarrow \Spec(S_r/p^nS_r))$ of $\Cris((R_r/pR_r)/W)$.
Again using the compatibility of $\D(G_0)$ with base change, 
we therefore see that each $\gamma\in \Gamma$ gives an $S_r$-linear map
\begin{equation*}
	\xymatrix{
		{\gamma^*(\D(G_0)_{S_r}) \simeq \D((G_0)_{\gamma})_{S_r}} \ar[r] & {D(G_0)_{S_r}}
	}
\end{equation*}
and hence an $S_r$-semilinear (over $\gamma$) endomorphism $\gamma$ of $\scrM_r(G)$.
One easily checks that the resulting action of $\Gamma$ on $\scrM_r(G)$
commutes with $\varphi_{\scrM,1}$ and preserves $\Fil^1\scrM_r(G)$.
By the compatibility of $\D(G_0)$ with base change and the obvious fact that
the $W$-algebra surjection $\tau:S_r\twoheadrightarrow W$
is a PD-morphism over the canonical surjection $R_r/pR_r\twoheadrightarrow k$,
there is a natural isomorphism
\begin{equation}
	\scrM_r(G)\otimes_{S_r} W \simeq \D(\o{G})_W. 
\end{equation}
It follows easily from this and the diagram (\ref{GammaAction})
that the action of $\Gamma_r$ on $\scrM_r(G)/I\scrM_r(G)$
is trivial.

To define $\m_r(G)$, we functorially descend the $S_r$-module $\scrM_r(G)$
along the structure morphism $\alpha_r:\s_r\rightarrow S_r$.  More precisely, 
for $\m\in \BT_{\s_r}^{\varphi,\Gamma}$, we define 
${\alpha_r}_*(\m):=(M,\Fil^1M,\Phi_1)\in \BT_{S_r}^{\varphi,\Gamma}$ via:

\begin{equation}
	\begin{gathered}
		M:=\m\tens_{\s_r,\varphi} S_r\qquad\text{with diagonal $\Gamma$-action}\\
		\Fil^1 M :=\left\{ m\in M\ :\ (\varphi_{\m}\otimes\id)(m) \in  \m\otimes_{\s_r} \Fil^1 S_r
		\subseteq \m\otimes_{\s_r} S_r \right\}  \\
		\xymatrix{
			{\Phi_1: \Fil^1 M} \ar[r]^-{\varphi_{\m}\otimes\id} & { \m\tens_{\s_r} \Fil^1 S_r}
			\ar[r]^-{\id\otimes\varphi_1} & {\m\tens_{\s_r,\varphi} S_r = M} 
		}.
\end{gathered}
\label{BreuilSrDef}	
\end{equation}
The following is the key technical point of \cite{CaisLau}, and is proved using 
the theory of windows:
\begin{theorem}\label{Lau}
	For each $r$, the functor ${\alpha_r}_*:\BT_{\s_r}^{\varphi,\Gamma}\rightarrow \BT_{S_r}^{\varphi,\Gamma}$
	is an equivalence of categories, compatible with change in $r$.
\end{theorem}

\begin{definition}
	For $G\in \pdiv_{R_r}^{\Gamma}$, we write $\m_r(G)$
	for the functorial descent of $\scrM_r(G)$ to an object of $\BT_{\s_r}^{\varphi,\Gamma}$
	as guaranteed by Theorem \ref{Lau}.
	By construction, we have a natural isomorphism
of functors ${\alpha_r}_*\circ \m_r\simeq \scrM_r$ on $\pdiv_{R_r}^{\Gamma}$.
\end{definition}

\begin{example}\label{GmQpZpExamples}
Using Messing's description of the Dieudonn\'e crystal of a $p$-divisible group
in terms of the Lie algebra of its universal extension (cf. remark \ref{MessingRem}),
one calculates that for $r\ge 1$
	\begin{subequations}
	\begin{equation}
		\m_r(\Q_p/\Z_p)  = \s_r,\qquad \varphi_{\m_r(\Q_p/\Z_p)}:= \varphi,\qquad \gamma:=\gamma
		\label{MrQpZp}
	\end{equation}
	\begin{equation}
		\m_r(\mu_{p^{\infty}})  = \s_r,\qquad \varphi_{\m_r(\mu_{p^{\infty}})}:= \omega\cdot\varphi,
		\qquad \gamma:=\chi(\gamma)^{-1}\varphi^{r-1}(\gamma u_r/u_r)\cdot \gamma 
	\label{MrMu}
	\end{equation}
\end{subequations}
with $\gamma\in \Gamma$ acting as indicated.  
Note that both $\m_r(\Q_p/\Z_p)$ and $\m_r(\Gm[p^{\infty}])$
arise by base change from their incarnations when $r=1$,
as follows from the fact that $\omega = \varphi(u_1)/u_1$ and
$\varphi^{r-1}(\gamma u_r/u_r)=\gamma u_1/u_1$ via our identifications.
\end{example}

\subsection{The case of ordinary \texorpdfstring{$p$}{p}-divisible groups}\label{pDivOrdSection}

When $G\in \pdiv_{R_r}^{\Gamma}$ is ordinary,
one can say significantly more about the structure of the $\s_r$-module $\m_r(G)$.
To begin with, we observe that for arbitrary $G\in \pdiv_{R_r}^{\Gamma}$,
the formation of the maximal \'etale quotient of $G$ and of the maximal 
connected and multiplicative-type sub $p$-divisible groups of $G$ are functorial in $G$, 
so each of $G^{\et}$, $G^0$, and $G^{\mult}$ is naturally an object of $\pdiv_{R_r}^{\Gamma}$
as well.   We thus (functorially) obtain objects 
$\m_r(G^{\star})$ of $\BT_{\s_r}^{\varphi, \Gamma}$ which admit particularly simple 
descriptions when $\star=\et$ or $\mult$, as we now explain.

As usual, we write $\o{G}^{\star}$ for the special fiber of $G^{\star}$ and $\D(\o{G}^{\star})_W$
for its Dieudonn\'e module.
Twisting the $W$-algebra structure on $\s_r$ by the automorphism $\varphi^{r-1}$
of $W$, 
we define objects of $\BT_{\s_r}^{\varphi,\Gamma}$
\begin{subequations}
	\begin{equation}
		\m_r^{\et}(G) : = \D(\o{G}^{\et})_W\tens_{W,\varphi^{r-1}} \s_r,
		\qquad \varphi_{\m_r^{\et}}:= F\otimes \varphi,
		\qquad \gamma:=\gamma \otimes \gamma 
		\label{MrEtDef}
	\end{equation}
	\begin{equation}
		\m_r^{\mult}(G) : = \D(\o{G}^{\mult})_W\tens_{W,\varphi^{r-1}} \s_r,
		\qquad \varphi_{\m_r^{\mult}}:= V^{-1}\otimes E_r\cdot\varphi,
		\qquad \gamma:=\gamma \otimes \chi(\gamma)^{-1}\varphi^{r-1}(\gamma u_r/u_r)\cdot \gamma 
	\label{MrMultDef}
	\end{equation}
\end{subequations}
with $\gamma\in \Gamma$ acting as indicated.
Note that these formulae make sense and do indeed give objects of $\BT_{\s_r}^{\varphi,\Gamma}$
as $V$ is 
invertible\footnote{A $\varphi^{-1}$-semilinear map of $W$-modules $V:D\rightarrow D$ 
is {\em invertible} if there exists a $\varphi$-semilinear endomorphism $V^{-1}$ whose composition
with $V$ in either order is the identity.  This is easily seen to be equivalent to 
the invertibility of the linear map $V\otimes 1: D\rightarrow \varphi^* D$, with $V^{-1}$
the composite of  $(V\otimes 1)^{-1}$ and the $\varphi$-semilinear map $\id\otimes 1:D\rightarrow \varphi^*D$.
}
 on $\D(\o{G}^{\mult})_W$ and $\gamma u_r/u_r \in \s_r^{\times}$.
It follows easily from these definitions that $\varphi_{\m_r^{\star}}$
linearizes to an isomorphism when $\star=\et$ and has image
contained in $\omega\cdot \m_r^{\mult}(G)$ when $\star=\mult$. 
Of course, $\m_r^{\star}(G)$ 
is contravariantly functorial in---and depends only on---the closed fiber $\o{G}^{\star}$ of $G^{\star}$.

\begin{proposition}\label{EtaleMultDescription}
	Let $G$ be an object of $\pdiv_{R_r}^{\Gamma}$ and let $\m_r^{\et}(G)$ and $\m_r^{\mult}(G)$
	be as in $(\ref{MrEtDef})$ and $(\ref{MrMultDef})$, respectively.  The map $F^r:G_0 \rightarrow G_0^{(p^r)}$ 
	$($respectively $V^r:G_0^{(p^r)}\rightarrow G_0$$)$ induces a natural isomorphism
	in $\BT_{\s_r}^{\Gamma}$
	\begin{equation}
			\m_r(G^{\et}) \simeq \m_r^{\et}(G)\qquad\text{respectively}\qquad
			\m_r(G^{\mult}) \simeq \m_r^{\mult}(G).\label{EtMultSpecialIsoms}
	\end{equation}
	These identifications are compatible with change in $r$
	in the sense that for $\star=\et$ $($respectively $\star=\mult$$)$ there is a canonical
	commutative diagram in $\BT_{\s_{r+1}}^{\Gamma}$
	\begin{equation}
	\begin{gathered}
		\xymatrix{
			{\m_{r+1}(G^{\star}\times_{R_r} R_{r+1})} 
			\ar[r]_-{\simeq}^-{(\ref{EtMultSpecialIsoms})}\ar[d]_-{\simeq} & 
			{\m_{r+1}^{\star}(G\times_{R_r} R_{r+1})} \ar@{=}[r] &
			 {\D(\o{G}^{\star})_W\tens_{W,\varphi^r} \s_{r+1}}  
			 \ar[d]^-{F\otimes\id\ (\text{respectively}\ V^{-1}\otimes\id)}_-{\simeq} \\
			{\m_r(G^{\star})\tens_{\s_r} \s_{r+1}} \ar[r]^-{\simeq}_-{(\ref{EtMultSpecialIsoms})} & 
			{\m_r^{\star}(G)\tens_{\s_r} \s_{r+1}} \ar@{=}[r] & 
			 {\D(\o{G}^{\star})_W\tens_{W,\varphi^{r-1}} \s_{r+1}}
		}
	\end{gathered}
	\label{EtMultSpecialIsomsBC}
	\end{equation}
	where the left vertical isomorphism is deduced from Theorem $\ref{CaisLauMain}$ $(\ref{BaseChangeIsom}).$
\end{proposition}

\begin{proof}
	For ease of notation, we will write $\m_r^{\star}$ and
	and $\D^{\star}$ for  $\m_r^{\star}(G)$ and $\D(\o{G}^{\star})_W$, respectively. 
	Using (\ref{BreuilSrDef}), one finds that $\scrM_r^{\et}:={\alpha_r}_*(\m_r^{\et})\in \BT_{S_r}^{\varphi,\Gamma}$
	is given by the triple
	\begin{subequations}
	\begin{equation}
		\scrM_r^{\et}:=(\D^{\et}\otimes_{W,\varphi^r} S_r,\ \D^{\et}\otimes_{W,\varphi^r} \Fil^1 S_r,\  
		F\otimes \varphi_1)
	\end{equation}	
	with $\Gamma$ acting diagonally on the tensor product.  Similarly,
	${\alpha_r}_*(\m_r^{\mult})$ is given by the triple
	\begin{equation}
		(\D^{\mult}\otimes_{W,\varphi^r} S_r,\ \D^{\mult}\otimes_{W,\varphi^r} S_r,\ 
		V^{-1} \otimes v_r\cdot\varphi)\label{WindowMultCase}
	\end{equation}
	\end{subequations}
	where $v_r=\varphi(E_r)/p$ and $\gamma\in \Gamma$ acts on $\D^{\mult}\otimes_{W,\varphi^r} S_r$
	as $\gamma \otimes \chi(\gamma)^{-1} \varphi^r(\gamma u_r/u_r)\cdot \gamma$.	
	It follows from (\ref{lambdaTransformation})
	that the $S_r$-module automorphism of $\D^{\mult}\otimes_{W,\varphi^r} S_r$
	given by multiplication by $\lambda$
	carries (\ref{WindowMultCase}) isomorphically onto the object of $\BT_{S_r}^{\varphi,\Gamma}$
	given by the triple
	\begin{equation}
		\scrM_r^{\mult}:=(\D^{\mult}\otimes_{W,\varphi^r} S_r,\ \D^{\mult}\otimes_{W,\varphi^r} S_r,\  
		V^{-1}\otimes\varphi)
	\end{equation}
	with $\Gamma$ acting {\em diagonally} on the tensor product.

	On the other hand, since $G_0^{\et}$ (respectively $G_0^{\mult}$) is \'etale 
	(respectively of multiplicative type) over $R_r/pR_r$, the relative Frobenius 
	(respectively Verscheibung) morphism of $G_0$ induces isomorphisms
	\begin{subequations}
	\begin{equation}
		\xymatrix{
			{G_0^{\et}} \ar[r]_-{\simeq}^-{F^r} &  {(G_0^{\et})^{(p^r)} \simeq 
			{\varphi^r}^*\o{G}^{\et} \times_k R_r/pR_r} 
		}\label{Ftrick}
	\end{equation}
	respectively
	\begin{equation}
		\xymatrix{
			{G_0^{\mult}} & \ar[l]^-{\simeq}_-{V^r}   {(G_0^{\mult})^{(p^r)} 
			\simeq {\varphi^r}^*\o{G}^{\mult} \times_k R_r/pR_r} 
		}\label{Vtrick}
	\end{equation} 
	\end{subequations}
	of $p$-divisible groups over $R_r/pR_r$, where we have used the fact that the map $x\mapsto x^{p^r}$
	of $R_r/pR_r$ factors as $R_r/pR_r \twoheadrightarrow k \hookrightarrow R_r/pR_r$
	in the final isomorphisms of both lines above.  Since the Dieudonn\'e crystal is compatible
	with base change and the canonical map $W\rightarrow S_r$ extends to a PD-morphism 
	$(W,p)\rightarrow (S_r, pS_r+\Fil^1 S_r)$ over $k\rightarrow R_r/pR_r$, 
	applying $\D(\cdot)_{S_r}$ to (\ref{Ftrick})--(\ref{Vtrick}) yields natural isomorphisms
	$\D(G_0^{\star})_{S_r} \simeq \D^{\star}\otimes_{W,\varphi^r} S_r$ for $\star=\et,\mult$
	which carry $F$ to $F\otimes \varphi$.  It is a straightforward exercise using the construction
	of $\scrM_r(G^{\star})$ explained in \S\ref{PhiGammaCrystals} to check
	that these isomorphisms extend to give isomorphisms $\scrM_r(G^{\et}) \simeq \scrM_r^{\et}$
		and $\scrM_r(G^{\mult}) \simeq \scrM_r^{\mult}$ in $\BT_{S_r}^{\varphi,\Gamma}$.
	By Theorem \ref{Lau},
	we conclude that we have natural isomorphisms in $\BT_{\s_r}^{\varphi,\Gamma}$
	as in (\ref{EtMultSpecialIsoms}).  The commutativity of (\ref{EtMultSpecialIsomsBC}) 
	is straightforward, using the definitions of the base change isomorphisms.
\end{proof}

Now suppose that $G$ is ordinary.
As $\m_r$ is exact by Theorem \ref{CaisLauMain} (\ref{exequiv}),
applying $\m_r$ to the connected-\'etale sequence of $G$ gives 
a short exact sequence in $\BT_{\s_r}^{\varphi,\Gamma}$
\begin{equation}
	\xymatrix{
		0\ar[r] & {\m_r(G^{\et})} \ar[r] & {\m_r(G)} \ar[r] & {\m_r(G^{\mult})} \ar[r] & 0
	}\label{ConEtOrdinary}
\end{equation}
which is contravariantly functorial and exact in $G$.
Since $\varphi_{\m_r}$ linearizes to an isomorphism on $\m_r(G^{\et})$
and is topologically nilpotent on $\m_r(G^{\mult})$, we think of
(\ref{ConEtOrdinary}) as the ``slope flitration" for Frobenius acting on $\m_r(G)$.
On the other hand, Proposition \ref{BTgroupUnivExt} and Theorem \ref{CaisLauMain} (\ref{EvaluationONR})
provide 
a canonical ``Hodge filtration" of $\m_r(G)\tens_{\s_r,\varphi} R_r\simeq \D(G_0)_{R_r}$:
\begin{equation}
	\xymatrix{
		0\ar[r] & {\omega_{G}} \ar[r] & {\D(G_0)_{R_r}} \ar[r] & {\Lie(G^t)} \ar[r] & 0
	}\label{HodgeFilOrd}
\end{equation}
which is contravariant and exact in $G$.  
Our assumption that $G$ is ordinary yields 
({\em cf.} \cite{KatzSerreTate}):

\begin{lemma}\label{HodgeFilOrdProps}
	With notation as above, there are natural and $\Gamma$-equivariant 
	isomorphisms 
\begin{equation}
	 \Lie(G^t)\simeq \D(G_0^{\et})_{R_r}  \qquad\text{and} \qquad \D(G_0^{\mult})_{R_r}\simeq \omega_G.
	 \label{FlankingIdens}
\end{equation}
	Composing these isomorphisms with the canonical maps obtained by applying $\D(\cdot)_{R_r}$ 
	to the connected-\'etale sequence of $G_0$
	yield functorial $R_r$-linear splittings of the Hodge filtration $(\ref{HodgeFilOrd})$.
	Furthermore, there is a canonical and $\Gamma$-equivariant isomorphism of split exact
	sequences of $R_r$-modules	
	\begin{equation}
	\begin{gathered}
			\xymatrix{
			0\ar[r] & {\omega_{G}} \ar[r]\ar[d]^-{\simeq} & {\D(G_0)_{R_r}} \ar[r]\ar[d]^-{\simeq} & 
			{\Lie(G^t)} \ar[r]\ar[d]^-{\simeq} & 0\\
				0 \ar[r] & {\D(\o{G}^{\mult})_W\tens_{W,\varphi^r} R_r} \ar[r]_-{i} & 
			{\D(\o{G})_W\tens_{W,\varphi^r} R_r} \ar[r]_-{j} & 
			{\D(\o{G}^{\et})_W\tens_{W,\varphi^r} R_r}\ar[r] & 0
				}
	\end{gathered}
	\label{DescentToWIsom}
	\end{equation}
	with $i,j$ the inclusion and projection mappings
	obtained from the canonical direct sum decomposition 
	$\D(\o{G})_W\simeq \D(\o{G}^{\mult})_W\oplus \D(\o{G}^{\et})_W$.
\end{lemma}

\begin{proof}
Applying $\D(\cdot)_{R_r}$ to the connected-\'etale sequence
of $G_0$ and using Proposition \ref{BTgroupUnivExt} yields a commutative diagram
with exact columns and rows
\begin{equation}
\begin{gathered}
	\xymatrix{
						&     & 0\ar[d] & 0 \ar[d] & \\
		 & 0 \ar[r]\ar[d] & {\omega_{G}}\ar[r]\ar[d] & 
		{\omega_{G^{\mult}}} \ar[r]\ar[d] & 0\\
		0\ar[r] & {\D(G_0^{\et})_{R_r}} \ar[r]\ar[d] & {\D(G_0)_{R_r}}\ar[r]\ar[d] & 
		{\D(G_0^{\mult})_{R_r}} \ar[r]\ar[d] & 0\\
		0 \ar[r] & {\Lie({G^{\et}}^t)} \ar[r]\ar[d] & {\Lie(G^t)}\ar[r]\ar[d] & 
		0 & \\
		& 0 & 0 &  &
		}
\end{gathered}
\label{OrdinaryDiagram}
\end{equation}
where we have used the fact that that the invariant differentials
and Lie algebra of an \'etale $p$-divisible group
(such as $G^{\et}$ and ${G^{\mult}}^t\simeq {G^t}^{\et}$)
are both zero.  The isomorphisms (\ref{FlankingIdens})
follow at once.  We likewise immediately see that the short exact sequence 
in the center column of (\ref{OrdinaryDiagram}) is functorially and $R_r$-linearly
split.  Thus, to prove the claimed identification in (\ref{DescentToWIsom}),
it suffices to exhibit natural isomorphisms of free $R_r$-modules with $\Gamma$-action
\begin{equation}
	\D(G_0^{\et})_{R_r} \simeq \D(\o{G}^{\et})_W\tens_{W,\varphi^r} R_r
	\qquad\text{and}\qquad
	\D(G_0^{\mult})_{R_r} \simeq \D(\o{G}^{\mult})_W\tens_{W,\varphi^r} R_r,
	\label{TwistyDieuIsoms}
\end{equation} 
both of which follow easily by applying $\D(\cdot)_{R_r}$ to
(\ref{Ftrick}) and (\ref{Vtrick}) and using the compatibility 
of the Dieudonn\'e crystal with base change as in the proof of Proposition (\ref{EtaleMultDescription}).
\end{proof}

From the slope filtration (\ref{ConEtOrdinary}) of $\m_r(G)$
we can recover both the (split) slope filtration of $\D(\o{G})_W$
and the (split) Hodge filtration (\ref{HodgeFilOrd}) of $\D(G_0)_{R_r}$:

\begin{proposition}\label{MrToHodge}
	There are canonical and $\Gamma$-equivariant isomorphisms of short exact sequences
	\begin{subequations}
	\begin{equation}
	\begin{gathered}
		\xymatrix{
			0\ar[r] & {\m_r(G^{\et})\tens_{\s_r,\varphi\circ\tau} W} \ar[r]\ar[d]^-{\simeq} & 
			{\m_r(G)\tens_{\s_r,\varphi\circ\tau} W} \ar[r]\ar[d]^-{\simeq} & 
			{\m_r(G^{\mult})\tens_{\s_r,\varphi\circ\tau} W} \ar[r]\ar[d]^-{\simeq} & 0 \\
			0 \ar[r] & {\D(\o{G}^{\et})_W} \ar[r] & {\D(\o{G})_W} \ar[r] & {\D(\o{G}^{\mult})_W}
			\ar[r] & 0
		}\label{MrToDieudonneMap}
	\end{gathered}
	\end{equation}
	\begin{equation}
	\begin{gathered}
		\xymatrix{
			0\ar[r] & {\m_r(G^{\et})\tens_{\s_r,\theta\circ\varphi} R_r} \ar[r]\ar[d]^-{\simeq} & 
			{\m_r(G)\tens_{\s_r,\theta\circ\varphi} R_r} \ar[r]\ar[d]^-{\simeq} & 
			{\m_r(G^{\mult})\tens_{\s_r,\theta\circ\varphi} R_r} \ar[r]\ar[d]^-{\simeq} & 0 \\
			0 \ar[r] &  {\Lie(G^t)} \ar[r]_-{i} & {\D(G_0)_{R_r}} \ar[r]_-{j} & {\omega_{G}} \ar[r] & 0\\
		}
	\end{gathered}\label{MrToHodgeMap}
	\end{equation}
	\end{subequations}
	Here, $i:\Lie(G^t)\hookrightarrow \D(G_0)_{R_r}$
	and $j:\D(G_0)_{R_r}\twoheadrightarrow \omega_{G}$
	are the canonical splittings of Lemma $\ref{HodgeFilOrdProps}$,
	the top row of $(\ref{MrToHodgeMap})$ is obtained from $(\ref{ConEtOrdinary})$ by extension of scalars,
	and the isomorphism $(\ref{MrToDieudonneMap})$ intertwines $\varphi_{\m_r(\cdot)}\otimes \varphi$ with $F$.
\end{proposition}

\begin{proof}
	This follows immediately from Theorem \ref{CaisLauMain} (\ref{EvaluationONW}) and Lemma \ref{HodgeFilOrdProps}.
\end{proof}

\section{Ordinary \texorpdfstring{$\Lambda$}{Lambda}-adic Dieudonn\'e and \texorpdfstring{$(\varphi,\Gamma)$}{(phi,Gamma)}-modules}\label{results}

In this section, we will state and prove our main results as described in \S\ref{resultsintro}.
Throughout, we will use the notation of \S\ref{resultsintro} and of \cite[\S2.2]{CaisHida1},
which we now briefly recall.  

For $r\ge 1$, we write $X_r:=X_1(Np^r)$ 
for the canonical model over $\Q$ with rational cusp at $i\infty$
of the modular curve arising as the quotient of the extended upper-halfplane
by the congruence subgroup $\Upgamma_1(Np^r)$ ({\em cf.} \cite[Remark 2.2.4]{CaisHida1}).
There are two natural degeneracy mappings $\rho,\sigma:X_{r+1}\rightrightarrows X_r$
of curves over $\Q$ induced by the self-maps of the upper-halfplane $\rho:\tau\mapsto \tau$ and 
$\sigma:\tau\mapsto p\tau$; see \cite[Remark 2.2.5]{CaisHida1}. 
Denote by $J_r:=\Pic^0_{X_r/\Q}$ the Jacobian of $X_r$ over $\Q$
and write $\H_r(\Z)$
for the $\Z$-subalgebra of $\End_{\Q}(J_r)$ generated by the
Hecke operators $\{T_{\ell}\}_{\ell\nmid Np}$, $\{U_{\ell}\}_{\ell|Np}$
and the Diamond operators $\{\langle u\rangle\}_{u\in \Z_p^{\times}}$.
We define $\H_r(\Z)^{*}$ similarly, using instead the ``transpose"
Hecke and diamond operators, and set $\H_r:=\H_r(\Z)\otimes_{\Z}\Z_p$
and $\H_r^*:=\H_r(\Z)^*\otimes_{\Z}\Z_p$; see \cite[2.2.21--2.2.23]{CaisHida1}.
As usual, we write $e_r\in \H_r$ and $e_r^*\in \H_r^*$ for the idempotents of these
semi-local $\Z_p$-algebras corresponding to the Atkin operators $U_p$ and $U_p^*$,
respectively, and we put $e:=(e_r)_r$ and $e^*:=(e_r^*)_r$
for the induced idempotents of the ``big" $p$-adic Hecke 
algebras $\H:=\varprojlim_r \H_r$ and $\H^*:=\varprojlim_r \H_r^*$;
here, the maps in these projective limits are 
induced by the natural transition mappings on Jacobians $J_r\rightrightarrows J_{r'}$
for $r'\ge r$ arising (via Picard functoriality) from $\sigma$ and $\rho$, respectively.
Let $w_r$ be the Atkin--Lehner 
``involution" of $X_r$ over $\Q(\mu_{Np^r})$ corresponding to a choice
of primitive $Np^r$-th root of unity as in the discussion preceding \cite[Proposition 2.2.6]{CaisHida1};
following the conventions of \cite[\S2.2]{CaisHida1}, we simply write $w_r$
for the automorphism $\Alb(w_r)$ of $J_r$ over $\Q(\mu_{Np^r})$ induced by Albanese
functoriality.
We note that for any Hecke operator $T\in \H_r(\Z)$, one has the relation 
$w_rT=T^*w_r$ as endomorphisms of $J_r$ over $\Q(\mu_{Np^r})$
\cite[Proposition 2.2.24]{CaisHida1}.

\subsection{\texorpdfstring{$\Lambda$}{Lambda}-adic Barsotti-Tate groups}\label{BTfamily}

In order to construct a crystalline analogue of Hida's ordinary $\Lambda$-adic \'etale cohomology,
we will apply the theory of \S\ref{PhiGammaCrystals} to a certain ``tower" 
$\{\G_r\}_{r\ge 1}$ of $p$-divisible groups (a $\Lambda$-adic Barsotti Tate group
in the sense of Hida \cite{HidaLambdaBT}, \cite{HidaNotes}, \cite{HidaNotes2}) whose construction involves 
artfully cutting out certain $p$-divisible 
subgroups of $J_r[p^{\infty}]$ over $\Q$ and the ``good reduction'' theorems of Langlands-Carayol-Saito. 
The construction of $\{\G_r\}_{r\ge 1}$ is certainly well-known (e.g. \cite[\S1]{MW-Hida}, 
\cite[Chapter 3, \S1]{MW-Iwasawa}, \cite[Definition 1.2]{Tilouine}  
and \cite[\S 3.2]{OhtaEichler}), but as we shall need substantially finer information about the $\G_r$
than is available in the literature, we devote this section to recalling their construction and properties.

As in \cite[\S 3.3]{CaisHida1}, for a ring $A$, a nonegative integer $k$, and a congruence subgroup $\Upgamma$
of $\SL_2(\Z)$, we write
$S_k(\Upgamma;A)$ for the space of weight $k$ cuspforms for $\Gamma$ over $A$, and
for ease of notation we put $S_k(\Upgamma):=S_k(\Upgamma;\Qbar)$.
If $\Upgamma',$ $\Upgamma$ are congruence subgroups, then associated to any 
$\gamma\in \GL_2(\Q)$ with $\gamma^{-1}\Upgamma'\gamma\subseteq \Upgamma$ is an injective pullback mapping
$\xymatrix@1{{\iota_{\gamma}:S_k(\Upgamma)} \ar@{^{(}->}[r] & {S_k(\Upgamma')}}$ 
given by $\iota_{\gamma}(f):=f\big|_{\gamma^{-1}}$,
as well as a surjective ``trace" mapping
\begin{equation}
	\xymatrix{
		{\tr_{\gamma}:S_k(\Upgamma')} \ar@{->>}[r] & {S_k(\Upgamma)}
		}
		\qquad\text{given by}\qquad
		\tr_{\gamma}(f):=\sum_{\delta\in \gamma^{-1}\Upgamma'\gamma\backslash\Upgamma} (f\big|_{\gamma})\big|_{\delta}
		\label{MFtrace}
\end{equation}
with $\tr_{\gamma}\circ\iota_{\gamma}$ multiplication by $[\Upgamma: \gamma^{-1}\Upgamma'\gamma]$
on $S_k(\Upgamma)$.  If $\Upgamma'\subseteq \Upgamma$, then
{\em unless specified to the contrary}, we will
always view $S_k(\Upgamma)$ as a subspace of $S_k(\Upgamma')$ via $\iota_{\id}$.

For nonnegative integers $i\le r$ we set
$\Upgamma_r^i:=\Upgamma_1(Np^i)\cap \Upgamma_0(p^r)$ for the intersection $($taken inside $\SL_2(\Z)$$)$,
and put $\Upgamma_r:=\Upgamma_r^r$.
We will need the following fact 
({\em cf.} \cite[pg. 339]{Tilouine}, \cite[2.3.3]{OhtaEichler})
concerning the trace mapping $(\ref{MFtrace})$ attached to the canonical inclusion 
$\Upgamma_{r}\subseteq \Upgamma_i$ for $r\ge i$; for notational clarity, 
we will write $\tr_{r,i}:S_k(\Upgamma_r)\rightarrow S_k(\Upgamma_i)$ for this map.

\begin{lemma}\label{MFtraceLem}
	Fix integers $i\le r$ and let $\tr_{r,i}:S_k(\Upgamma_r)\rightarrow S_k(\Upgamma_i)$ 
	be the trace mapping $(\ref{MFtrace})$ attached to the inclusion 
	$\Upgamma_r\subseteq \Upgamma_i$.  For $\alpha:=\left(\begin{smallmatrix} 1 & 0 \\ 0 & p\end{smallmatrix}\right)$, 
	we 	have an equality
	of $\o{\Q}$-endomorphisms of $S_k(\Upgamma_{r})$
	\begin{equation}
		\iota_{\alpha^{r-i}}\circ \tr_{r,i} = (U_p^*)^{r-i} 
		\sum_{\delta\in \Delta_i/\Delta_{r}} \langle \delta \rangle.
		\label{DualityIdentity}
	\end{equation}
\end{lemma}

\begin{proof}
	We have
	index $p^{r-i}$ inclusions of groups $\Upgamma_{r} \subseteq \Upgamma_{r}^i \subseteq \Upgamma_i$
	with $\Upgamma_{r}$ normal in $\Upgamma_{r}^i$, as it is the kernel of the canonical
	surjection $\Upgamma_{r}^i\twoheadrightarrow \Delta_i/\Delta_{r}$.    
	For each $\delta\in \Delta_i/\Delta_{r}$, we
	fix a choice of $\sigma_{\delta}\in \Upgamma_{r}^i$ mapping to $\delta$
	and calculate that	
	\begin{equation}
		\Upgamma_i = \coprod_{\delta\in \Delta_i/\Delta_{r}} \coprod_{j=0}^{p^{r-i}-1}   		
		\Upgamma_{r}\sigma_{\delta} \varrho_j
		\qquad\text{where}\qquad
		\varrho_j:=\begin{pmatrix} 1 & 0 \\ jNp^i & 1\end{pmatrix}.\label{CosetDecomp}
	\end{equation}  
	On the other hand, for each $0\le j < p^{r-i}$ one has the equality of matrices in $\GL_2(\Q)$
	\begin{equation}
		p^{r-i}\varrho_j \alpha^{-(r-i)} = \tau_{r} \begin{pmatrix} 1 & -j \\ 0 & p^{r-i} \end{pmatrix} \tau_{r}^{-1}
		\qquad\text{for}\qquad \tau_{r} := \begin{pmatrix} 0 & -1 \\ Np^{r} & 0 \end{pmatrix}. 
		\label{EasyMatCalc}
	\end{equation}
	The claimed equality (\ref{DualityIdentity}) follows easily from (\ref{CosetDecomp}) and (\ref{EasyMatCalc}), using
	the equalities of operators $(\cdot)\big|_{\sigma_{\delta}}=\langle \delta\rangle $ 
	and $U_p^* = w_{r} U_p w_{r}^{-1}$ on $S_k(\Upgamma_{r})$ \cite[Proposition 2.2.24]{CaisHida1}.
\end{proof}

Perhaps the most essential ``classical" fact for our purposes is that 
the Hecke operator $U_p$ acting on spaces of  modular forms ``contracts" the $p$-level, as is made precise by the following:  

\begin{lemma}\label{UpContract}
	If $f\in S_k(\Upgamma_r^i)$ then $U_p^{d}f$ is in the image of the canonical map
	$\iota_{\id}:S_k(\Upgamma_{r-d}^i)\hookrightarrow S_k(\Upgamma_r^i)$ for each integer $d\le r-i$.  In particular,
	$U_p^{r-i}f$ is in the image of $S_k(\Upgamma_i)\hookrightarrow S_k(\Upgamma_r^i)$.	
\end{lemma}

Certainly Lemma \ref{UpContract} is well-known 
(e.g. \cite{Tilouine}, \cite{Ohta1}, \cite{HidaLambdaBT});
because of its importance in our subsequent applications, we sketch a proof (following
the proof of \cite[Lemma 1.2.10]{Ohta1}; see also \cite{HidaLambdaBT} and \cite[\S 2]{HidaNotes}).
We note that $\Upgamma_r\subseteq \Upgamma_r^i$ for all $i\le r$, and the resulting inclusion
$S_k(\Upgamma_r^i)\hookrightarrow S_k(\Upgamma_r)$ has image 
consisting of forms on $\Upgamma_r$ which are eigenvectors for the diamond operators
and whose associated character has conductor with $p$-part dividing $p^{i}$.

\begin{proof}[Proof of Lemma $\ref{UpContract}$]
	Fix $d$ with $0\le d\le r-i$ and let $\alpha:=\left(\begin{smallmatrix} 1 & 0 \\ 0 & p\end{smallmatrix}\right)$
	be as in Lemma \ref{MFtraceLem};
	then $\alpha^d$ is an element of the commeasurator of $\Upgamma_{r}^i$ in $\SL_2(\Q)$.  Consider the following
	subgroups of $\Upgamma_{r-d}^i$:
	\begin{align*}
		H&:= \Upgamma_{r-d}^i \cap \alpha^{-d}\Upgamma_{r}^i\alpha^d\\
		H'&:= \Upgamma_{r-d}^i \cap \alpha^{-d}\Upgamma_{r-d}^i\alpha^d,
	\end{align*}
	with each intersection taken inside of $\SL_2(\Q)$.  We claim that $H=H'$ inside $\Upgamma_{r-d}^i$.
	Indeed, as $\Upgamma_{r}^i\subseteq \Upgamma_{r-d}^i$, the inclusion $H\subseteq H'$ is clear.
	For the reverse inclusion, if $\gamma:=\left(\begin{smallmatrix} * & * \\ x & *\end{smallmatrix}\right)\in \Upgamma_{r-d}^i$,
	then we have $\alpha^{-d}\gamma\alpha^d = \left(\begin{smallmatrix} * & * \\ p^{-d}x & *\end{smallmatrix}\right)$,
	so if this lies in $\Upgamma_{r-d}^i$ we must have $x\equiv 0\bmod p^r$ and hence $\gamma\in \Upgamma_r^i$.
	We conclude that the coset spaces $H\backslash\Upgamma_{r-d}^i$ and $H'\backslash\Upgamma_{r-d}^i$
	are equal.  On the other hand, for {\em any} commeasurable subgroups $\Upgamma,\Upgamma'$ of a group $G$
	and any $g$ in the commeasurator of $\Upgamma$ in $G$,
	an elementary computation shows that we have a bijection of coset spaces
	\begin{align*}
		(\Upgamma'\cap g^{-1}\Upgamma g )\backslash \Upgamma' \simeq \Upgamma\backslash\Upgamma g\Upgamma'
	\end{align*}
	via $(\Upgamma'\cap g^{-1}\Upgamma g)\gamma\mapsto \Upgamma g\gamma$.
	Applying this with $g=\alpha^d$ in our situation and using the 
	decomposition
	\begin{equation*}
				\Upgamma_{r-d}^i \alpha^d \Upgamma_{r-d}^i = \coprod_{j=0}^{p^{d}-1} \Upgamma_{r-d}^i\begin{pmatrix} 1 & j \\ 0 & p^{d}\end{pmatrix}
	\end{equation*}
	(see, e.g. \cite[proposition 3.36]{Shimura}), we deduce that we also have
	\begin{equation}\label{disjointHecke} 
		\Upgamma_r^i \alpha^d \Upgamma_{r-d}^i = \coprod_{j=0}^{p^{d}-1} \Upgamma_r^i\begin{pmatrix} 1 & j \\ 0 & p^{d}\end{pmatrix}.
	\end{equation}
	Writing $U:S_k(\Upgamma_r^i)\rightarrow S_k(\Upgamma_{r-d}^i)$ for the ``Hecke operator" given by
	(e.g. \cite[\S3.4]{Ohta1})
	$\Upgamma_r^i \alpha^d \Upgamma_{r-d}^i$, an easy computation using \ref{disjointHecke} shows that the composite
	\begin{equation*}
		\xymatrix{
			S_k(\Upgamma_r^i) \ar[r]^-{U} & S_{k}(\Upgamma_{r-d}^i) \ar@{^{(}->}[r] & S_k(\Upgamma_r^i) 
		}
	\end{equation*}
	coincides with $U_p^d$ on $q$-expansions.  By the $q$-expansion principle, we deduce that $U_p^d$ on $S_k(\Upgamma_r^i)$
	indeed factors through the subspace $S_k(\Upgamma_{r-d}^i)$, as desired.	
\end{proof}

For each integer $i$ and any character $\varepsilon:(\Z/Np^i\Z)^{\times}\rightarrow \Qbar^{\times}$, we denote by
$S_2(\Upgamma_i,\varepsilon)$
the $\H_i$-stable subspace of weight 2 cusp forms for $\Upgamma_i$ over $\Qbar$ 
on which the diamond operators act through $\varepsilon(\cdot)$. 
Define
\begin{equation}
	\o{V}_r := \bigoplus_{i=1}^r\bigoplus_{\varepsilon } S_2(\Upgamma_i,\varepsilon)
	\label{VrDef}
\end{equation}
where the inner sum is over all Dirichlet characters defined modulo $Np^i$ whose $p$-parts are {\em primitive}
({\em i.e.} whose conductor has $p$-part exactly $p^i$).
We view $\o{V}_r$ as a $\Qbar$-subspace of $S_2(\Upgamma_r)$ in the usual way
({\em i.e.} via the embeddings $\iota_{\id}$).  
We define $\o{V}_r^*$ as the direct sum (\ref{VrDef}), but viewed
as a subspace of $S_2(\Upgamma_r)$ via the ``nonstandard" embeddings
$\iota_{\alpha^{r-i}}:S_2(\Upgamma_i)\rightarrow S_2(\Upgamma_r)$.

As in \cite[2.5.17]{CaisHida1}, we write $f'$
for the idempotent of $\Z_{(p)}[\F_p^{\times}]$ corresponding to ``projection away from the trivial $\F_p^{\times}$-eigenspace;"
explicitly, we have
\begin{equation}
	f':=1 - \frac{1}{p-1}\sum_{g\in \F_p^{\times}} g.\label{projaway}
\end{equation} 
We set $h':=(p-1)f'$, so that $h'^2 = (p-1)h'$ and define endomorphisms of $S_2(\Upgamma_r)$:
\begin{equation}
	U_r^*:=h'\circ (U_p^*)^{r+1} = (U_p^*)^{r+1}\circ h'\quad\text{and}\quad
	U_r:=h'\circ (U_p)^{r+1} = (U_p)^{r+1}\circ h'.
	\label{UrDefinition}
\end{equation}

\begin{corollary}\label{UpProjection}
	As subspaces of $S_2(\Upgamma_r)$ we have $w_r(\o{V}_r^*)=\o{V}_r$.
	The space $\o{V}_r$ $($respectively $\o{V}_r^*$$)$ is naturally an $\H_r$ $($resp. $\H_r^*$$)$-stable 
	subspace of $S_2(\Upgamma_r)$, and admits a canonical descent to $\Q$.
	Furthermore, the endomorphisms 
	$U_r$ and $U_r^*$ of $S_2(\Upgamma_r)$ 
	factor through $\o{V}_r$ and $\o{V}_r^*$, respectively.
\end{corollary}

\begin{proof}
	The first assertion follows from the relation $w_r\circ \iota_{\alpha^{r-i}}=\iota_{\id}\circ w_i$
	as maps $S_2(\Upgamma_i)\rightarrow S_2(\Upgamma_r)$, together with the fact that $w_i$ on $S_2(\Upgamma_i)$
	carries $S_2(\Upgamma_i,\varepsilon)$ isomorphically onto $S_2(\Upgamma_i,\varepsilon^{-1})$.
	The $\H_r$-stability of $\o{V}_r$ is clear as each of $S_2(\Upgamma_i,\varepsilon)$ is an $\H_r$-stable subspace
	of $S_2(\Upgamma_r)$; that $\o{V}_r^*$ is $\H_r^*$-stable follows from this and the comutation
	relation $T^* w_r = w_r T$  \cite[Proposition 2.2.4]{CaisHida1}.  That $\o{V}_r$  and $\o{V}_r^*$
	admit canonical descents to $\Q$ is clear, as $\scrG_{\Q}$-conjugate Dirichlet characters have equal
	conductors.  The final assertion concerning the endomorphisms $U_r$ and $U_r^*$
	follows easily from Lemma \ref{UpContract}, 
	using the fact that $h':S_2(\Upgamma_r)\rightarrow S_2(\Upgamma_r)$
	has image contained in $\bigoplus_{i=1}^r S_k(\Upgamma_r^i)$.
\end{proof}

\begin{definition}
	We denote by $V_r$ and $V_r^*$ the canonical descents to $\Q$ of $\o{V}_{r}$
	and $\o{V}_r^*$, respectively.  
\end{definition}

Following \cite[Chapter \Rmnum{3}, \S1]{MW-Iwasawa} and \cite[\S2]{Tilouine},
we recall the construction of certain ``good" quotient abelian varieties of $J_r$ whose
cotangent spaces are naturally identified with $V_r$ and $V_r^*$.  In what follows,
we will make frequent use of the following elementary result:  

\begin{lemma}\label{LieFactorization}
	Let $f:A\rightarrow B$ be a homomorphism of commutative group varieties 
	over a field $K$ of characteristic $0$.  Then:
\begin{enumerate}
	\item The formation of $\Lie$ and $\Cot$ commutes with the formation of kernels and images: the kernel $($respectively image$)$ of $\Lie(f)$ is canonically
	isomorphic to the Lie algebra of the kernel $($respectively image$)$ of $f$, and 
	similarly for cotangent spaces at the identity.	
	In particular, if $A$ is connected and $\Lie(f)=0$ $($respectively $\Cot(f)=0$$)$ 
	then $f=0$.\label{ExactnessOfLie}
	
	\item Let $i:B'\hookrightarrow B $ be a closed immersion of commutative
	group varieties over $K$ with $B'$ connected.  If $\Lie(f)$ factors through $\Lie(i)$
	then $f$ factors $($necessarily uniquely$)$ through $i$. 
	\label{InclOnLie}
	
	\item Let $j:A\twoheadrightarrow A''$ be a surjection of commutative 
	group varieties over $K$ with connected kernel.  If $\Cot(f)$
	factors through $\Cot(j)$ then $f$ factors $($necessarily uniquely$)$ through $j$.
	\label{LieFactorizationSurj} 
\end{enumerate}
\end{lemma}

\begin{proof}
	The key point is that because $K$ has characteristic zero, the functors $\Lie(\cdot)$
	and $\Cot(\cdot)$ on the category of commutative group schemes are {\em exact}.
	Indeed, since $\Lie(\cdot)$ is always left 
	exact, the exactness of $\Lie(\cdot)$ follows easily from the fact that any 
	quotient mapping 
	$G\twoheadrightarrow H$
	of group varieties in characteristic zero is smooth (as the kernel is
	a group variety over a field of characteristic zero and hence automatically smooth),
	so the induced map on Lie algebras is a surjection.
	By similar reasoning 
	one shows that the right exact $\Cot(\cdot)$ is likewise exact,
	and the first part of (\ref{ExactnessOfLie}) follows easily.  In particular,
	if $\Lie(f)$ is the zero map then $\Lie(\im(f))=0$, so $\im(f)$ is zero-dimensional.
	Since it is  also smooth, it must be \'etale.  Thus, if $A$ is connected,
	then $\im(f)$ is both connected and \'etale, whence it is a single point; by 
	evaluation of $f$ at the identity of $A$ we conclude that $f=0$.
	The assertions (\ref{InclOnLie}) and (\ref{LieFactorizationSurj})
	now follow immediately by using universal mapping properties.
\end{proof}

To proceed with the construction of good quotients of $J_r$, we write
$Y_r:=X_1(Np^r; Np^{r-1})$ for the canonical model over $\Q$ with rational cusp at $i\infty$
of the modular curve corresponding to the congruence subgroup $\Upgamma_{r+1}^r$ 
({\em cf.} \cite[Remark 2.2.18]{CaisHida1}),
and consider the diagrams of ``degeneracy mappings" of curves over $\Q$ for $i=1,2$
\begin{equation}
\addtocounter{equation}{1}
	\xymatrix{
		{X_r} \ar[r]^-{\pi} & {Y_{r-1}} \ar[r]^-{\pi_i} & {X_{r-1}}
	}
	\tag{$\arabic{section}.\arabic{subsection}.\arabic{equation}_i$}
	\label{DegeneracyDiag}
\end{equation}
where $\pi$ and $\pi_2$ are induced by the canonical inclusions 
of subgroups $\Upgamma_{r} \subseteq \Upgamma_r^{r-1}\subseteq \Upgamma_{r-1}$
via the upper-halfplane self map $\tau\mapsto \tau$, and $\pi_1$ is induced
by the inclusion $\alpha^{-1} \Upgamma_r^{r-1} \alpha \subseteq \Upgamma_{r-1}$
via the mapping $\tau\mapsto p\tau$ where $\alpha$ is as in Lemma \ref{MFtraceLem};
see \cite[2.2.9]{CaisHida1} for a moduli-theoretic description of these maps.
We note that the compositions $\pi\circ \pi_2$ and $\pi\circ\pi_1$
coincide with the degeneracy maps $\rho$ and $\sigma$, respectively \cite[Remark 2.2.18]{CaisHida1}.

These mappings covariantly (respectively contravariantly) induce
mappings on the associated Jacobians via Albanese (respectively Picard) functoriality.
Writing $JY_r:=\Pic^0_{Y_r/\Q}$ and setting $K_1^i:=JY_1$ for $i=1,2$
we inductively define abelian subvarieties $\iota_r^i:K_r^i\hookrightarrow JY_r$ and abelian variety quotients
$\alpha_r^i:J_r\twoheadrightarrow B_r^i$ as follows:
\begin{equation}
	\addtocounter{equation}{1}
	B_{r-1}^i:= J_{r-1}/\Pic^0(\pi)(K_{r-1}^i) 
	\qquad\text{and}\qquad
	K_{r}^i:=\ker(JY_r \xrightarrow{\alpha_{r-1}^i\circ\Alb(\pi_i)}  B_{r-1}^i)^0
	\tag{$\arabic{section}.\arabic{subsection}.\arabic{equation}_i$}
	\label{BrDef}
\end{equation}
for $r\ge 2$, $i=1,2$, with $\alpha_{r-1}^i$ and $\iota_r^i$ the obvious mappings;
here, $(\cdot)^0$ denotes the connected component of the identity of $(\cdot)$.
As in \cite[\S 3.2]{OhtaEichler}, we have modified Tilouine's construction
\cite[\S2]{Tilouine} so that the kernel of $\alpha_r^i$ is connected; {\em i.e.} is an abelian
subvariety of $J_r$ ({\em cf.} Remark \ref{TilouineReln}).
Note that we have a commutative diagram of abelian varieties over $\Q$
for $i=1,2$
\begin{equation}
\addtocounter{equation}{1}
\begin{gathered}
	\xymatrix@C=50pt@R=26pt{
		& {J_{r-1}}\ar@{->>}[r]^-{\alpha_{r-1}^i}  & 
		{B_{r-1}^i} \ar@{=}[d] \\
		{K_r^i} \ar@{^{(}->}[r]^-{\iota_r^i}\ar@{=}[d] & 
		{JY_r} \ar[r]^-{\alpha_{r-1}^i\circ \Alb(\pi_i)} \ar[d]|-{\Pic^0(\pi)}\ar[u]|-{\Alb(\pi_i)} & 
		B_{r-1}^i \\
		K_r^i \ar[r]_-{\Pic^0(\pi)\circ \iota_r} & {J_r} \ar@{->>}[r]_-{\alpha_r^i} & {B_r^i}
	}
\end{gathered}
\tag{$\arabic{section}.\arabic{subsection}.\arabic{equation}_i$}
\label{BrDefiningDiag}
\end{equation}
with bottom two horizontal rows that are complexes.  

\begin{warning}\label{GoodQuoWarning}
While the bottom row of (\ref{BrDefiningDiag}) is exact in the middle {\em by definition} of $\alpha_r^i$,
the central row is {\em not} exact in the middle: it follows from
the fact that $\Alb(\pi_i)\circ\Pic^0(\pi_i)$ is multiplication by $p$
on $J_{r-1}$ that the component group of the kernel of
$\alpha_{r-1}^i\circ\Alb(\pi_i):JY_r\rightarrow  B_{r-1}^i$ is nontrivial
with order divisible by $p$.  
Moreover, there is no mapping $B_{r-1}^i\rightarrow B_r^i$
which makes the diagram (\ref{BrDefiningDiag}) commute.
\end{warning}

In order to be consistent with the literature, we adopt the following convention:
\begin{definition}\label{BalphDef}
	We set $B_r:=B_r^2$ and $B_r^*:=B_r^1$, with $B_r^i$ defined inductively by (\ref{BrDef}).
 	We likewise set $\alpha_r:=\alpha_r^2$ and $\alpha_r^*:=\alpha_r^1$.
\end{definition} 

\begin{remark}\label{TilouineReln}
	We briefly comment on the relation between our quotient $B_r$
	and the ``good" quotients of $J_r$ considered by Ohta \cite{OhtaEichler},
	by Mazur-Wiles \cite{MW-Iwasawa}, and by Tilouine \cite{Tilouine}.
	Recall \cite[\S2]{Tilouine} that Tilouine constructs\footnote{The notation Tilouine 
	uses for his quotient is the same as the notation we have used for our (slightly modified)
	quotient.  To avoid conflict, we have therefore chosen to 
	alter his notation.} an abelian variety quotient
	$\alpha_r':J_r\twoheadrightarrow B_r'$ via an inductive
	procedure nearly identical to the one used to define $B_r=B_r^1$:
	one sets $K_1':=JY_1$, and for $r\ge 2$ defines
\begin{equation*}
		B_{r-1}':= J_{r-1}/\Pic^0(\pi)(K_{r-1}') 
		\qquad\text{and}\qquad
		K_{r}':=\ker(JY_r \xrightarrow{\alpha_{r-1}'\circ\Alb(\pi_2)}  B_{r-1}').
\end{equation*}
	Using the fact that the 
	formation of images and identity components commutes, 
	one shows via a straightforward induction argument that 
	$\alpha_r:J_r\twoheadrightarrow B_r$
	identifies $B_r$ with $J_r/(\ker\alpha_r')^0$; in particular,
	our $B_r$ is the same as Ohta's \cite[\S3.2]{OhtaEichler}
	and Tilouine's quotient $\alpha_r':J_r\rightarrow B_r'$ uniquely
	factors through $\alpha_r$ via an isogeny $B_r\twoheadrightarrow B_r'$
	which has degree divisible by $p$ by Warning \ref{GoodQuoWarning}.
	Due to this fact, it is {\em essential} for our purposes to work with $B_r$ rather than $B_r'$.
	Of course, following \cite[3.2.1]{OhtaEichler}, we could have simply 
	{\em defined} $B_r$ as $J_r/(\ker\alpha_r')^0$,
	but we feel that the construction we have given is more natural.
	On the other hand, we remark that $B_r$
	is naturally a quotient of the ``good" quotient $J_r\twoheadrightarrow A_r$ constructed
	by Mazur-Wiles in \cite[Chapter \Rmnum{3}, \S1]{MW-Iwasawa},
	and the kernel of the corresponding surjective homomorphism 
	$A_r\twoheadrightarrow B_r$
	is isogenous to $J_0\times J_0$.
\end{remark}

\begin{proposition}\label{BrCotIden}
	Over $F:=\Q(\mu_{Np^r})$, the automorphism $w_r$ of ${J_r}_{F}$ induces
	an isomorphism of quotients ${B_r}_{F}\simeq {B_r^*}_{F}$.	
	The abelian variety $B_r$ $($respectively $B_r^*$$)$ is the unique quotient of $J_r$ by a $\Q$-rational 
	abelian subvariety with the property that
	the induced map on cotangent spaces 
	\begin{equation*}
		\xymatrix{
			{\Cot(B_r)} \ar@{^{(}->}[r]_-{\Cot(\alpha_r)} &  
			{\Cot(J_r)\simeq S_2(\Upgamma_r;\Q)}
			}
			\quad\text{respectively}\quad
					\xymatrix{
			{\Cot(B_r^*)} \ar@{^{(}->}[r]_-{\Cot(\alpha_r^*)} &  
			{\Cot(J_r)\simeq S_2(\Upgamma_r;\Q)}
			}
	\end{equation*}
	has image precisely $V_r$ $($respectively $V_r^*$$)$. 
	In particular, there are canonical actions
	of the Hecke algebras\footnote{We must warn the reader that
	Tilouine \cite{Tilouine} writes $\H_r(\Z)$ for the
	$\Z$-subalgebra of $\End(J_r)$ generated by the Hecke operators
	acting via the $(\cdot)^*$-action ({\em i.e.} by ``Picard" functoriality)
	whereas our $\H_r(\Z)$ is defined using the $(\cdot)_*$-action.
	This discrepancy is due primarily to the fact that Tilouine 
	identifies {\em tangent} spaces of modular abelian varieties
	with spaces of modular forms, rather than cotangent spaces as is our convention.  
	Our notation for regarding Hecke algebras
	as sub-algebras of $\End(J_r)$ agrees with that of Mazur-Wiles
	\cite[Chapter \Rmnum{2}, \S5]{MW-Iwasawa}, \cite[\S7]{MW-Hida} and 
	Ohta \cite[3.1.5]{OhtaEichler}.
	}
	$\H_r(\Z)$ on $B_r$ and $\H_r^*(\Z)$ on $B_r^*$ for which $\alpha_r$ and $\alpha_r^*$
	are equivariant.  
\end{proposition}

\begin{proof}
	By the construction of $B_r^i$ and the fact that $\pr w_{r} = w_{r-1}\ps$
	as maps ${X_{r}}_F \rightarrow {X_{r-1}}_F$ \cite[Proposition 2.2.6]{CaisHida1}
	the automorphism $w_{r}$ of ${J_r}_F$ carries $\ker(\alpha_r)$
	to $\ker(\alpha_r^*)$ and induces an isomorphsm ${B_r}_F \simeq {B_r^*}_F$ over $F$
	that intertwines the action of $\H_r$ on $B_r$ with $\H_r^*$ on $B_r^*$.
	The isogeny $B_r\twoheadrightarrow B_r'$ of Remark \ref{TilouineReln} induces
	an isomorphism on cotangent spaces, compatibly with the inclusions into
	$\Cot(J_r)$.  Thus, the claimed identification of the image of
	$\Cot(B_r)$ with $V_r$ follows from \cite[Proposition 2.1]{Tilouine}
	(using \cite[Definition 2.1]{Tilouine}).  The claimed uniqueness
	of $J_r\twoheadrightarrow B_r$ follows easily from Lemma \ref{LieFactorization} 
	(\ref{LieFactorizationSurj}). Similarly, since the subspace $V_r$ 
	of $S_2(\Upgamma_r)$ is stable under $\H_r$, we conclude from
	Lemma \ref{LieFactorization} (\ref{LieFactorizationSurj}) that for any 
	$T\in \H_r(\Z)$, the induced morphism $J_r\xrightarrow{T} J_r\twoheadrightarrow B_r$
	factors through $\alpha_r$, and hence that $\H_r(\Z)$ acts on $B_r$
	compatibly (via $\alpha_r$) with its action on $J_r$.
\end{proof}

\begin{lemma}\label{Btower}
	There exist unique morphisms $B_r^*\leftrightarrows B_{r-1}^*$
	of abelian varieties over $\Q$ making
		\begin{equation*}
			\xymatrix{
				{J_{r}} \ar[r]^-{\alpha_r^*} \ar[d]_-{\Alb(\ps)} &{B_r^*} \ar[d] \\
				{J_{r-1}} \ar[r]_-{\alpha_{r-1}^*} & {B_{r-1}^*}
			}\qquad\raisebox{-18pt}{and}\qquad
			\xymatrix{
				{J_{r}} \ar[r]^-{\alpha_r^*} &{B_r^*}  \\
				{J_{r-1}} \ar[u]^-{\Pic^0(\pr)} \ar[r]_-{\alpha_{r-1}^*} & {B_{r-1}^*}\ar[u]
			}
		\end{equation*}
	commute; these maps are moreover $\H_r^*(\Z)$-equivariant.
	By a slight abuse of notation, we will simply write $\Alb(\ps)$ and $\Pic^0(\pr)$
	for the induced maps on $B_r^*$ and $B_{r-1}^*$, respectively.
\end{lemma}

\begin{proof}
	Under the canonical identification of $\Cot(J_r)\otimes_{\Q}\o{\Q}$ with $S_2(\Upgamma_r)$,
	the mapping on cotangent spaces induced by $\Alb(\ps)$ (respectively
	$\Pic^0(\pr)$) coincides with $\iota_{\alpha}:S_2(\Upgamma_{r-1})\rightarrow S_2(\Upgamma_r)$
	(respectively $\tr_{r,r-1}:S_2(\Upgamma_r)\rightarrow S_2(\Upgamma_{r-1})$).
	As the kernel of $\alpha_r^*:J_r\twoheadrightarrow B_r^*$ is connected by definition,
	thanks to Lemma \ref{LieFactorization} (\ref{LieFactorizationSurj}) it suffices
	to prove that $\iota_{\alpha}$ (respectively $\tr_{r,r-1}$)
	carries $V_{r-1}^*$ to $V_{r}^*$ (respectively $V_{r}^*$ to $V_{r-1}^*$).
	On one hand, the composite 
	$\iota_{\alpha}\circ \iota_{\alpha^{r-1-i}}:S_2(\Upgamma_i,\varepsilon)\rightarrow S_2(\Upgamma_r)$
	coincides with the embedding $\iota_{\alpha^{r-i}}$, and it follows immediately from the
	definition of $V_r^*$ that $\iota_{\alpha}$ carries $V_{r-1}^*$ into $V_r^*$.	
	On the other hand, an easy calculation using (\ref{DualityIdentity}) shows that 
	one has equalities of maps $S_2(\Upgamma_i,\varepsilon)\rightarrow S_2(\Upgamma_r)$
	\begin{equation*}
		\iota_{\alpha}\circ \tr_{r,r-1}\circ \iota_{\alpha^{(r-i)}} = \begin{cases}
			\iota_{\alpha^{(r-i)}}pU_p^* & \text{if}\ i< r \\
			0 & \text{if}\ i=r
		\end{cases}.
	\end{equation*}
	Thus,  
	the image of $\iota_{\alpha}\circ\tr_{r,r-1}:V_r^*\rightarrow S_2(\Upgamma_r)$ 
	is contained in the image of $\iota_{\alpha}:V_{r-1}^*\rightarrow S_2(\Upgamma_r)$;
	since $\iota_{\alpha}$ is injective, we conclude that 
	the image of $\tr_{r,r-1}:V_r^*\rightarrow S_2(\Upgamma_{r-1})$ is contained in $V_{r-1}^*$
	as desired.
\end{proof}

For $f'$ as in (\ref{projaway}), we write ${e^*}':=f'e^*\in \H^*$ and $e':=f'e\in \H$ 
the sub-idempotents of $e^*$ and $e$, respectively, 
corresponding to projection away from the trivial eigenspace of $\mu_{p-1}$.

\begin{proposition}\label{GoodRednProp}	
	The maps $\alpha_r$ and $\alpha_r^*$
	induce isomorphisms of $p$-divisible groups over $\Q$
	\begin{equation}
			{e^*}'J_r[p^{\infty}] \simeq {e^*}'B_r^*[p^{\infty}]\quad\text{and}\quad
			{e}'J_r[p^{\infty}] \simeq {e}'B_r[p^{\infty}],
	\label{OrdBTisoms}
	\end{equation}
	respectively, that are $\H^*$ $($respectively $\H$$)$ equivariant 
	and compatible with change in $r$ via $\Alb(\ps)$ and $\Pic^0(\pr)$
	$($respectively $\Alb(\pr)$ and $\Pic^0(\ps)$$)$.
\end{proposition}

We view the maps (\ref{UrDefinition})
as endomorphisms of $J_r$ in the obvious way, and again write
$U_r^*$ and $U_r$ for the induced endomorphism of
$B_r^*$ and $B_r$, respectively.  To prove Proposition \ref{GoodRednProp},
we need the following geometric incarnation of
Corollary \ref{UpProjection}:

\begin{lemma}\label{UFactorDiagLem}
		There exists a unique $\H_r^*(\Z)$ $($respectively $\H_r(\Z)$$)$-equivariant 
		map $W_r^*:B_r^*\rightarrow J_r$ $($respectively $W_r:B_r\rightarrow J_r$$)$
		of abelian varieties over $\Q$ such that the diagram
	\begin{equation}
	\begin{gathered}
		\xymatrix@C=30pt@R=35pt{
			 {J_r}\ar[d]_-{U_r^*} 
			\ar@{->>}[r]^-{\alpha_r^*} & {B_r^*} \ar[dl]|-{W_r^*} \ar[d]^-{U_r^*} \\
		 {J_r}\ar@{->>}[r]_-{\alpha_r^*} & {B_r^*}
		}
		\quad\raisebox{-24pt}{respectively}\quad
		\xymatrix@C=30pt@R=35pt{
			 {J_r}\ar[d]_-{U_r} 
			\ar@{->>}[r]^-{\alpha_r} & {B_r} \ar[dl]|-{W_r} \ar[d]^-{U_r} \\
		 {J_r}\ar@{->>}[r]_-{\alpha_r} & {B_r}
		}\label{UFactorDiag}
	\end{gathered}
	\end{equation}
	commutes.  
\end{lemma}

\begin{proof}
	Consider the endomorphism of $J_r$ given by $U_r$.
	Due to Corollary \ref{UpProjection},
	the induced mapping on cotangent spaces factors through the inclusion
	$\Cot(B_r)\hookrightarrow \Cot(J_r)$. Since the kernel of the quotient
	mapping $\alpha_r:J_r\twoheadrightarrow B_r$ giving rise to this inclusion is connected,
	we conclude from Lemma \ref{LieFactorization} (\ref{LieFactorizationSurj})
	that $U_r$ factors uniquely through $\alpha_r$
	via an $\H_r$-equivariant morphism $W_r:B_r\rightarrow J_r$.  
	The corresponding statements for $B_r^*$ are proved similarly.
\end{proof}

\begin{proof}[Proof of Proposition $\ref{GoodRednProp}$]
	From (\ref{UFactorDiag}) we get commutative diagrams of $p$-divisible groups 
	over $\Q$
	\begin{equation}
	\begin{gathered}
		\xymatrix{
			{e^*}'{J_r}[p^{\infty}]\ar[d]_-{U_r^*}^-{\simeq} \ar[r]^-{\alpha_r^*} & 
			{e^*}'{B_r^*}[p^{\infty}] \ar[dl]|-{W_r^*} \ar[d]^-{U_r^*}_-{\simeq} \\
			{e^*}'{J_r}[p^{\infty}]\ar[r]_-{\alpha_r^*} & 
			{e^*}'{B_r^*}[p^{\infty}]
		}
		\quad\raisebox{-24pt}{and}\quad
		\xymatrix{
			e'{J_r}[p^{\infty}]\ar[d]_-{U_r}^-{\simeq} \ar[r]^-{\alpha_r} & 
			e'{B_r}[p^{\infty}] \ar[dl]|-{W_r} \ar[d]^-{U_r}_-{\simeq} \\
			e'{J_r}[p^{\infty}]\ar[r]_-{\alpha_r} & 
			e'{B_r}[p^{\infty}]
		}
		\label{UFactorDiagpDiv}
	\end{gathered}
	\end{equation}
	in which all vertical arrows are isomorphisms due to the very definition of the
	idempotents ${e^*}'$ and $e'$.  An easy diagram chase then shows that {\em all}
	arrows must be isomorphisms.
\end{proof}

As in the introduction, 
we put $K_r=\Q_p(\mu_{p^r})$, $K_r':=K_r(\mu_N)$ and
write $R_r$ and $R_r'$ for the valuation rings of $K_r$
and $K_r'$, respectively.  We set $\Gamma:=\Gal(K_{\infty}/K_0)$,
and write $a:\Gal(K_0'/K_0)\rightarrow (\Z/N\Z)^{\times}$
the character giving the tautological action of $\Gal(K_0'/K_0)$
on $\mu_N$.


\begin{proposition}\label{GoodRedn}
	The abelian varieties $B_r$ and $B_r^*$ acquire good reduction over $K_r$.
\end{proposition}

\begin{proof}
	See \cite[Chap \Rmnum{3}, \S2, Proposition 2]{MW-Iwasawa} 
	and {\em cf.} \cite[\S9, Lemma 9]{HidaGalois}.	
\end{proof}

We will write $\B_r$, $\B^*_r$, and $\J_r$, respectively, for the N\'eron models of
the base changes $(B_r)_{K_r}$, $(B_r^*)_{K_r}$ and $(J_r)_{K_r}$
over $T_r:=\Spec(R_r)$;  due to Proposition \ref{GoodRednProp}, both $\B_r$ and $\B_r^*$ are abelian
schemes over $T_r$.  
By the N\'eron mapping property, there are canonical actions of $\H_r(\Z)$ on $\B_r$, $\J_r$
and of $\H_r^*(\Z)$ on $\B_r^*$, $\J_r$ over $R_r$ extending the actions on generic fibers
as well as ``semilinear" actions of $\Gamma$
over the $\Gamma$-action on $R_r$ ({\em cf.} (\ref{GammaAction})).
For each $r$, the N\'eron mapping property further provides diagrams 
\begin{equation}
	\begin{gathered}
	\xymatrix{
		 {\J_r \times_{T_r} T_{r+1}}\ar@<-1ex>[d]_{\Pic^0(\pr)} \ar[r]^-{\alpha_r^*} 
		& {\B_r^* \times_{T_r} T_{r+1}}\ar@<1ex>[d]^{\Pic^0(\pr)} \\
		{\J_{r+1}} \ar[r]_-{\alpha_{r+1}^*} \ar@<-1ex>[u]_-{\Alb(\ps)} & \ar@<1ex>[u]^-{\Alb(\ps)} {\B_{r+1}^*} 
	}
	\quad\raisebox{-24pt}{respectively}\quad
	\xymatrix{
		 {\J_r \times_{T_r} T_{r+1}}\ar@<-1ex>[d]_{\Pic^0(\ps)} \ar[r]^-{\alpha_r} 
		& {\B_r \times_{T_r} T_{r+1}}\ar@<1ex>[d]^{\Pic^0(\ps)} \\
		{\J_{r+1}} \ar[r]_-{\alpha_{r+1}}\ar@<-1ex>[u]_-{\Alb(\pr)} & \ar@<1ex>[u]^-{\Alb(\pr)} {\B_{r+1}} 
	}
	\label{Nermaps}
	\end{gathered}
\end{equation}
of smooth commutative group schemes over $T_{r+1}$ in which the inner and outer rectangles commute, and
all maps are $\H_{r+1}^*(\Z)$ (respectively $\H_{r+1}(\Z)$) and $\Gamma$ equivariant.

\begin{definition}\label{ordpdivdefn}
	We define $\G_r:={e^*}'\left(\B_r^*[p^{\infty}]\right)$ and 
	we write $\G_r':=\G_r^{\vee}$ for its Cartier dual,
	each of which is canonically an object of $\pdiv_{R_r}^{\Gamma}$.
	For each $r\ge s$, noting that $U_p^*$ is an automorphism of $\G_r$,
	we obtain from (\ref{Nermaps}) canonical morphisms
	\begin{equation}
		\xymatrix@C=45pt{
		{\rho_{r,s}:\G_{s}\times_{T_{s}} T_{r}} \ar[r]^-{\Pic^0(\pr)^{r-s}} &  {\G_{r}}
		}
		\qquad\text{and}\qquad
		\xymatrix@C=70pt{
		{\rho_{r,s}' : \G_{s}'\times_{T_{s}} T_{r}} \ar[r]^-{{({U_p^*}^{-1}\Alb(\ps))^{\vee}}^{r-s}} & {\G_{r}'}
		}\label{pdivTowers}
	\end{equation}
	in $\pdiv_{R_r}^{\Gamma}$, where $(\cdot)^{i}$ denotes the $i$-fold composition, formed in the obvious manner.
	In this way, we get towers of $p$-divisible groups 
	$\{\G_r,\rho_{r,s}\}$ and $\{\G_r',\rho_{r,s}'\}$;
	we will write $G_r$ and $G_r'$ for the unique descents of the generic fibers of $\G_r$
	and $\G_r'$ to $\Q_p$, respectively.\footnote{Of course, $G_r'=G_r^{\vee}$.  
	Our non-standard notation $\G_r'$ for the Cartier  
	dual of $\G_r$ is preferrable, due to the fact that $\rho_{r,s}'$ is 
	{\em not} simply the dual of $\rho_{r,s}$; indeed, these two mappings go in opposite
	directions!}
	We let $T^*\in \H_r^*$ act on $\G_r$ 
	through the action of $\H_r^*(\Z)$ on $\B_r^*$,
	and on $\G_r'=\G_r^{\vee}$ by duality ({\em i.e.} as $(T^*)^{\vee}$).  
	The maps (\ref{pdivTowers}) are then $\H_r^*$-equivariant.
\end{definition}

By Proposition \ref{GoodRednProp}, $G_r$ 
is canonically isomorphic to ${e^*}'J_r[p^{\infty}]$, compatibly with the action of $\H_r^*$.
Since $J_r$ is a Jacobian---hence principally polarized---one might expect that $\G_r$
is isomorphic to its dual in $\pdiv_{R_r}^{\Gamma}$.  However, this is {\em not quite} the case
as the canonical isomorphism $J_r\simeq J_r^{\vee}$ intertwines the actions of $\H_r$
and $\H_r^*$, thus interchanging the idempotents ${e^*}'$ and $e'$.  To describe
the precise relationship between $\G_r^{\vee}$ and $\G_r$, we proceed as follows.
For each $\gamma\in \Gal(K_{\infty}'/K_0)\simeq \Gamma\times \Gal(K_0'/K_0)$, 
let us write $\phi_{\gamma}: {G_r}_{K_r'}\xrightarrow{\simeq} \gamma^*({G_r}_{K_r'})$
for the descent data isomorphisms encoding the unique $\Q_p=K_0$-descent of ${G_r}_{K_r'}$ furnished by $G_r$.
We ``twist" this descent data by the $\Aut_{\Q_p}(G_r)$-valued character $\langle \chi\rangle\langle a\rangle_N$
of $\Gal(K_{\infty}'/K_0)$:
explicitly, for $\gamma\in \Gal(K_{r}'/K_0)$ we set 
$\psi_{\gamma}:= \phi_{\gamma}\circ \langle \chi(\gamma)\rangle\langle a(\gamma)\rangle_N$
and note that since $\langle \chi(\gamma)\rangle\langle a(\gamma)\rangle_N$ 
is defined over $\Q_p$, the map $\gamma\rightsquigarrow \psi_{\gamma}$
really does satisfy the cocycle condition.  
We denote by $G_r(\langle \chi\rangle\langle a\rangle_N)$ the unique $p$-divisible group over $\Q_p$
corresponding to this twisted descent datum.
Since the diamond operators commute with the Hecke operators, there is a canonical
induced action of $\H_r^*$ on $G_r(\langle \chi\rangle\langle a\rangle_N)$.  
By construction, there is a canonical 
$K_r'$-isomorphism 
$G_r(\langle \chi\rangle\langle a\rangle_N)_{K_r'}\simeq {G_r}_{K_r'}$. Since
$G_r$ acquires good reduction over $K_r$ and the $\scrG_{K_r}$-representation
afforded by the Tate module of $G_r(\langle \chi\rangle\langle a\rangle_N)$
is the twist of $T_pG_r$ by the {\em unramified} character $\langle a\rangle_N$,
we conclude that $G_r(\langle \chi\rangle\langle a\rangle_N)$ also acquires 
good reduction over $K_r$, and we denote the resulting object of $\pdiv_{R_r}^{\Gamma}$
by $\G_r(\langle \chi\rangle\langle a\rangle_N)$.

\begin{proposition}\label{GdualTwist}
	There is a natural $\H_r^*$-equivariant isomorphism of $p$-divisible groups over $\Q_p$
	\begin{equation}
		G_r' \simeq G_r(\langle \chi\rangle \langle a\rangle_N)
		\label{GrprimeGr}
	\end{equation}
	which uniquely extends to an isomorphism of the corresponding objects in $\pdiv_{R_r}^{\Gamma}$
	and is compatible with change in $r$ using $\rho_{r,s}'$ on $G_r'$ and $\rho_{r,s}$ on $G_r$.
\end{proposition}

\begin{proof}
	Let $\varphi_r: J_r\rightarrow J_r^{\vee}$ be the canonical principal polarization over $\Q_p$;
	one then has the relation $\varphi_r\circ T = (T^*)^{\vee}\circ \varphi_r$
	for each $T\in \H_r(\Z)$.  On the other hand, the $K_r'$-automorphism 
	$w_r: {J_r}_{K_r'}\rightarrow {J_r}_{K_r'}$ intertwines $T\in \H_r(\Z)$ with $T^*\in \H_r^*(\Z)$.
	Thus, the $K_r'$-morphism 
	\begin{equation*}
		\xymatrix{
			{\psi_r:{J_r}_{K_r'}^{\vee}} \ar[r]^-{({U_p^*}^{r})^{\vee}} & 
			{{J_r}_{K_r'}^{\vee}} \ar[r]^-{\varphi_r^{-1}}_{\simeq} & {J_r}_{K_r'} 
			\ar[r]^-{w_r}_{\simeq} & {{J_r}_{K_r'}} 
		}
	\end{equation*}
	is $\H_r^*(\Z)$-equivariant. 
	Passing to the induced map on $p$-divisible groups and applying ${e^*}'$, we
	obtain from Proposition \ref{GoodRednProp} an $\H_r^*$-equivariant isomorphism 
	of $p$-divisible groups $\psi_r: {G_r'}_{K_r'} \simeq {G_r}_{K_r'}$. As
	\begin{equation*}
		\xymatrix@C=35pt{
			{{J_r}_{K_r'}} \ar[r]^-{\langle \chi(\gamma)\rangle \langle a\rangle_N w_r}\ar[d]_-{1\times \gamma} & 
			{{J_r}_{K_r'}}\ar[d]^-{1\times \gamma} \\
			{({J_r}_{K_r'})_{\gamma}} \ar[r]_-{\gamma^*(w_r)} & {({J_r}_{K_r'})_{\gamma}}	
		}
	\end{equation*}
	commutes for all $\gamma\in \Gal(K_r'/K_0)$ \cite[Proposition 2.2.6]{CaisHida1}, 
	the $K_r'$-isomorphism $\psi_r$ uniquely descends to an
	$\H_r^*$-equivariant isomorphism (\ref{GrprimeGr})
	of $p$-divisible groups over $\Q_p$.
	By Tate's Theorem, this identification 
	uniquely extends to an isomorphism of the corresponding objects in $\pdiv_{R_r}^{\Gamma}$.
	The asserted compatibility with change in $r$ boils down to the commutativity of the diagrams
	\begin{equation*}
	\begin{gathered}
		\xymatrix{
 {{e^*}'J_s[p^{\infty}]^{\vee}} \ar[r]^-{({U_p^*}^{s})^{\vee}} \ar[d]_-{{({U_p^*}^{-1}\Alb(\ps))^{\vee}}^{r-s}}&
 {{e^*}'J_s[p^{\infty}]^{\vee}} \ar[d]^-{{\Alb(\ps)^{\vee}}^{r-s}} \\
{{e^*}'J_r[p^{\infty}]^{\vee}} \ar[r]_-{({U_p^*}^{r})^{\vee}} & {{e^*}'J_r[p^{\infty}]^{\vee}} \\
		}
		\quad\raisebox{-22pt}{and}\quad
		\xymatrix{
			 {{J_s}_{K_r'}^{\vee}} \ar[r]^-{\varphi_s^{-1}} \ar[d]_-{{\Alb(\ps)^{\vee}}^{r-s}}
			  & {J_s}_{K_r'} \ar[r]^-{w_s} \ar[d]|-{\Pic^0(\ps)^{r-s}} & {{J_s}_{K_r'}}\ar[d]^-{\Pic^0(\pr)^{r-s}}\\
			{{J_r}_{K_r'}^{\vee}} \ar[r]_-{\varphi_r^{-1}}
			 & {J_r}_{K_r'} \ar[r]_-{w_r} & {{J_r}_{K_r'}} 
		}
	\end{gathered}
	\end{equation*}
	for all $s\le r$.  The commutativity of the first diagram is clear, while that of the second follows
	from \cite[Proposition 2.2.6]{CaisHida1} and the fact that 
	for {\em any} finite morphism $f:Y\rightarrow X$ of smooth curves over a field $K$, 
	one has $\varphi_Y\circ \Pic^0(f)=\Alb(f)^{\vee}\circ \varphi_X$, where 
	$\varphi_{\star}:J_{\star}\rightarrow J_{\star}^{\vee}$ is the canonical
	principal polarization on Jacobians for $\star=X,Y$ (see, for example, the proof of Lemma 5.5 in \cite{CaisNeron}).
\end{proof}

We now wish to study the special fiber of $\G_r$, and relate it to the
special fibers of the integral models of modular curves studied in
\cite[2.2]{CaisHida1}.  To that end, let $\X_r$ be 
the Katz--Mazur integral model of $X_r$ over $R_r$
defined in \cite[2.2.3]{CaisHida1}; it is a regular scheme that is proper
and flat of pure relative dimension 1 over $\Spec R_r$
with smooth generic fiber naturally isomorphic to ${X_r}_{K_r}$.
According to \cite[Proposition 2.2.10]{CaisHida1}, 
the special fiber $\o{\X}_r:=\X_r\times_{R_r} \F_p$
is the ``disjoint union with crossings at the supersingular points"
\cite[13.1.5]{KM} of smooth and proper Igusa curves 
$I_{(a,b,u)}:=\Ig_{\max(a,b)}$ indexed by triples $(a,b,u)$
with $a,b$ running over nonnegative integers that sum to $r$
and $u\in (\Z/p^{\min(a,b)}\Z)^{\times}$; in particular, 
$\o{\X}_r$ is geometrically {\em reduced}.  
We write $\nor{\o{\X}}_r$ for the normalization of $\o{\X}_r$,
which is a disjoint union of Igusa curves $I_{(a,b,u)}$.
The canonical semilinear action of $\Gamma$ on $\X_r$ that encodes
the descent data of the generic fiber to $\Q_p$ \cite[2.2.3]{CaisHida1} induces,
by base change, an $\F_p$-linear ``geometric inertia action" of $\Gamma$
on $\nor{\o{\X}}_r$; in this way the $p$-divisible group
$\Pic^0_{\nor{\o{\X}}_r/\F_p}[p^{\infty}]$ 
of the Jacobian of 
$\nor{\o{\X}}_r$ over $\F_p$ is equipped
with an action of $\Gamma$ over $\F_p$ and (via the Hecke
correspondences \cite[2.2.21]{CaisHida1})
canonical actions of 
$\H_r$ and $\H_r^*$.

\begin{definition}\label{pDivGpSpecial}
	Define $\Sigma_r:={e_r^*}'\Pic^0_{\nor{\o{\X}}_r/\F_p}[p^{\infty}]$,
	equipped with the induced actions of $\H_r^*$ and $\Gamma$. 
\end{definition}

Since $\X_r$ is regular, and proper flat over $R_r$ with (geometrically) reduced special fiber,
$\Pic^0_{\X_r/R_r}$ is a smooth $R_r$-scheme by \S8.4 Proposition 2 and \S9.4 Theorem 2 of \cite{BLR}.
By the N\'eron mapping property, we thus have a natural mapping $\Pic^0_{\X_r/R_r}\rightarrow \J_r^0$
that recovers the canonical identification on generic fibers, and is in fact an isomorphism
by \cite[\S9.7, Theorem 1]{BLR}.  Composing with the map $\alpha_r^*:\J_r\rightarrow \B_r^*$ 
and passing to special fibers
yields a homomorphism of smooth commutative algebraic groups over $\F_p$
\begin{equation}
	\xymatrix{
 		{\Pic^0_{\o{\X}_r/\F_p}} \ar[r]^-{\simeq} & {\o{\J}_r^0} \ar[r] & {\o{\B}^*_r}
	}\label{PicToB}
\end{equation}
Due to \cite[\S9.3, Corollary 11]{BLR},
the normalization map $\nor{\o{\X}}_r\rightarrow \o{\X}$ induces a surjective homomorphism 
$\Pic^0_{\o{\X}_r/\F_p}\rightarrow {\Pic^0_{\nor{\o{\X}}_r/\F_p}}$
with kernel that is a smooth, connected {\em linear} algebraic group over $\F_p$.
As any homomorphism from an affine group variety to an abelian variety is zero, 
we conclude that (\ref{PicToB}) uniquely factors through this quotient, and we obtain
a natural map of abelian varieties:
\begin{equation}
	\xymatrix{
		{\Pic^0_{\nor{\o{\X}}_r/\F_p}} \ar[r] & {\o{\B}_r^*}
		}\label{AbVarMaps}
\end{equation}
that is necessarily equivariant for the actions of $\H_r^*(\Z)$ and $\Gamma$.
The following Proposition relates the special fiber of $\G_r$ to the 
$p$-divisible group $\Sigma_r$ of Definition \ref{pDivGpSpecial}, and will allow us
in Corollary \ref{SpecialFiberOrdinary1}
to give an explicit description of the special fiber
of $\G_r$.

\begin{proposition}\label{SpecialFiberDescr}
	The mapping $(\ref{AbVarMaps})$ induces an isomorphism
	of $p$-divisible groups over $\F_p$
	\begin{equation}
		\o{\G}_r := {{e^*}'\o{\B}_r^*[p^{\infty}]} \simeq 
		{{e^*}'\Pic^0_{\nor{\o{\X}}_r/\F_p}[p^{\infty}]}=:\Sigma_r
		\label{OnTheNose}
	\end{equation}
	that is $\H_r^*$ and $\Gamma$-equivariant and compatible with change in $r$ 
	via the maps $\pr_{r,s}$ on $\o{\G}_r$ and the maps $\Pic^0(\pr)^{r-s}$ on $\Sigma_r$.
\end{proposition}

\begin{proof}
	The diagram (\ref{UFactorDiag}) induces a corresponding diagram of
	N\'eron models over $R_r$ and hence of special fibers over $\F_p$.
	Arguing as above, we obtain a commutative diagram of abelian
	varieties
	\begin{equation}
	\begin{gathered}
		\xymatrix@C=30pt@R=35pt{
			{\Pic^0_{\nor{\o{\X}}_r/\F_p}}\ar[d]_-{U_r^*} 
			\ar[r]^-{\o{\alpha}^*_r} & {\o{\B}_r^*} \ar[dl]^-{W_r^*} \ar[d]^-{U_r^*} \\
		 	{\Pic^0_{\nor{\o{\X}}_r/\F_p}}\ar[r]_-{\o{\alpha}^*_r} & {\o{\B}_r^*}
		}\label{UFactorDiagmodp}
	\end{gathered}
	\end{equation}
	 over $\F_p$.  The proof of \ref{GoodRednProp} now goes through {\em mutatis mutandis} to give the claimed
	 isomorphism (\ref{OnTheNose}).  
\end{proof}

In \cite[\S2.5]{CaisHida1}, we analyzed the the structure of the de Rham cohomology
of the smooth and proper curve $\nor{\o{\X}}_r$ over $\F_p$; we now apply this
analysis and Oda's description (Proposition \ref{OdaDieudonne}) of 
Dieudonn\'e modules in terms of de Rham cohomology to understand the structure of 
$\Sigma_r$.  For each $r$, as in \cite[Remark 2.2.12]{CaisHida1} 
we write $I_r^{\infty}:=I_{(r,0,1)}$ and $I_r^0:=I_{(0,r,1)}$
for the two ``good" irreducible components of $\o{\X}_r$; 
by \cite[Proposition 2.5.6]{CaisHida1}, the ordinary part of the
de Rham cohomology of $\nor{\o{\X}}_r$ is entirely captured
by the de Rham cohomology of these two good components. 
Writing $i_r^{\star}:I_r^{\star}\hookrightarrow \nor{\o{\X}}_r$
for the canonical closed immersions, we reinterpret this fact
in the language of Dieudonn\'e modules:

\begin{proposition}\label{GisOrdinary}
	For each $r$, 
	there is a natural isomorphism of $A:=\Z_p[F,V]$-modules
	\begin{equation}
		\D(\Sigma_r)_{\F_p} \simeq {e_r^*}'H^1_{\dR}(\nor{\o{\X}}_r/\F_p)\simeq 
		f'H^0(I_r^{\infty},\Omega^1)^{V_{\ord}}\oplus f'H^1(I_r^0,\O)^{F_{\ord}}.\label{DieudonneDesc}
	\end{equation}
	which is compatible with $\H_r^*$, $\Gamma$, and change in $r$ and which
	carries $\D(\Sigma_r^{\mult})_{\F_p}$ $($respectively $\D(\Sigma_r^{\et})_{\F_p}$$)$ isomorphically 
	onto $f'H^0(I_r^{0},\Omega^1)^{V_{\ord}}$ $($respectively
	$f'H^1(I_r^{\infty},\O)^{F_{\ord}}$$)$.  In particular, 
	$\Sigma_r$ is ordinary.
\end{proposition}

\begin{proof}
	The identifications of \cite[Proposition 2.5.6]{CaisHida1}
	are induced by the closed immersions $i_r^{\star}$
	and are therefore compatible with the natural actions of Frobenius and the Cartier operator.
	The isomorphism (\ref{DieudonneDesc}) is then an immediate consequence of
	Proposition \ref{OdaDieudonne} and \cite[Proposition 2.5.6]{CaisHida1}.  Since this isomorphism is
	compaible with $F$ and $V$, we have 
	\begin{subequations}
	\begin{equation}  
		\D(\Sigma_r^{\mult})_{\F_p}   
		\simeq \D(\Sigma_r)_{\F_p}^{V_{\ord}} 
		\simeq f'H^0(I_r^{0},\Omega^1)^{V_{\ord}}
	\end{equation}
	and
	\begin{equation}
		\D(\Sigma_r^{\et})\otimes_{\Z_p}\F_p 
		\simeq \D(\Sigma_r)_{\F_p}^{F_{\ord}} 
		\simeq f'H^1(I_r^{\infty},\O)^{F_{\ord}}
	\end{equation}
	\end{subequations}
	and we conclude that the canonical inclusion 
	$\D(\Sigma_r^{\mult})_{\Z_p}\oplus\D(\Sigma_r^{\et})_{\Z_p}\hookrightarrow \D(\Sigma_r)_{\Z_p}$
	is surjective, whence $\Sigma_r$ is ordinary by Dieudonn\'e theory.
\end{proof}

With Proposition \ref{GisOrdinary} as a starting point, 
we can now completely describe the structure of $\Sigma_r$
in terms of the two good components $I_r^{\star}$.
Since $\nor{\o{\X}}_r$
is the disjoint union of proper smooth and irreducible Igusa curves $I_{(a,b,u)}$,
we have
a canonical identification of abelian varieties over $\F_p$
\begin{equation}
	\Pic^0_{\nor{\o{\X}}_r/\F_p} = \prod_{(a,b,u)} \Pic^0_{I_{(a,b,u)}/\F_p}.
	\label{Pic0Iden}
\end{equation}
For $\star=0,\infty$ let us write $j_r^{\star}:=\Pic^0_{I_r^{\star}/\F_p}$
for the Jacobian of $I_r^{\star}$ over $\F_p$.
The canonical closed immersions 
$i_r^{\star}:I_r^{\star}\hookrightarrow \nor{\o{\X}}_r$ yield (by Picard and Albanese functoriality) 
homomorphisms of abelian varieties over $\F_p$
\begin{equation}
	\xymatrix{
		{\Alb(i_r^{\star}):j_r^{\star}} \ar[r] & {\Pic^0_{\nor{\o{\X}}_r/\F_p}}
		}
		\quad\text{and}\quad
		\xymatrix{
		{\Pic^0(i_r^{\star}):\Pic^0_{\nor{\o{\X}}_r/\F_p}} \ar[r] & {j_r^{\star}}
		}.\label{AlbPicIncl}
\end{equation}
Via the identification (\ref{Pic0Iden}), we know that $j_r^{\star}$ is a direct factor
of $\Pic^0_{\nor{\o{\X}}_r/\F_p}$; in these terms $\Alb(i_r^{\star})$ is the unique mapping
which projects to the identity on $j_r^{\star}$ and to the zero map on all other factors,
while $\Pic^0(i_r^{\star})$ is simply projection onto the factor $j_r^{\star}$.
As $\Sigma_r$ is a direct factor of ${f' \Pic^0_{\nor{\o{\X}}_r/\F_p}[p^{\infty}]}$,
these mappings induce homomorphisms of $p$-divisible groups over $\F_p$
\begin{subequations}
\begin{equation}
	\xymatrix@C=35pt{
		{f'j_r^{0}[p^{\infty}]^{\mult}} \ar[r]^-{\Alb(i_r^{0})} & 
		{f' \Pic^0_{\nor{\o{\X}}_r/\F_p}[p^{\infty}]^{\mult}} \ar[r]^-{\proj} & 
		{\Sigma_r^{\mult}}
		}\label{Alb0}
\end{equation}		
\begin{equation}
	\xymatrix@C=35pt{
		{\Sigma_r^{\et}} \ar[r]^-{\incl} &
		{f' \Pic^0_{\nor{\o{\X}}_r/\F_p}[p^{\infty}]^{\et}} \ar[r]^-{\Pic^0(i_r^{\infty})} & 
		{f'j_r^{\infty}[p^{\infty}]^{\et}}
		}\label{Picinfty}	
\end{equation}
\end{subequations}
which we (somewhat abusively) again denote by $\Alb(i_r^{0})$ and $\Pic^0(i_r^{\infty})$, respectively.
The following is a sharpening of \cite[Chapter 3, \S3, Proposition 3]{MW-Iwasawa} 
(see also \cite[Proposition 3.2]{Tilouine}):

\begin{proposition}\label{MWSharpening}
	The mappings $(\ref{Alb0})$ and $(\ref{Picinfty})$ are isomorphisms.  They induce 
	a canonical split short exact sequences of $p$-divisible groups over $\F_p$
	\begin{equation}
		\xymatrix@C=45pt{
			0 \ar[r] & {f'j^0_r[p^{\infty}]^{\mult}} \ar[r]^-{\Alb(i_r^0)\circ V^r} &
			{\Sigma_r} \ar[r]^-{\Pic^0(i_r^{\infty})} & {f'j^{\infty}_r[p^{\infty}]^{\et}} \ar[r] & 0
		}\label{pDivUpic}
	\end{equation}
	which is:
	\begin{enumerate}
		\item $\Gamma$-equivariant for the geometric inertia action on $\Sigma_r$, the trivial
		action on $f'j_r^{\infty}[p^{\infty}]^{\et}$, and the action via $\langle \chi(\cdot) \rangle^{-1}$
		on $f'j_r^{0}[p^{\infty}]^{\mult}$.  \label{GammaCompatProp}
		
		\item $\H_r^*$-equivariant with $U_p^*$ acting on $f'j_r^{\infty}[p^{\infty}]^{\et}$
		as $F$ and on $f'j_r^{0}[p^{\infty}]^{\mult}$ as $\langle p\rangle_N V$.
		
		\item	Compatible with change in $r$ via the mappings $\Pic^0(\pr)$ on $j_r^{\star}$ and $\Sigma_r$.
		\label{ChangerProp}
\end{enumerate}	
\end{proposition}

\begin{proof}
	It is clearly enough to prove that the sequence (\ref{pDivUpic}) induced by $(\ref{Alb0})$ and $(\ref{Picinfty})$
	is exact.
	Since the contravariant Dieudonn\'e module functor from the category of $p$-divisible groups
	over $\F_p$ to the category of $A$-modules which are $\Z_p$ finite and free is an exact 
	anti-equivalence, it suffices to prove such exactness
	after applying $\D(\cdot)_{\Z_p}$.  
	As the resulting sequence consists of finite free $\Z_p$-modules, 
	exactness may be checked modulo $p$ where it follows immediately from 
	Proposition \ref{GisOrdinary} by using \cite[Proposition 2.5.6]{CaisHida1}.
	The claimed compatibility with $\Gamma$, $\H_r^*$, and change in $r$ follows easily
	from Propositions 2.2.14, 2.2.20 and 2.2.13 of \cite{CaisHida1}, respectively.
\end{proof}

\begin{remark}
	It is possible to give a short proof of Proposition \ref{MWSharpening} 
	along the lines of \cite{MW-Iwasawa} or \cite{Tilouine} by using
	\cite[Proposition 2.2.20]{CaisHida1} directly.
	We stress, however, that our approach via Dieudonn\'e modules gives more refined information,
	most notably that the Dieudonn\'e module of $\Sigma_r[p]$ is free as an $\F_p[\Delta/\Delta_r]$-module.
	This fact will be crucial in our later arguments. 
\end{remark}

Together, Proposotions \ref{SpecialFiberDescr} and \ref{MWSharpening} give the
desired description of the special fiber of $\G_r$
({\em cf}. \S3 and \S4, Proposition 1 of \cite{MW-Hida} and pgs. 267--274 of \cite{MW-Iwasawa}):

\begin{corollary}\label{SpecialFiberOrdinary1}
	For each $r$, the $p$-dividible group $\G_r/R_r$ is ordinary, and there is a canonical exact sequence,
	compatible with change in $r$ via $\pr_{r,s}$ on $\o{\G}_r$ and $\Pic^0(\pr)^{r-s}$ on $j_r^{\star}[p^{\infty}]$
	\begin{equation}
 		\xymatrix@C=45pt{
			0 \ar[r] & {f'j_r^{0}[p^{\infty}]^{\mult}} \ar[r]^-{\Alb(i_r^{0})\circ V^r} &
			{\o{\G}_r} \ar[r]^-{\Pic^0(i_r^{\infty})} & {f'j_r^{\infty}[p^{\infty}]^{\et}} \ar[r] & 0  
		}\label{GrSpecialExact}
	\end{equation}
	where $i_r^{\star}:I_r^{\star}\hookrightarrow \nor{\o{\X}}_r$ are the canonical closed immersions
	for $\star=0,\infty$.  Moreover, $(\ref{GrSpecialExact})$ is compatible with the actions of
	$\H^*$ and $\Gamma$, with $U_p^*$ $($respectively $\gamma\in \Gamma$$)$
	acting on $f'j_r^{0}[p^{\infty}]^{\mult}$ as $\langle p\rangle_N V$ 
	$($respectively $\langle \chi(\gamma)\rangle^{-1}$$)$ 
	and on $f'j_r^{\infty}[p^{\infty}]^{\et}$ as $F$ $($respectively $\id$$)$.

\end{corollary}

\subsection{Ordinary families of Dieudonn\'e modules}\label{OrdDieuSection}

Let $\{\G_r/R_r\}_{r\ge 1}$ be the tower of $p$-divisible groups given by Definition \ref{ordpdivdefn}.  
From the canonical morphisms $\rho_{r,s}: \G_{s}\times_{T_{s}} T_r\rightarrow \G_{r}$ we obtain
a map on special fibers $\o{\G}_{s}\rightarrow \o{\G}_r$ over $\F_p$
for each $r\ge s$; applying the contravariant Dieudonn\'e module functor
$\D(\cdot):=\D(\cdot)_{\Z_p}$ yields a projective
system of finite free $\Z_p$-modules $\{\D(\o{\G}_r)\}_r$ with compatible linear endomorphisms $F,V$
satisfying $FV=VF=p$.
 
\begin{definition}\label{DinftyDef}
	We write $\D_{\infty}:=\varprojlim_r \D(\o{\G}_r)$ for the projective limit
	of the system $\{\D(\o{\G}_r)\}_r$.  For $\star\in \{\et,\mult\}$
	we write $\D_{\infty}^{\star}:=\varprojlim_r \D(\o{\G}_r^{\star})$
	for the corresponding projective limit.
\end{definition}

Since $\H_r^*$ acts by endomorphisms on $\o{\G}_r$, compatibly with change in $r$,
we obtain an action of $\H^*$ on $\D_{\infty}$ and on $\D_{\infty}^{\star}$.
Likewise, the ``geometric inertia action" of $\Gamma$ on $\o{\G}_r$
gives an action of $\Gamma$ on $\D_{\infty}$ and $\D_{\infty}^{\star}$.
As $\o{\G}_r$ is ordinary thanks to Corollary \ref{SpecialFiberOrdinary1}, applying $\D(\cdot)$
to the (split) connected-\'etale squence of $\o{\G}_r$ gives, for each $r$, 
a functorially split exact sequence 
\begin{equation}
	\xymatrix{
		0 \ar[r] & {\D(\o{\G}_r^{\et})} \ar[r] & {\D(\o{\G}_r)} \ar[r] & 
		{\D(\o{\G}_r^{\mult})} \ar[r] & 0
	}\label{DieudonneFiniteExact}
\end{equation}
with $\Z_p$-linear actions of $\Gamma$, $F$, $V$, and $\H_r^*$.
Since projective limits commute with finite direct sums, we obtain
a split short {\em exact} sequence of $\Lambda$-modules with linear $\H^*$ and $\Gamma$-actions
and commuting linear endomorphisms $F,V$ satisfying $FV=VF=p$:
\begin{equation}
	\xymatrix{
		0 \ar[r] & {\D_{\infty}^{\et}} \ar[r] & {\D_{\infty}} \ar[r] & {\D_{\infty}^{\mult}} \ar[r] & 0
	}.\label{DieudonneInfiniteExact}
\end{equation}

\begin{theorem}\label{MainDieudonne}
	Let 
	$d':=\sum_{k=3}^p  d_k$ for $d_k:=\dim_{\F_p}S_k(\Upgamma_1(N);\F_p)^{\ord}$ the $\F_p$-dimension
	of the space of $p$-ordinary mod $p$ cuspforms of weight $k$ and level $N$.  Then:
	\begin{enumerate} 
		\item $\D_{\infty}$ is a free $\Lambda$-module of rank $2d'$, and $\D_{\infty}^{\star}$
		is free of rank $d'$ over $\Lambda$ for $\star\in \{\et,\mult\}$.  
		\label{MainDieudonne1}
		
		\item For each $r\ge 1$, applying $\otimes_{\Lambda} \Z_p[\Delta/\Delta_r]$ to 
		$(\ref{DieudonneInfiniteExact})$  yields the short exact sequence $(\ref{DieudonneFiniteExact})$, 
		compatibly with $\H^*$, $\Gamma$, $F$ and $V$.
		\label{MainDieudonne2}
		
		\item Under the canonical splitting of $(\ref{DieudonneInfiniteExact})$, $\D_{\infty}^{\et}$ 
		is the maximal subspace of $\D_{\infty}$ on which $F$ acts invertibly, while 
		$\D_{\infty}^{\mult}$ corresponds to the maximal subspace of $\D_{\infty}$ on which $V$ acts 
		invertibly.  
		\label{MainDieudonne3}
		
		\item The Hecke operator $U_p^*$ acts as $F$ on $\D_{\infty}^{\et}$ and as $\langle p\rangle_NV$ on $\D_{\infty}^{\mult}$.
		\label{MainDieudonne4}
	
		\item $\Gamma$ acts trivially on $\D_{\infty}^{\et}$ and via $\langle \chi\rangle^{-1}$
		on $\D_{\infty}^{\mult}$.
		\label{MainDieudonne5}
\end{enumerate}
\end{theorem}

\begin{proof}
	In \cite[\S 3.1]{CaisHida1}, we established a general commutative algebra
	formalism for dealing with projective
	limits of modules and proving structural and control theorems
	as in (\ref{MainDieudonne1}) and (\ref{MainDieudonne2}), respectively.
	In order to apply the main result
	of our formalism to the present situation, we take (in the notation
	of \cite[Lemma 3.1.2]{CaisHida1}) $A_r=\Z_p$, $I_r=(p)$, and
	$M_r$ each one of the terms in (\ref{DieudonneFiniteExact}),
	and we must check that the hypotheses
	\begin{enumerate}
	\setcounter{equation}{3}
		\renewcommand{\theenumi}{\theequation{\rm\alph{enumi}}}
		{\setlength\itemindent{10pt} 
			\item $\o{M}_r:=M_r/pM_r$ is a free $\F_p[\Delta/\Delta_r]$-module of rank 
			$d'$\label{freehyp1}}
		{\setlength\itemindent{10pt} 
		\item For all $s\le r$ the induced transition maps 
		$\xymatrix@1{
				{\overline{\pr}_{r,s}: \o{M}_r}\ar[r] & 
				{\o{M}_{s}}
				}$\label{surjhyp1}}
		are surjective
	\end{enumerate}
	hold.  By Propositions \ref{SpecialFiberDescr} and \ref{GisOrdinary}, 
	there is a natural isomorphism of split short exact sequences
	\begin{equation*}
		\xymatrix{
			0 \ar[r] & {\D(\o{\G}_r^{\et})_{\F_p}} \ar[r]\ar[d]^-{\simeq} & 
			{\D(\o{\G}_r)_{\F_p}} \ar[r] \ar[d]^-{\simeq}& 
			{\D(\o{\G}_r^{\mult})_{\F_p}} \ar[r] \ar[d]^-{\simeq}& 0 \\
			0 \ar[r] & {f'H^1(I_r^0,\O)^{F_{\ord}}}\ar[r] &
			{f'H^0(I_r^{\infty},\Omega^1)^{V_{\ord}}\oplus f'H^1(I_r^0,\O)^{F_{\ord}}} \ar[r] &
			{f'H^0(I_r^{\infty},\Omega^1)^{V_{\ord}}} \ar[r] & 0
		}
	\end{equation*}
	that is compatible with change in $r$ using the trace mappings attached to 
	$\rho:I_r^{\star}\rightarrow I_{s}^{\star}$ and the maps on Dieudonn\'e modules 
	induced by $\o{\rho}_{r,s}:\o{\G}_{s} \rightarrow \o{\G}_r$.
	The hypotheses (\ref{freehyp1}) and (\ref{surjhyp1})
	are therefore satisfied
	thanks to Proposition 2.4.1 
	and Lemma 2.5.5 of \cite{CaisHida1}.  
	It follows that the conclusions of \cite[Lemma 3.1.2]{CaisHida1}
	hold in the present situation, which gives
	(\ref{MainDieudonne1}) and (\ref{MainDieudonne2}). 
	As $F$ (respectively $V$) acts invertibly on $\D(\o{\G}_r^{\et})$ (respectively 
	$\D(\o{\G}_r^{\mult})$) for all $r$, assertion (\ref{MainDieudonne3}) is clear, while
	(\ref{MainDieudonne4}) and (\ref{MainDieudonne5}) follow immediately from 
	Corollary \ref{SpecialFiberOrdinary1}.
\end{proof}

The short exact sequence (\ref{DieudonneInfiniteExact}) is very nearly ``auto dual":

\begin{proposition}\label{DieudonneDuality}		
		There is a canonical isomorphism of short exact sequences of $\Lambda_{R_0'}$-modules
		\begin{equation}
		\begin{gathered}
			\xymatrix{
	0 \ar[r] & {\D_{\infty}^{\et}(\langle \chi \rangle\langle a\rangle_N)_{\Lambda_{R_0'}}} \ar[r]\ar[d]^-{\simeq} & 
	{\D_{\infty}(\langle \chi \rangle\langle a\rangle_N)_{\Lambda_{R_0'}}}\ar[r]\ar[d]^-{\simeq} & 
	{\D_{\infty}^{\mult}(\langle \chi \rangle\langle a\rangle_N)_{\Lambda_{R_0'}}}\ar[r]\ar[d]^-{\simeq} & 0 \\		
		0\ar[r] & {(\D_{\infty}^{\mult})^{\vee}_{\Lambda_{R_0'}}} \ar[r] & 
				{(\D_{\infty})^{\vee}_{\Lambda_{R_0'}}} \ar[r] & 
				{(\D_{\infty}^{\et})^{\vee}_{\Lambda_{R_0'}}}\ar[r] & 0
			}
		\end{gathered}
		\label{DmoduleDuality}	
		\end{equation}
		that is $\H^*$ and $\Gamma\times \Gal(K_0'/K_0)$-equivariant, 
		and intertwines $F$
		$($respectively $V$$)$ on the top row with $V^{\vee}$
		$($respectively $F^{\vee}$$)$ on the bottom.
\end{proposition}

\begin{proof}
	As in the proof of Theorem \ref{MainDieudonne},
	we apply the formalism of \cite[\S3.1]{CaisHida1}.	
	Let us write
	$\pr_{r,s}':\o{\G}_r'\rightarrow \o{\G}_s'$ for the maps on special fibers
	induced by (\ref{pdivTowers}). 
	Thanks to Proposition \ref{GdualTwist}, the definition \ref{ordpdivdefn}  of $\o{\G}_r':=\o{\G}_r^{\vee}$,
	the natural isomorphism $\G_r\times_{R_r} R_r' \simeq \G_r(\langle \chi\rangle\langle a\rangle_N)\times_{R_r} R_r'$,
	and the compatibility of the Dieudonn\'e module functor with duality, there are natural
	isomorphisms of $R_0'$-modules
	\begin{equation}
		\D(\o{\G}_r)(\langle \chi\rangle\langle a\rangle_N) \tens_{\Z_p} R_0' \simeq
		\D(\o{\G_r(\langle \chi\rangle\langle a\rangle_N)})\tens_{\Z_p} R_0' 
		\simeq \D(\o{\G}_r')\tens_{\Z_p} R_0' = \D(\o{\G}_r^{\vee})\tens_{\Z_p} R_0'\simeq 
		(\D(\o{\G}_r))_{R_0'}^{\vee}
		\label{evpairingDieudonne}
	\end{equation}
	that are $\H^*_r$-equivariant, $\Gal(K_r'/K_0)$-compatible 
	for the standard action $\sigma\cdot f (m):=\sigma f(\sigma^{-1}m)$
	on the $R_0'$-linear dual of $\D(\o{\G}_r)\otimes_{\Z_p} R_0'$, 
	and compatible with change in $r$ using $\pr_{r,s}$
	on $\D(\o{\G}_r)$ and $\pr_{r,s}'$ on $\D(\o{\G}_r')$.  We claim that the resulting 
	perfect ``evaluation" pairings 
	\begin{equation}
		\xymatrix{
			{\langle\cdot,\cdot\rangle_r : \D(\o{\G}_r)(\langle \chi\rangle\langle a\rangle_N)\tens_{\Z_p}{R_0'} 
			\times  \D(\o{\G}_r)\tens_{\Z_p}{R_0'}} \ar[r] & {R_0'}
			}\label{pdivSpecialTwistPai}
	\end{equation}
	are compatible with change in $r$ via the maps $\rho_{r,s}$
	and $\rho_{r,s}'$ in the sense of \cite[3.1.4]{CaisHida1}; {\em i.e.} that
	\begin{equation}
		\langle\pr_{r,s}x, \pr_{r,s}'y\rangle_{s} = 
		\sum_{\delta\in \Delta_{s}/\Delta_{r}} \langle x,\delta^{-1} y\rangle_{r}
		\label{pairingchangeinr1}
	\end{equation}
	holds for all $x,$ $y$.
	Indeed, the compatibility of (\ref{evpairingDieudonne}) with change in $r$ 
	and the very definition (\ref{pdivTowers}) of the transition maps $\pr_{r,s}'$
	implies that for $r\ge s$
	\begin{equation}
		\langle \D(\Pic^0(\pr)^{r-s}) x, y\rangle_s = \langle x , \D({U_p^*}^{s-r}\Alb(\ps)^{r-s}) y\rangle_r;
		\label{PairingCompatComp}
	\end{equation}
	on the other hand, it follows from 
	Lemma \ref{MFtraceLem} (using Lemma \ref{LieFactorization}) that 
	we have
	\begin{equation}
		\Pic(\pr)\circ \Alb(\ps) = U_p^* \sum_{\delta\in \Delta_r/\Delta_{r+1}} \langle \delta^{-1}\rangle,
		\label{PicAlbRelation}
	\end{equation}
	in $\End_{\Q_p}(J_{r+1})$, and together (\ref{PairingCompatComp})--(\ref{PicAlbRelation})
	imply the desired compatibility (\ref{pairingchangeinr1}).
	It follows that the hypotheses of \cite[Lemma 3.1.4]{CaisHida1} is verified, and we conclude
	that the pairings (\ref{pdivSpecialTwistPai}) give rise to
	a perfect $\Gal(K_{\infty}'/K_0)$-compatible duality pairing
	$\langle\cdot,\cdot \rangle: \D_{\infty}(\langle \chi\rangle\langle a\rangle_N)
	\otimes_{\Lambda} \Lambda_{R_0'} \times \D_{\infty}\otimes_{\Lambda} \Lambda_{R_0'} \rightarrow \Lambda_{R_0'}$
	with respect to which $T^*$ is self-adjoint for all $T^*\in \H^*$
	as this is true at each finite level $r$ thanks to 
	the $\H_r^*$-compatibility of (\ref{evpairingDieudonne}).
	That the resulting isomorphism (\ref{DmoduleDuality})
	intertwines 
	$F$ with $V^{\vee}$
	and $ V $ with $F^{\vee}$ is an immediate consequence of the compatibility of the Dieudonn\'e module
	functor with duality. 
\end{proof}

We can interpret $\D_{\infty}^{\star}$ in terms of the crystalline cohomology of the
Igusa tower as follows.  Let $I_r^0$ and $I_r^{\infty}$ be the two ``good" 
components of $\o{\X}_r$ as in the discussion preceding Proposition \ref{GisOrdinary}, 
and form the projective limits
\begin{equation*}
	H^1_{\cris}(I^{\star}) := \varprojlim_{r} H^1_{\cris}(I_r^{\star})
\end{equation*}
for $\star\in \{\infty,0\}$, taken with respect to the trace maps on crystalline
cohomology (see \cite[\Rmnum{7}, \S2.2]{crystal2}) induced by the canonical degeneracy mappings 
$\rho:I_{r}^{\star}\rightarrow I_{s}^{\star}$.  
Then $H^1_{\cris}(I^{\star})$ 
is naturally a $\Lambda$-module (via the diamond operators), equipped with 
a commuting action of $F$ (Frobenius) and $V$ (Verscheibung) satisfying $FV=VF=p$.  
Letting $U_p^*$ act as $F$ (respectively $\langle p\rangle_N V$) on $H^1_{\cris}(I^{\star})$ for $\star=\infty$
(respectively $\star=0$) and the Hecke operators outside $p$ (viewed as correspondences
on the Igusa curves) act via pullback and trace at each level $r$, we obtain 
an action of $\H^*$ on $H^1_{\cris}(I^{\star})$.  Finally, we 
let $\Gamma$ act trivially on $H^1_{\cris}(I^{\star})$ for $\star=\infty$
and via $\langle\chi^{-1}\rangle$ for $\star=0$.

\begin{theorem}\label{DieudonneCrystalIgusa}
	There is a canonical $\H^*$ and $\Gamma$-equivariant
	isomorphism of $\Lambda$-modules
	\begin{equation*}
		\D_{\infty} = \D_{\infty}^{\mult}\oplus \D_{\infty}^{\et} \simeq
		f'H^1_{\cris}(I^{0})^{V_{\ord}} \oplus 
		f'H^1_{\cris}(I^{\infty})^{F_{\ord}}
	\end{equation*}
	which respects the given direct sum decompositions and is compatible with $F$ and $V$.
\end{theorem}

\begin{proof}
	From the exact sequence (\ref{GrSpecialExact}), we obtain for each $r$ isomorphisms
	\begin{equation}
		\xymatrix@C=55pt{
			{\D(\o{\G}_r^{\mult})} \ar[r]^-{\simeq}_-{V^r \circ \D(\Alb(i_r^{0}))} &
			 {f'\D(j_r^{0}[p^{\infty}])^{V_{\ord}}}
		}\qquad\text{and}\qquad
		\xymatrix@C=55pt{
			{f'\D(j_r^{\infty}[p^{\infty}])^{F_{\ord}}} \ar[r]^-{\simeq}_-{\D(\Pic^0(i_r^{\infty}))} &
			{\D(\o{\G}_r^{\et})}
		}\label{IgusaInterpretation}
	\end{equation}
	that are $\H^*$ and $\Gamma$-equivariant (with respect to the actions
	specified in Corollary \ref{SpecialFiberOrdinary1}), and compatible with change in $r$
	via the mappings $\D(\pr_{r,s})$ on $\D(\o{\G}_r^{\star})$ and $\D(\pr)$
	on $\D(j_r^{\star}[p^{\infty}])$.  On the other hand, for {\em any} smooth and proper curve
	$X$ over a perfect field $k$ of characteristic $p$, 
	thanks to \cite{MM} and \cite[\Rmnum{2}, \S3 C Remarque 3.11.2]{IllusiedR}
	there are natural isomorphisms
	of $W(k)[F,V]$-modules  
	\begin{equation}
		\D(J_X[p^{\infty}]) \simeq H^1_{\cris}(J_X/W(k)) \simeq H^1_{\cris}(X/W(k))\label{MMIllusie}
	\end{equation}
	that for any finite map of smooth proper curves $f:Y\rightarrow X$ over $k$
	intertwine $\D(\Pic(f))$ and $\D(\Alb(f))$ with trace and pullback by $f$ on crystalline cohomology,
	respectively.  Applying this to $X=I_r^{\star}$ for $\star=0,\infty$, appealing to 
	the identifications (\ref{IgusaInterpretation}), and passing to inverse limits completes the proof.
\end{proof}

We now wish to relate the slope filtration (\ref{DieudonneInfiniteExact}) 
to the Hodge filtration (\ref{orddRseq})
of our ordinary $\Lambda$-adic de Rham cohomology studied in \cite{CaisHida1}.
Applying the idempotent $f'$ of (\ref{projaway}) to (\ref{orddRseq})
yields a short exact sequence of free $\Lambda_{R_{\infty}}$-modules with
semilinear $\Gamma$-action and commuting action of $\H^*$:
\begin{equation}
	\xymatrix{
		0 \ar[r] & {{e^*}'H^0(\omega)} \ar[r] & {{e^*}'H^1_{\dR}} \ar[r] & {{e^*}'H^1(\O)} \ar[r] & 0
	}.\label{LambdaHodgeFilnomup}
\end{equation}
The key to relating (\ref{LambdaHodgeFilnomup}) to the slope filtration (\ref{DieudonneInfiniteExact}) 
is the following comparison isomorphism:

\begin{proposition}\label{KeyComparison}
	For each positive integer $r$, there is a natural isomorphism of short exact sequences
	\begin{equation}
	\begin{gathered}
		\xymatrix{
			0\ar[r] & {\omega_{\G_r}} \ar[r]\ar[d]_-{\simeq} & {\D(\G_{r,0})_{R_r}} \ar[r]\ar[d]^-{\simeq} & 
			{\Lie(\Dual{\G}_r)} \ar[r]\ar[d]^-{\simeq} & 0 \\
			0\ar[r] & {{e^*}'H^0(\omega_r)} \ar[r] & {{e^*}'H^1_{\dR,r}} \ar[r] & {{e^*}'H^1(\O_r)} \ar[r] & 0
		}
	\end{gathered}\label{CollectedComparisonIsom}
	\end{equation}
	that is compatible with $\H_r^*$, $\Gamma$, and change in $r$ using
	the mappings $(\ref{pdivTowers})$ on the top row and the maps $\pr_*$ on the bottom.
	Here, the bottom row---with obvious abbreviated notation---is obtained from $(\ref{finiteleveldRseq})$ 
	by applying ${e^*}'$
	and the top row is the Hodge filtration of $\D(\G_{r,0})_{R_r}$ given by
	Proposition $\ref{BTgroupUnivExt}$.  
\end{proposition}

\begin{proof}
	Let $\alpha_r^*: J_r\twoheadrightarrow B_r^*$ be the map of Definition \ref{BalphDef}.
	We claim that $\alpha_r^*$ induces a canonical isomorphism of short exact sequences of free 
	$R_r$-modules
	\begin{equation}
	\begin{gathered}
		\xymatrix{
			0 \ar[r] & {\omega_{\G_r}}\ar[d]_-{\simeq} \ar[r] & {\D(\G_{r,0})_{R_r}} \ar[d]_-{\simeq}\ar[r] & 
			{\Lie(\G_r^t)} \ar[d]^-{\simeq}\ar[r] & 0 \\
			0 \ar[r] & {{e^*}'\omega_{\J_r}} \ar[r] & {{e^*}'\Lie\scrExtrig(\J_r,\Gm)} \ar[r] & 
			{{e^*}'\Lie({\J_r^t}^0)} \ar[r] & 0
		}
	\end{gathered}\label{HodgeToExtrigMap}
	\end{equation}
	that is $\H_r^*$ and $\Gamma$-equivariant and compatible with change in $r$ using
	the map on N\'eron models induced by $\Pic^0(\pr)$ and the maps (\ref{pdivTowers})
	on $\G_r$.
	Granting this claim, the proposition then follows immediately from Proposition \ref{intcompare}.

	To prove our claim, we introduce the following notation:
	set $V:=\Spec(R_r)$, and for $n\ge 1$ put $V_n:=\Spec(R_r/p^nR_r)$. For any scheme 
	(or $p$-divisible group) $T$ over $V$, we put $T_n:=T\times_V V_n$.
	If $\A$ is a N\'eron model over $V$,
	we will write $H(\A)$ for the short exact sequence of free $R_r$-modules obtained by
	applying $\Lie$ to the canonical extension (\ref{NeronCanExt}) of $\Dual{\A}^0$.
	If $G$ is a $p$-divisible group over $V$, we similalry
	write $H(G_n)$ for the short exact sequence of Lie algebras associated to the universal extension
	of $G_n^t$ by a vector group over $V_n$ (see Theorem \ref{UniExtCompat}, (\ref{UniExtCompat2})).  
	If $\A$ is an abelian scheme over $V$ 
	then we have natural and compatible (with change in $n$) isomorphisms
	\begin{equation}
		H(\A_n[p^{\infty}])\simeq H(\A_n)\simeq H(\A)/p^n,\label{AbSchpDiv}
	\end{equation}
	thanks to Theorem \ref{UniExtCompat}, (\ref{UniExtCompat3}) and (\ref{UniExtCompat1}); in particular, this
	justifies our slight abuse of notation.  
	
	Applying the contravariant functor ${e^*}'H(\cdot)$ to the diagram of N\'eron models over $V$ 
	induced by (\ref{UFactorDiag}) yields a commutative diagram of short exact sequences of 
	free $R_r$-modules
	\begin{equation}
	\begin{gathered}
		\xymatrix{
			{{e^*}'H(\J_r)} & {{e^*}'H(\B_r^*)}\ar[l] \\
			{{e^*}'H(\J_r)} \ar[u]^-{U_r^*}\ar[ur] & {{e^*}'H(\B_r)}\ar[u]_-{U_r^*}\ar[l] 
			}
	\end{gathered}		
	\end{equation}
	in which both vertical arrows are isomorphisms by definition of ${e^*}'$.  As in the proofs of Propositions
	\ref{GoodRednProp} and \ref{SpecialFiberDescr}, it follows that
	the horizontal maps must be isomorphisms as well:
	\begin{equation}
		{e^*}'H(\J_r)\simeq {e^*}'H(\B_r^*)
		\label{alphaIdenOrd}
	\end{equation}
	Since these isomorphisms are induced via the N\'eron mapping property and the functoriality
	of $H(\cdot)$ by the $\H_r^*(\Z)$-equivariant map $\alpha_r^*:J_r\twoheadrightarrow B_r^*$,
	they are themselves $\H_r^*$-equivariant.  Similarly, since $\alpha_r^*$ is defined 
	over $\Q$ and compatible with change in $r$ as in Lemma \ref{Btower}, the isomorphism
	(\ref{alphaIdenOrd}) is compatible with the given actions of $\Gamma$ (arising via the N\'eron
	mapping property from the semilinear action of $\Gamma$ over $K_r$ giving the descent
	data of ${J_r}_{K_r}$ and ${B_r}_{K_r}$ to $\Q_p$) and change in $r$.
	Reducing (\ref{alphaIdenOrd}) modulo $p^n$ and using the canonical isomorphism (\ref{AbSchpDiv}) yields
	the identifications
	\begin{equation}
		{e^*}'H(\J_r)/p^n\simeq {e^*}'H(\B_r^*)/p^n \simeq {e^*}'H(\B_{r,n}^*[p^{\infty}]) 
		\simeq H({e^*}'\B_{r,n}^*[p^{\infty}]) =: H(\G_{r,n})\label{ModPowersIsom}
	\end{equation}
	which are clearly compatible with change in $n$, and which are easily checked
	(using the naturality of (\ref{AbSchpDiv}) and our remarks above) to be
	$\H_r^*$ and $\Gamma$-equivariant, and compatible with change in $r$.
	Since the surjection $R_r\twoheadrightarrow R_r/pR_r$ is a PD-thickening,
	passing to inverse limits (with respect to $n$) on (\ref{ModPowersIsom}) and using 
	Proposition \ref{BTgroupUnivExt} now completes the proof.	
\end{proof}

\begin{corollary}\label{RelationToHodgeCor}
	Let $r$ be a positive integer. Then the short exact sequence of free $R_r$-modules
	\begin{equation}
		\xymatrix{
			0\ar[r] & {{e^*}'H^0(\omega_r)} \ar[r] & {{e^*}'H^1_{\dR,r}} \ar[r] & {{e^*}'H^1(\O_r)} \ar[r] & 0
		}\label{TrivialEigenHodge}
	\end{equation}
	is functorially split; in particular, 
	it is split compatibly with the actions of $\Gamma$ and $\H_r^*$.
	Moreover, $(\ref{TrivialEigenHodge})$ admits a functorial descent to $\Z_p$:
	there is a natural isomorphism of split short exact sequences 
	\begin{equation}
	\begin{gathered}
			\xymatrix{
			0\ar[r] & {{e^*}'H^0(\omega_r)} \ar[r]\ar[d]_-{\simeq} & 
			{{e^*}'H^1_{\dR,r}} \ar[r]\ar[d]^-{\simeq} & {{e^*}'H^1(\O_r)} \ar[r]\ar[d]^-{\simeq} & 0\\
			0 \ar[r] & {\D(\o{\G}_r^{\mult})\tens_{\Z_p} R_r} \ar[r] &
			{\D(\o{\G}_r)\tens_{\Z_p} R_r}\ar[r] &
			{\D(\o{\G}_r^{et})\tens_{\Z_p} R_r} \ar[r] & 0
			}
	\end{gathered}\label{DescentZp}		
	\end{equation}
	that is $\H^*$ and $\Gamma$ equivariant, with
	$\Gamma$ acting trivially on $\o{\G}_r^{\et}$ and through $\langle \chi\rangle^{-1}$ on $\o{\G}_r^{\mult}$.
	The identification $\ref{DescentZp}$ is compatible with change in $r$ using the maps $\pr_*$ on the top
	row and the maps induced by 
	\begin{equation*}
			\xymatrix@C=35pt{
				{\o{\G}_r=\o{\G}_r^{\mult} \times \o{\G}_r^{\et}} \ar[r]^{V^{-1}\times F} &
				 {\o{\G}_r^{\mult} \times \o{\G}_r^{\et}=\o{\G}_r} \ar[r]^-{\o{\rho}} & 
				 {\o{\G}_{r+1}} 
			}
	\end{equation*}	
	on the bottom row.
\end{corollary}

\begin{proof}
	Consider the isomorphism (\ref{CollectedComparisonIsom}) of Proposition \ref{KeyComparison}.
	As  $\G_r$ is an ordinary $p$-divisible group by Corollary \ref{SpecialFiberOrdinary1},
	the top row of (\ref{CollectedComparisonIsom}) is functorially split
	by Lemma \ref{HodgeFilOrdProps},  and this gives our first assertion.
	Composing the inverse of (\ref{CollectedComparisonIsom})
	with the isomorphism (\ref{DescentToWIsom}) of	Lemma \ref{HodgeFilOrdProps} gives
	the claimed identification (\ref{DescentZp}).
	That this isomorphism is compatible with change in $r$ via the specified maps
	follows easily from the construction of (\ref{DescentToWIsom}) via
	(\ref{TwistyDieuIsoms}).
\end{proof}

We can now prove Theorem \ref{dRtoDieudonne}.  Let us recall the statement:

\begin{theorem}\label{dRtoDieudonneInfty}
	There is a canonical isomorphism of 
	short exact sequences of finite free $\Lambda_{R_{\infty}}$-modules 
	\begin{equation}
	\begin{gathered}
		\xymatrix{
		0 \ar[r] & {{e^*}'H^0(\omega)} \ar[r]\ar[d]^-{\simeq} & 
		{{e^*}'H^1_{\dR}} \ar[r]\ar[d]^-{\simeq} & {{e^*}'H^1(\O)} \ar[r]\ar[d]^-{\simeq} & 0 \\
		0 \ar[r] & {\D_{\infty}^{\mult}\tens_{\Lambda} \Lambda_{R_{\infty}}} \ar[r] &
		{\D_{\infty}\tens_{\Lambda} \Lambda_{R_{\infty}}} \ar[r] &
		{\D_{\infty}^{\et}\tens_{\Lambda} \Lambda_{R_{\infty}}} \ar[r] & 0
		}
	\end{gathered}
	\end{equation}
	that is $\Gamma$ and $\H^*$-equivariant.  
	Here, the mappings on bottom row are the canonical inclusion and projection morphisms
	corresponding to the direct sum decomposition $\D_{\infty}=\D_{\infty}^{\mult}\oplus \D_{\infty}^{\et}$.
	In particular, the Hodge filtration exact sequence $(\ref{LambdaHodgeFilnomup})$ is canonically 
	split, and admits a canonical descent to $\Lambda$.
\end{theorem}

\begin{proof}
	Applying $\otimes_{R_r} R_{\infty}$ to $(\ref{DescentZp})$ and passing to projective limits 
	yields an isomorphism of split exact sequences
	\begin{equation*}
		\xymatrix{
			0\ar[r] & {{e^*}'H^0(\omega)} \ar[r]\ar[d]_-{\simeq} & 
			{{e^*}'H^1_{\dR}} \ar[r]\ar[d]^-{\simeq} & {{e^*}'H^1(\O)} \ar[r]\ar[d]^-{\simeq} & 0\\
0 \ar[r] & {\varprojlim\limits_{\o{\rho}\circ V^{-1}} \left(\D(\o{\G}_r^{\mult})\tens_{\Z_p} R_{\infty}\right)} \ar[r] &
{\varprojlim\limits_{\o{\rho}\circ (V^{-1}\times F)}\left(\D(\o{\G}_r)\tens_{\Z_p} R_{\infty}\right)}\ar[r] &
			{\varprojlim\limits_{\o{\rho}\circ F}\left(\D(\o{\G}_r^{et})\tens_{\Z_p} R_{\infty}\right)} \ar[r] & 0
			}
		\end{equation*}
	On the other hand, the isomorphisms
$		\xymatrix@1@C=37pt{
			{\o{\G}_r = \o{\G}_r^{\mult}\times \o{\G}_r^{\et} } \ar[r]^-{V^{-r}\times F^{r}} &
			{\o{\G}_r^{\mult}\times \o{\G}_r^{\et} =\o{\G}_r}
		}
$
	induce an isomorphism of projective limits
	\begin{equation*}
		\xymatrix{
			{\varprojlim\limits_{\o{\rho}}\left(\D(\o{\G}_r)\tens_{\Z_p} R_{\infty}\right)}	\ar[r]^-{\simeq} &
			{\varprojlim\limits_{\o{\rho}\circ (V^{-1}\times F)}\left(\D(\o{\G}_r)\tens_{\Z_p} R_{\infty}\right)}
		}
	\end{equation*}
	which is visibly compatible with the the canonical splittings of source and target.  
	The result now follows from \cite[Lemma 3.1.2 (5)]{CaisHida1}
	and the proof of Theorem \ref{MainDieudonne},
	which guarantee that the canonical mapping 
	$\D_{\infty}\otimes_{\Lambda}\Lambda_{R_{\infty}}\rightarrow 
	\varprojlim_{\o{\rho}} (\D(\o{\G}_r)\otimes_{\Z_p}R_{\infty})$
	is an isomorphism respecting the natural splittings.
\end{proof}

As in \cite[\S3.3]{CaisHida1}, for any subfield $K$ of $\c_p$ with ring of integers $R$, we denote by
$eS(N;\Lambda_R)$ the module of ordinary $\Lambda_R$-adic cuspforms of level $N$
in the sense of \cite[2.5.5]{OhtaEichler}.  Following our convention 
above Proposition \ref{GoodRednProp},	
we write $e'S(N;\Lambda_R)$ for the direct summand of $eS(N;\Lambda_R)$
on which $\mu_{p-1}\hookrightarrow \Z_p^{\times}\subseteq \H$ acts 
nontrivially.

\begin{corollary}\label{MFIgusaDieudonne}
	There is a canonical isomorphism of finite free $\Lambda$-modules 
	\begin{equation}
		{e}'S(N;\Lambda) \simeq \D_{\infty}^{\mult}   
		\label{LambdaFormsCrystalline}
	\end{equation}
	that intertwines $T\in \H$ on $e'S(N;\Lambda)$ with $T^*\in \H^*$
	on $\D_{\infty}^{\mult}$, 
	where $U_p^*$ acts on $\D_{\infty}^{\mult}$as $\langle p\rangle_N V$.
\end{corollary}

\begin{proof}
	We claim that there are natural isomorphisms of  finite free $\Lambda_{R_{\infty}}$-modules
	\begin{equation}
		\D_{\infty}^{\mult} \otimes_{\Lambda} \Lambda_{R_{\infty}} \simeq 
		{e^*}'H^0(\omega) \simeq {e}'S(N,\Lambda_{R_{\infty}}) \simeq 
		e'S(N,\Lambda)\otimes_{\Lambda} \Lambda_{R_{\infty}}\label{TakeGammaInvariants}
	\end{equation}		
	and that the resulting composite isomorphism 
	intertwines $T^*\in \H^*$ on $\D_{\infty}^{\mult}$ with $T\in \H$ on 
	$e'S(N,\Lambda)$ and is $\Gamma$-equivariant, with $\gamma\in\Gamma$
	acting as $\langle \chi(\gamma)\rangle^{-1}\otimes \gamma$ on each tensor product.  
	Indeed, the first and second isomorphisms are
	due to Theorem \ref{dRtoDieudonneInfty} and \cite[Corollary 3.3.5]{CaisHida1}, respectively,
	while the final isomorphism is a consequence of the definition 
	of $e'S(N;\Lambda_R)$ and the facts that this $\Lambda_R$-module is free of finite rank
	\cite[Corollary 2.5.4]{OhtaEichler} and specializes as in \cite[2.6.1]{OhtaEichler}.
	Twisting the
	$\Gamma$-action on the source and target of the composite (\ref{TakeGammaInvariants}) by 
	$\langle \chi \rangle$ therefore gives a $\Gamma$-equivariant isomorphism
	\begin{equation}
		\D_{\infty}^{\mult} \otimes_{\Lambda} \Lambda_{R_{\infty}} \simeq 
		S(N,\Lambda)\otimes_{\Lambda} \Lambda_{R_{\infty}}\label{TwistedGammaIsom}
	\end{equation}
	with $\gamma\in \Gamma$ acting as $1\otimes \gamma$ on source and target.  Passing to $\Gamma$-invariants on
	(\ref{TwistedGammaIsom}) yields (\ref{LambdaFormsCrystalline}).	
\end{proof}

\begin{remark}\label{MFIgusaCrystal}
	Via Proposition \ref{DieudonneDuality} and the natural $\Lambda$-adic
	duality between $e\H$ and $eS(N;\Lambda)$ \cite[Theorem 2.5.3]{OhtaEichler}, we obtain
	a canonical $\Gal(K_0'/K_0)$-equivariant isomorphism of $\Lambda_{R_0'}$-modules
	\begin{equation*}
		e'\H\tens_{\Lambda} \Lambda_{R_0'} \simeq \D_{\infty}^{\et}(\langle a\rangle_N)\tens_{\Lambda}{\Lambda_{R_0'}}
	\end{equation*}
	that intertwines $T\otimes 1$ for $T\in \H$ acting on $e'\H$ by multiplication
	with $T^*\otimes 1$, with $U_p^*$ acting on $\D_{\infty}^{\et}(\langle a\rangle_N)$ as $F$.
	From Theorem \ref{DieudonneCrystalIgusa} and Corollary \ref{MFIgusaDieudonne}  
	we then obtain canonical
	isomorphisms 
	\begin{equation*}
		e'S(N;\Lambda)\simeq f'H^1_{\cris}(I^0)^{V_{\ord}}\qquad\text{respectively}\qquad
		e'\H\tens_{\Lambda}\Lambda_{R_0'} \simeq 
		f'H^1_{\cris}(I^{\infty})^{F_{\ord}}(\langle a\rangle_{N})\tens_{\Lambda}\Lambda_{R_0'}
	\end{equation*}
	intertwing $T$ (respectively $T\otimes 1$) with $T^*$ (respectively $T^*\otimes 1$)
	where $U_p^*$ acts on crystalline cohomology as $\langle p\rangle_N V$
	(respectively $F\otimes 1$).  The second of these isomorphisms is moreover $\Gal(K_0'/K_0)$-equivariant.
\end{remark}

In order to relate the slope filtration (\ref{DieudonneInfiniteExact}) of $\D_{\infty}$
to the ordinary filtration of ${e^*}'H^1_{\et}$, we 
require:
\begin{lemma}
	Let $r$ be a positive integer let $G_r={e^*}'J_r[p^{\infty}]$ 
	be the unique $\Q_p$-descent of the generic fiber of $\G_r$, as in Definition $\ref{ordpdivdefn}$.
	There are canonical isomophisms of free $W(\o{\F}_p)$-modules
	\begin{subequations}	
		\begin{equation}
			\D(\o{\G}_r^{\et})\tens_{\Z_p} W(\o{\F}_p) \simeq \Hom_{\Z_p}(T_pG_r^{\et},\Z_p)\tens_{\Z_p} W(\o{\F}_p)
		\label{etalecase}
		\end{equation}
		\begin{equation}
			\D(\o{\G}_r^{\mult})(-1)\tens_{\Z_p} W(\o{\F}_p) \simeq \Hom_{\Z_p}(T_pG_r^{\mult},\Z_p)
			\tens_{\Z_p} W(\o{\F}_p).	
		\label{multcase}	
		\end{equation}
		that are $\H_r^*$-equivariant and $\scrG_{\Q_p}$-compatible for the diagonal action on source and target,
		with $\scrG_{\Q_p}$ acting trivially on $\D(\o{\G}^{\et}_r)$ and via $\chi^{-1}\cdot \langle \chi^{-1}\rangle$
		on $\D(\o{\G}_r^{\mult})(-1):=\D(\o{\G}_r^{\mult})\otimes_{\Z_p} \Z_p(-1)$.  The isomorphism
		$(\ref{etalecase})$ intertwines $F\otimes\sigma$ with $1\otimes \sigma$
		while $(\ref{multcase})$ intertwines $V\otimes\sigma^{-1}$ with $1\otimes\sigma^{-1}$.
	\end{subequations}
\end{lemma}

\begin{proof}
	Let $\G$ be any object of $\pdiv_{R_r}^{\Gamma}$ and write $G$ for the unique
	descent of the generic fiber $\G_{K_r}$ to $\Q_p$.  We recall that the semilinear $\Gamma$-action
	on $\G$ gives the $\Z_p[\scrG_{K_r}]$-module
	$T_p\G:=\Hom_{\O_{\C_p}}(\Q_p/\Z_p,\G_{\O_{\C_p}})$ 
	the natural structure of $\Z_p[\scrG_{\Q_p}]$-module via $g\cdot f:= g^{-1}\circ g^*f\circ g$.
	It is straightforward to check that the natural map $T_p\G\rightarrow T_pG$, which is an isomorphism
	of $\Z_p[\scrG_{K_r}]$-modules by Tate's theorem, is an isomorphism of $\Z_p[\scrG_{\Q_p}]$-modules
	as well.

	For {\em any} \'etale $p$-divisible group $H$ over a perfect field $k$, one has a canonical
	isomorphism of $W(\o{k})$-modules with semilinear $\scrG_k$-action
	\begin{equation*}
		\D(H)\tens_{W(k)} W(\o{k}) \simeq \Hom_{\Z_p}(T_pH,\Z_p)\tens_{\Z_p} W(\o{k})
	\end{equation*}
	that intertwines $F\otimes\sigma$ with $1\otimes\sigma$ and $1\otimes g$
	with $g\otimes g$ for $g\in \scrG_k$; for example, this can be deduced
	by applying \cite[\S4.1 a)]{BBM1} to $H_{\o{k}}$ 
	and using the fact that
	the Dieudonn\'e crystal is compatible with base change.
	In our case, the \'etale $p$-divisible group $\G_r^{\et}$ lifts $\o{\G}_r^{\et}$
	over $R_r$, and we obtain a natural isomorphism of $W(\o{\F}_p)$-modules
	with semilinear $\scrG_{K_r}$-action
	\begin{equation*}
		\D(\o{\G}_r^{\et})\tens_{\Z_p} W(\o{\F}_p) \simeq \Hom_{\Z_p}(T_p\G_r^{\et},\Z_p)\tens_{\Z_p} W(\o{\F}_p).
	\end{equation*}
	By naturality in $\G_r$, this identification respects the semilinear $\Gamma$-actions
	on both sides (which are trivial, as $\Gamma$ acts trivially on $\G_r^{\et}$); as explained
	in our initial remarks, it is precisely this action which allows us to view $T_p\G_r^{\et}$ as a 
	$\Z_p[\scrG_{\Q_p}]$-module, and we deduce (\ref{etalecase}).  The proof of (\ref{multcase})
	is similar, using the natural isomorphism (proved as above) for any multiplicative $p$-divisible group $H/k$
	\begin{equation*}
		\D(H)\tens_{W(k)} W(\o{k}) \simeq T_p\Dual{H}\tens_{\Z_p} W(\o{k}),
	\end{equation*}
	which intertwines $V\otimes\sigma^{-1}$ with $1\otimes\sigma^{-1}$
	and $1\otimes g$ with $g\otimes g$, for $g\in \scrG_k$.
\end{proof}

\begin{proof}[Proof of Theorem $\ref{FiltrationRecover}$ and Corollary $\ref{MWmainThmCor}$]
	For a $p$-divisible group $H$ over a field $K$, we will write $H^1_{\et}(H):=\Hom_{\Z_p}(T_pH,\Z_p)$;
	our notation is justified by the standard fact that, for $J_X$ the Jacobian
	of a curve $X$ over $K$, there is a natural isomorphisms of $\Z_p[\scrG_K]$-modules
	\begin{equation}
		H^1_{\et}(J_X[p^{\infty}]) \simeq H^1_{\et}(X_{\Kbar},\Z_p).\label{etalecohcrvjac}
	\end{equation}
	It follows from (\ref{etalecase})--(\ref{multcase}) and 
	Theorem \ref{MainDieudonne} (\ref{MainDieudonne1})--(\ref{MainDieudonne2})
	that $H^1_{\et}(G_r^{\star})\otimes_{\Z_p} W(\o{\F}_p)$ is a free
	$W(\o{\F}_p)[\Delta/\Delta_r]$-module of rank $d'$
	for $\star\in \{\et,\mult\}$,
	and hence that $H^1_{\et}(G_r^{\star})$ is a free $\Z_p[\Delta/\Delta_r]$-module
	of rank $d'$ by \cite[Lemma 3.1.3]{CaisHida1}.  In a similar manner, using
	the faithful flatness of $W(\o{\F}_p)[\Delta/\Delta_r]$ over $\Z_p[\Delta/\Delta_r]$,
	we deduce that the canonical trace mappings
	\begin{equation}  
		\xymatrix{
			{H^1_{\et}(G_r^{\star})} \ar[r] & {H^1_{\et}(G_{r'}^{\star})}
			}\label{cantraceetale}
	\end{equation}
	are surjective for all $r\ge r'$.
	From the commutative algebra formalism of \cite[Lemma 3.1.2]{CaisHida1}, we conclude that
		$H^1_{\et}(G_{\infty}^{\star}):=\varprojlim_r H^1_{\et}(G_r^{\star})$
	is a free $\Lambda$-module of rank $d'$ and that there are canonical
	isomorphisms of $\Lambda_{W(\o{\F}_p)}$-modules
	\begin{equation*}
		H^1_{\et}(G_{\infty}^{\star})\tens_{\Lambda} \Lambda_{W(\o{\F}_p)} \simeq \varprojlim_r \left(H^1_{\et}(G_r^{\star})\tens_{\Z_p} W(\o{\F}_p)\right)
	\end{equation*}
	for $\star\in \{\et,\mult\}$.  Since we likewise have canonical identifications
	\begin{equation*}
		\D_{\infty}^{\star}\tens_{\Lambda} \Lambda_{W(\o{\F}_p)} \simeq \varprojlim_r \left(\D(G_r^{\star})\tens_{\Z_p} W(\o{\F}_p)\right)
	\end{equation*}
	thanks again to \cite[Lemma 3.1.2]{CaisHida1} and (the proof of) Theorem \ref{MainDieudonne}, 
	passing to inverse limits on (\ref{etalecase})--(\ref{multcase}) gives a canonical
	isomorphism of $\Lambda_{W(\o{\F}_p)}$-modules
	\begin{equation}
		\D_{\infty}^{\star}\tens_{\Lambda} \Lambda_{W(\o{\F}_p)} \simeq
		H^1_{\et}(G_{\infty}^{\star})\tens_{\Lambda} \Lambda_{W(\o{\F}_p)}\label{dieudonneordfillimit}
	\end{equation}
	for $\star\in \{\et,\mult\}$.
	
		Applying the functor $H^1_{\et}(\cdot)$ to the connected-\'etale sequence of $G_r$
	yields a short exact sequence of $\Z_p[\scrG_{\Q_p}]$-modules
	\begin{equation*}
		\xymatrix{
			0\ar[r] & {H^1_{\et}(G_r^{\et})} \ar[r] & {H^1_{\et}(G_r)} \ar[r] & {H^1_{\et}(G_r^{\mult})}\ar[r] & 0
		}
	\end{equation*}
	which naturally identifies ${H^1_{\et}(G_r^{\star})}$ with the invariants (respectively covariants)
	of $H^1_{\et}(G_r)$ under the inertia subgroup $\I\subseteq \scrG_{\Q_p}$  for $\star=\et$ (respectively 
	$\star=\mult$).
	As $G_r={e^*}'J_r[p^{\infty}]$ by definition, we deduce from this and (\ref{etalecohcrvjac})
	a natural isomorphism of short exact sequences of $\Z_p[\scrG_{\Q_p}]$-modules
	\begin{equation}
	\begin{gathered}
		\xymatrix{
			 0 \ar[r] & {H^1_{\et}(G_r^{\et})} \ar[r]\ar[d]^-{\simeq} & 
			 {H^1_{\et}(G_r)} \ar[r]\ar[d]^-{\simeq} & 
			 {H^1_{\et}(G_r^{\mult})} \ar[r]\ar[d]^-{\simeq} & 0 \\
			 0 \ar[r] & {({e^*}'H^1_{\et,r})^{\I}} \ar[r] &
			 {{e^*}'H^1_{\et,r}} \ar[r] &
			{({e^*}'H^1_{\et,r})_{\I}} \ar[r] & 0 
			}
	\end{gathered}\label{inertialinvariantsseq}
	\end{equation}
	where for notational ease we abbreviate $H^1_{\et,r}:=H^1_{\et}({X_r}_{\Qbar_p},\Z_p)$.
	As the trace maps (\ref{cantraceetale}) are surjective, passing to inverse limits
	on (\ref{inertialinvariantsseq}) yields an isomorphism of short exact sequences
	\begin{equation}
	\begin{gathered}
		\xymatrix{
			0 \ar[r] & {H^1_{\et}(G_{\infty}^{\et})} \ar[r]\ar[d]^-{\simeq} & 
			{H^1_{\et}(G_{\infty})} \ar[r]\ar[d]^-{\simeq} & 
			{H^1_{\et}(G_{\infty}^{\mult})} \ar[r]\ar[d]^-{\simeq} & 0 \\
			0 \ar[r] & {\varprojlim_r ({e^*}'H^1_{\et,r})^{\I}} \ar[r] &
			{\varprojlim_r {e^*}'H^1_{\et,r}} \ar[r] &
			{\varprojlim_r ({e^*}'H^1_{\et,r})_{\I}} \ar[r] & 0
		}
	\end{gathered}\label{limitetaleseq}
	\end{equation}
	Since inverse limits commute with group invariants, the bottom row of (\ref{limitetaleseq})
	is canonically isomorphic to the ordinary filtration of Hida's ${e^*}'H^1_{\et}$,
	and Theorem \ref{FiltrationRecover} follows immediately from (\ref{dieudonneordfillimit}).
	Corollary \ref{MWmainThmCor} is then an easy consequence of Theorem \ref{FiltrationRecover} and 
	\cite[Lemma 3.1.3]{CaisHida1};
	alternately one can prove Corollary \ref{MWmainThmCor} directly from \cite[3.1.2]{CaisHida1},
	using what we have seen above.
\end{proof}

\subsection{Ordinary families of \texorpdfstring{$(\varphi,\Gamma)$}{phi,Gamma}-modules}\label{OrdSigmaSection}

We now study the family of Dieudonn\'e crystals attached to 
the tower of $p$-divisible groups $\{\G_{r}/R_r\}_{r\ge 1}$.
For each pair of positive integers $r\ge s$,
we have a morphism $\rho_{r,s}: \G_{s}\times_{T_{s}} T_r\rightarrow \G_{r}$
in $\pdiv_{R_{r}}^{\Gamma}$;  
applying the contravariant functor $\m_r:\pdiv_{R_r}^{\Gamma}\rightarrow \BT_{\s_r}^{\Gamma}$
studied in \S\ref{PhiGammaCrystals}--\ref{pDivOrdSection} to the map on connected-\'etale sequences induced by $\rho_{r,s}$ 
and using the exactness of $\m_r$ and its compatibility with base change
(Theorem \ref{CaisLauMain}), we obtain
morphisms in  
$\BT_{\s_{r}}^{\Gamma}$
\begin{equation}
\begin{gathered}
	\xymatrix{
		0\ar[r] & {\m_r(\G_r^{\et})} \ar[r]\ar[d]_-{\m_r(\rho_{r,s})} & 
		{\m_r(\G_r)} \ar[r]\ar[d]^-{\m_r(\rho_{r,s})} & {\m_r(\G_r^{\mult})} 
		\ar[r]\ar[d]^-{\m_r(\rho_{r,s})} & 0\\
		0\ar[r ] & {\m_{s}(\G_{s}^{\et})\tens_{\s_{s}} \s_r} \ar[r] & 
		{\m_r(\G_{s})\tens_{\s_{s}} \s_r} \ar[r] & {\m_r(\G_{s}^{\mult})\tens_{\s_{s}} \s_r} \ar[r] & 0
	}\label{BTindLim}
\end{gathered}	
\end{equation}

\begin{definition}\label{DieudonneLimitDef}
	Let $\star=\et$ or $\star=\mult$ and define
	\begin{align}
			\m_{\infty}&:=\varprojlim_r \left(\m_r(\G_r) \tens_{\s_r} \s_{\infty}\right)
			& 
			\m_{\infty}^{\star}&:=\varprojlim_r \left(\m_r(\G_r^{\star}) \tens_{\s_r} \s_{\infty}\right),
			\label{UncompletedDieudonneLimit}
	\end{align}
	with the projective limits taken with respect to the mappings induced by (\ref{BTindLim}).
\end{definition}	

Each of (\ref{UncompletedDieudonneLimit}) is naturally a module over 
the completed group ring $\Lambda_{\s_{\infty}}$
and is equipped with a semilinear action of $\Gamma$ and a $\varphi$-semilinear
Frobenius morphism defined by $F:=\varprojlim (\varphi_{\m_r}\otimes \varphi)$.
Since $\varphi$ is bijective on $\s_{\infty}$, we also have a $\varphi^{-1}$-semilinear
Verscheibung morphism defined as follows.  For notational ease, we provisionally 
set $M_r:=\m_r(\G_r)\otimes_{\s_r} \s_{\infty}$ and we define 
\begin{equation}
	\xymatrix@C=70pt{
		{V_r: M_r} \ar[r]^{\psi_{\m_r}} & 
		{{\varphi}^*M_r }
		\ar[r]^-{\alpha\otimes m\mapsto \varphi^{-1}(\alpha)m} &
		{M_r}
		}\label{Verscheidef}
\end{equation}
with $\psi_{\m_r}$ as above Definition \ref{DualBTDef}.  It is easy to see
that the $V_r$ are compatible with change in $r$, and we
put $V:=\varprojlim V_r$ on $\m_{\infty}$.  We define Verscheibung
morphisms on $\m_{\infty}^{\star}$ for $\star=\et,\mult$
similarly.  Using that the composite of $\psi_{\m_r}$ and $1\otimes\varphi_{\m_r}$
in either order is multiplication by $E_r(u_r) =:\omega$, one checks
\begin{equation*}
	FV = \omega \qquad\text{and}\qquad VF = \varphi^{-1}(\omega).
\end{equation*}
Due to the functoriality of $\m_r$, we moreover have a
$\Lambda_{\s_{\infty}}$-linear action of 
$\H^*$
on each of (\ref{UncompletedDieudonneLimit})
which commutes with $F$, $V$, and $\Gamma$.

\begin{theorem}\label{MainThmCrystal}
	Let $d'$ be the integer specified in Theorem $\ref{MainDieudonne}$.
	Then $\m_{\infty}$ $($respectively $\m_{\infty}^{\star}$ for $\star=\et,\mult$$)$
	is a free $\Lambda_{\s_{\infty}}$-module of rank $2d'$ $($respectively $d'$$)$
	and there is a canonical short exact sequence of $\Lambda_{\s_{\infty}}$-modules
	with linear $\H^*$-action and semi linear actions of $\Gamma$, $F$ and $V$
		\begin{equation}
		\xymatrix{
		0\ar[r] & {\m_{\infty}^{\et}} \ar[r] & {\m_{\infty}} \ar[r] & {\m_{\infty}^{\mult}} \ar[r] & 0
		}.\label{DieudonneLimitFil}
	\end{equation}
	Extension of scalars of $(\ref{DieudonneLimitFil})$
	along the quotient 
	$\Lambda_{\s_{\infty}}\twoheadrightarrow  \s_{\infty}[\Delta/\Delta_r]$
	recovers the exact sequence
	\begin{equation}
		\xymatrix{
			0\ar[r] & {\m_r(\G_r^{\et})\tens_{\s_r} \s_{\infty}} \ar[r] &
			{\m_r(\G_r)\tens_{\s_r} \s_{\infty}} \ar[r] &
			{\m_r(\G_r^{\mult})\tens_{\s_r} \s_{\infty}} \ar[r] & 0
		}.
	\end{equation}
	for each integer $r>0$, compatibly with $\H^*$, $\Gamma$, $F$, and $V$.	
	The Frobenius endomorphism $F$
	commutes with $\H^*$ and $\Gamma$, while $V$ commutes with $\H^*$ and satisfies
	$V\gamma = \varphi^{-1}(\omega/\gamma\omega)\cdot \gamma V$ for all $\gamma\in \Gamma$.
\end{theorem}

\begin{proof}

	Since $\varphi$ is an automorphism
	of $\s_{\infty}$, pullback by $\varphi$ commutes with 
	projective limits of $\s_{\infty}$-modules.
	As the canonical $\s_{\infty}$-linear map $\varphi^*\Lambda_{\s_{\infty}}\rightarrow \Lambda_{\s_{\infty}}$  
	is an isomorphism of rings (even of $\s_{\infty}$-algebras), it therefore suffices
	to prove the assertions of Theorem \ref{MainThmCrystal} after pullback by $\varphi$,
	which will be more convenient due to the relation between
	$\varphi^*\m_r(\G_r)$ and the Dieudonn\'e crystal of $\G_r$.
	
	Pulling back (\ref{BTindLim}) by $\varphi$ gives a commutative diagram
	with exact rows
	\begin{equation}
	\begin{gathered}
		\xymatrix{
			0\ar[r] & {\varphi^*\m_r(\G_r^{\et})} \ar[r]\ar[d] & 
			{\varphi^*\m_r(\G_r)} \ar[r]\ar[d] & {\varphi^*\m_r(\G_r^{\mult})} 
			\ar[r]\ar[d] & 0\\
			0\ar[r] & {\varphi^*\m_{s}(\G_{s}^{\et})\tens_{\s_{s}} \s_r} \ar[r] & 
			{\varphi^*\m_r(\G_{s})\tens_{\s_{s}} \s_r} \ar[r] &
			 {\varphi^*\m_r(\G_{s}^{\mult})\tens_{\s_{s}} \s_r} \ar[r] & 0
		}
	\end{gathered}\label{BTindLimPB}	
	\end{equation}	
	and, as in the proof of Theorem \ref{MainDieudonne},
	we apply the commutative algebra formalism of \cite[\S 3.1]{CaisHida1}.
	In the notation of \cite[Lemma 3.1.2]{CaisHida1}, we take
	$A_r:=\s_r$, $I_r:=(u_r)$, 
	$B=\s_{\infty}$, and $M_r$ each one of the
	terms in the top row of (\ref{BTindLimPB}),
	and we must check that the hypotheses
	 \begin{enumerate}
	\setcounter{equation}{7}
		\renewcommand{\theenumi}{\theequation{\rm\alph{enumi}}}
		{\setlength\itemindent{10pt} 
			\item $\o{M}_r:=M_r/u_r M_r$ is a free $\Z_p[\Delta/\Delta_r]$-module of rank 
			$d'$\label{freehyp2}}
		{\setlength\itemindent{10pt} 
		\item For all $s\le r$ the induced transition maps 
		$\xymatrix@1{
				{\overline{\pr}_{r,s}: \o{M}_r}\ar[r] & 
				{\o{M}_{s}}
				}$\label{surjhyp2}}
		are surjective
	\end{enumerate}
	hold.  
	The isomorphism (\ref{MrToDieudonneMap})
	of Proposition \ref{MrToHodge} 
	ensures, via Theorem \ref{MainDieudonne} (\ref{MainDieudonne1}), that the hypothesis (\ref{freehyp2})
	is satisfied.
	
	Due to the functoriality of (\ref{MrToDieudonneMap}), 
	for any $r\ge s$, the mapping obtained from (\ref{BTindLimPB}) by reducing
	modulo $I_r$ is identified with the mapping on (\ref{DieudonneFiniteExact}) induced (via functoriality
	of $\D(\cdot)$) by $\o{\pr}_{r,s}$.
	As was shown in the proof of Theorem (\ref{MainDieudonne}), these mappings are surjective
	for all $r\ge s$, and we conclude that the hypothesis (\ref{surjhyp2}) holds as well. 
	Moreover,
	the vertical mappings of (\ref{BTindLimPB}) are then surjective by Nakayama's Lemma,
	so as in the proof of Theorem \ref{MainDieudonne} (and keeping in mind
	that pullback by $\varphi$ commutes with projective limits of $\s_{\infty}$-modules), 
	we obtain, by applying $\otimes_{\s_r} \s_{\infty}$ to (\ref{BTindLimPB}), passing
	to projective limits, and pulling back by $(\varphi^{-1})^*$, 
	the short exact sequence (\ref{DieudonneLimitFil}).
	The final assertion is an immediate consequence of the functorial construction of $\varphi_{\m_{r}(\cdot)}$,
	the definition (\ref{Verscheidef}) of $V$, and the intertwining relation (\ref{psiGammarel}).
\end{proof}

\begin{remark}
	In the proof of Theorem \ref{MainThmCrystal}, we could have alternately applied \cite[Lemma 3.1.2]{CaisHida1}
	with $A_r=\s_r$ and $I_r:=(E_r)$, appealing to the identifications (\ref{MrToHodgeMap})
	of Proposition \ref{MrToHodge} and (\ref{CollectedComparisonIsom}) of Proposition \ref{KeyComparison},
	and to Theorem \ref{dRMain} (\cite[Theorem 3.2.3]{CaisHida1}).
\end{remark}

The short exact sequence (\ref{DieudonneLimitFil}) is closely related to its $\Lambda_{\s_{\infty}}$-linear
dual.  In what follows, we write $\s_{\infty}':=\varinjlim_r \Z_p[\mu_N][\![ u_r]\!]$,
taken along the mappings $u_r\mapsto \varphi(u_{r+1})$; it is naturally a $\s_{\infty}$-algebra.

\begin{theorem}\label{CrystalDuality}
		Let $\mu:\Gamma\rightarrow \Lambda_{\s_{\infty}}^{\times}$ be the crossed homomorphism 
		given by $\mu(\gamma):=\frac{u_1}{\gamma u_1}\chi(\gamma) \langle \chi(\gamma)\rangle$.
		There is a canonical $\H^*$ and $\Gal(K_{\infty}'/K_0)$-equivariant isomorphism of 
		exact sequences of $\Lambda_{\s_{\infty}'}$-modules
		\begin{equation}
		\begin{gathered}
			\xymatrix{
			0\ar[r] & {\m_{\infty}^{\et}(\mu \langle a\rangle_N)_{\Lambda_{\s_{\infty}'}}} \ar[r]\ar[d]_-{\simeq} & 
			{\m_{\infty}(\mu \langle a\rangle_N)_{\Lambda_{\s_{\infty}'}}} \ar[r]\ar[d]_-{\simeq} & 
				{\m_{\infty}^{\mult}(\mu \langle a\rangle_N)_{\Lambda_{\s_{\infty}'}}} \ar[r]\ar[d]_-{\simeq} & 0\\
	0\ar[r] & {(\m_{\infty}^{\mult})_{\Lambda_{\s_{\infty}'}}^{\vee}} \ar[r] & 
				{(\m_{\infty})_{\Lambda_{\s_{\infty}'}}^{\vee}} \ar[r] & 
				{(\m_{\infty}^{\et})_{\Lambda_{\s_{\infty}'}}^{\vee}} \ar[r] & 0 
		}
		\end{gathered}
		\label{MinftyDuality}	
		\end{equation}
		that intertwines $F$ 
		$($respectively $V$$)$ on the top row with  
		$V^{\vee}$ $($respectively $F^{\vee}$$)$  on the bottom.
\end{theorem}

\begin{proof}
	We first claim that there is a natural isomorphism of $\s_{\infty}'[\Delta/\Delta_r]$-modules
	\begin{equation}
		\m_r(\G_r)(\mu\langle a\rangle_N)\otimes_{\s_r} \s_{\infty}' \simeq 
		\Hom_{\s_{\infty}'}(\m_r(\G_r)\otimes_{\s_r}\s_{\infty}', \s_{\infty}')
		\label{twistyisom}
	\end{equation}
	that is $\H^*$-equivariant and $\Gal(K_{\infty}'/K_0)$-compatible for the standard action 
	$\gamma\cdot f(m):=\gamma f(\gamma^{-1}m)$ on the right side, and that intertwines
	$F$ and $V$ with $V^{\vee}$ and $F^{\vee}$, respectively.
	Indeed, this follows immediately from the identifications 
	\begin{equation}
			{\m_r(\G_r)(\langle \chi \rangle\langle a\rangle_N)\tens_{\s_r} \s_{\infty}'} \simeq 
			{\m_r(\G_r')\tens_{\s_r} \s_{\infty}'=:\m_r(\G_r^{\vee})\tens_{\s_r}\s_{\infty}'} 
			\simeq {\Dual{\m_r(\G_r)}\tens_{\s_r}{\s_{\infty}'}}
		\label{GrTwist}
	\end{equation}
	and the definition (Definition \ref{DualBTDef}) of duality in $\BT_{\s_r}^{\varphi,\Gamma}$; here, the
	first isomorphism in (\ref{GrTwist}) results from Proposition \ref{GdualTwist}
	and Theorem \ref{CaisLauMain} (\ref{BaseChangeIsom}), while the final 
	identification is due to Theorem \ref{CaisLauMain} (\ref{exequiv}). 
	The identification (\ref{twistyisom}) carries $F$ (respectively $V)$
	on its source to $V^{\vee}$ (respectively $F^{\vee}$) on its target due to the compatibility 
	of the functor $\m_r(\cdot)$ with duality (Theorem \ref{CaisLauMain} (\ref{exequiv})).
		
	From (\ref{twistyisom})
	we obtain a natural $\Gal(K_r'/K_0)$-compatible evaluation pairing of $\s_{\infty}'$-modules
	\begin{equation}
		\xymatrix{
			{\langle\cdot,\cdot\rangle_r: \m_r(\G_r)(\mu\langle a\rangle_N) \tens_{\s_r} \s_{\infty}' 
			\times \m_r(\G_r)\tens_{\s_r} \s_{\infty}'} \ar[r] & {\s_{\infty}'}
			}\label{crystalpairingdefs}
	\end{equation}
	with respect to which the natural action of $\H^*$ is self-adjoint, due to the
	fact that (\ref{GrTwist}) is $\H^*$-equivariant by Proposition \ref{GdualTwist}.
	Due to the compatibility with change in $r$ of the identification (\ref{GrprimeGr}) of Proposition \ref{GdualTwist}
	together with the definitions (\ref{pdivTowers}) of $\pr_{r,s}$ and 
	$\pr_{r,s}'$,
	the identification (\ref{GrTwist}) intertwines the map induced by $\Pic^0(\pr)$ on its source
	with the map induced by ${U_p^*}^{-1}\Alb(\ps)$ on its target. For $r\ge s$, we therefore have
	\begin{equation*}
		\langle \m_r(\rho_{r,s})x , \m_r(\rho_{r,s})y  \rangle_s = 
		\langle x, \m_r({U_p^*}^{s-r}\Pic^0(\pr)^{r-s}\Alb(\ps)^{r-s})y\rangle_r =
		\sum_{\delta\in \Delta_s/\Delta_r} \langle x, \delta^{-1} y \rangle_r, 
	\end{equation*}
	where the final equality follows from (\ref{PicAlbRelation}).
	Thus, the perfect pairings (\ref{crystalpairingdefs}) satisfy the compatibility condition 
	of \cite[Lemma 3.1.4]{CaisHida1} (as in (\ref{pairingchangeinr1}) of the proof 
	of Proposition \ref{DieudonneDuality}) which, together with
	Theorem \ref{MainThmCrystal}, completes the proof.
\end{proof}

The $\Lambda_{\s_{\infty}}$-modules $\m_{\infty}^{\et}$ and $\m_{\infty}^{\mult}$ admit
canonical descents to $\Lambda$: 

\begin{theorem}\label{etmultdescent}
	There are canonical $\H^*$, $\Gamma$, $F$ and $V$-equivariant isomorphisms
	of $\Lambda_{\s_{\infty}}$-modules
	\begin{subequations}
	\begin{equation}
		\m_{\infty}^{\et} \simeq \D_{\infty}^{\et}\tens_{\Lambda} \Lambda_{\s_{\infty}},
	\end{equation}	
	intertwining $F$ and $V$ with 
	$F\otimes \varphi$ and $F^{-1}\otimes \varphi^{-1}(\omega)\cdot \varphi^{-1}$, respectively,
	and $\gamma\in \Gamma$
	with $\gamma\otimes\gamma$, and
	\begin{equation}
		\m_{\infty}^{\mult}\simeq \D_{\infty}^{\mult}\tens_{\Lambda} \Lambda_{\s_{\infty}},	
	\end{equation}
	intertwing $F$ and $V$ with $V^{-1} \otimes \omega \cdot\varphi$
	and $V\otimes\varphi^{-1}$, respectively,
	and $\gamma$ with $\gamma\otimes \chi(\gamma)^{-1} \gamma u_1/u_1$.
	In particular, $F$ $($respectively $V$$)$ 
	acts invertibly on $\m_{\infty}^{\et}$ $($respectively $\m_{\infty}^{\mult}$$)$.
\end{subequations}
\end{theorem}

\begin{proof}
	We twist the identifications (\ref{EtMultSpecialIsoms}) of Proposition 
	\ref{EtaleMultDescription} to obtain natural isomorphisms 
	\begin{equation*}
		\xymatrix@C=40pt{
			{\m_r(\G_r^{\et})} \ar[r]^-{\simeq}_-{F^r \circ (\ref{EtMultSpecialIsoms})} & 
			{\D(\o{\G}_r^{\et})_{\Z_p}\otimes_{\Z_p} \s_r}
		}\qquad\text{and}\qquad
		\xymatrix@C=40pt{
			{\m_r(\G_r^{\mult})} \ar[r]^-{\simeq}_-{V^{-r} \circ (\ref{EtMultSpecialIsoms})} & 
			{\D(\o{\G}_r^{\mult})_{\Z_p}\otimes_{\Z_p} \s_r}
		}
	\end{equation*}
	that are $\H_r^*$-equivariant and, Thanks to \ref{EtMultSpecialIsomsBC}, 
	compatible with change in $r$ using the maps on source and target
	induced by $\pr_{r,s}$.  Passing to inverse limits and appealing again to \cite[Lemma 3.1.2]{CaisHida1}
	and (the proof of) Theorem \ref{MainDieudonne}, we deduce for $\star=\et,\mult$
	natural isomorphisms of $\Lambda_{\s_{\infty}}$-modules
	\begin{equation*}
		\m_{\infty}^{\star} \simeq \varprojlim_r \left( \D(\o{\G}_r^{\star})_{\Z_p}\otimes_{\Z_p} \s_{\infty}\right)
		\simeq \D_{\infty}^{\star}\otimes_{\Lambda} \Lambda_{\s_{\infty}}
	\end{equation*}
	that are $\H^*$-equivariant and satisfy the asserted compatibility with respect to
	Frobenius, Verscheibung, and the action of $\Gamma$ due to 
	Proposition \ref{EtaleMultDescription} and the definitions (\ref{MrEtDef})--(\ref{MrMultDef}).
\end{proof}

We can now prove Theorem \ref{MinftySpecialize}, which asserts that
the slope filtration (\ref{MinftySpecialize}) of $\m_{\infty}$
specializes, on the one hand, to the slope filtration (\ref{DieudonneInfiniteExact})
of $\D_{\infty}$, and on the other hand to the Hodge filtration (\ref{LambdaHodgeFilnomup})
(in the opposite direction!) of ${e^*}'H^1_{\dR}$.  We recall the precise statement:

\begin{theorem}\label{SRecovery}
	Let $\tau:\Lambda_{\s_{\infty}}\twoheadrightarrow \Lambda$ be the $\Lambda$-algebra
	surjection induced by $u_r\mapsto 0$.  
	There is a canonical $\Gamma$ and $\H^*$-equivariant 
	isomorphism of split exact sequences of finite free $\Lambda$-modules
	\begin{equation}
	\begin{gathered}
		\xymatrix{
			0 \ar[r] & {\m_{\infty}^{\et}\tens_{\Lambda_{\s_{\infty}},\tau} \Lambda}\ar[d]_-{\simeq} \ar[r] &
			{\m_{\infty}\tens_{\Lambda_{\s_{\infty}},\tau} \Lambda}\ar[r] \ar[d]_-{\simeq}&
			{\m_{\infty}^{\mult}\tens_{\Lambda_{\s_{\infty}},\tau} \Lambda} \ar[r]\ar[d]_-{\simeq} & 0\\
			0 \ar[r] & {\D_{\infty}^{\et}} \ar[r] & {\D_{\infty}} \ar[r] &
			{\D_{\infty}^{\mult}} \ar[r] & 0
		}
	\end{gathered}\label{OrdFilSpecialize}
	\end{equation}
	which carries $F\otimes 1$ to $F$ and $V\otimes 1$ to $V$.
	
	Let $\theta\circ\varphi:\Lambda_{\s_{\infty}}\rightarrow \Lambda_{R_{\infty}}$
	be the $\Lambda$-algebra surjection induced by $u_r\mapsto (\varepsilon^{(r)})^p-1$.
	There is a canonical $\Gamma$ and $\H^*$-equivariant
	isomorphism of split exact sequences of finite free $\Lambda_{R_{\infty}}$-modules
	\begin{equation}
	\begin{gathered}
		\xymatrix{
				0 \ar[r] & {\m_{\infty}^{\et}\tens_{\Lambda_{\s_{\infty}},\theta\varphi} \Lambda_{R_{\infty}}}
				\ar[d]_-{\simeq} \ar[r] &
			{\m_{\infty}\tens_{\Lambda_{\s_{\infty}},\theta\varphi} \Lambda_{R_{\infty}}}\ar[r] \ar[d]_-{\simeq}&
			{\m_{\infty}^{\mult}\tens_{\Lambda_{\s_{\infty}},\theta\varphi} \Lambda_{R_{\infty}}} \ar[r]\ar[d]_-{\simeq} & 0\\
		0 \ar[r] & {{e^*}'H^1(\O)} \ar[r]_{i} & 
		{{e^*}'H^1_{\dR}} \ar[r]_-{j} & {{e^*}'H^0(\omega)} \ar[r] & 0 
		}
	\end{gathered}
	\end{equation}
	where $i$ and $j$ are the canonical sections 
	given by the splitting in Theorem $\ref{dRtoDieudonne}$.
\end{theorem}

\begin{proof}
	To prove the first assertion, we apply \cite[Lemma 3.1.2]{CaisHida1}
	with $A_r=\s_r,$ $I_r=(u_r)$, $B=\s_{\infty}$, $B'=\Z_p$ (viewed as 
	a $B$-algebra via $\tau$) and $M_r=\m_r^{\star}$ for $\star\in \{\et,\mult,\Null\}$,
	and, as in the proofs of Theorems \ref{MainDieudonne} and \ref{MainThmCrystal}, we must verify the hypotheses
	 \begin{enumerate}
	\setcounter{equation}{15}
		\renewcommand{\theenumi}{\theequation{\rm\alph{enumi}}}
		{\setlength\itemindent{10pt} 
			\item $\o{M}_r:=M_r/u_r M_r$ is a free $\Z_p[\Delta/\Delta_r]$-module of rank 
			$d'$\label{freehyp3}}
		{\setlength\itemindent{10pt} 
		\item For all $s\le r$ the induced transition maps 
		$\xymatrix@1{
				{\overline{\pr}_{r,s}: \o{M}_r}\ar[r] & 
				{\o{M}_{s}}
				}$\label{surjhyp3}}
		are surjective.
	\end{enumerate}
	Thanks to (\ref{MrToDieudonneMap}) in the case $G=\G_r$,
	we have a canonical identification $\o{M}_r:=M_r/I_rM_r  \simeq \D(\o{\G}_r^{\star})_{\Z_p}$
	that is compatible with change in $r$ in the sense that the induced projective
	system $\{\o{M}_r\}_{r}$ is identified with that of Definition \ref{DinftyDef}.
	It follows from this and 
	Theorem \ref{MainDieudonne} (\ref{MainDieudonne1})--(\ref{MainDieudonne2}) that
	the hypotheses (\ref{freehyp3})--(\ref{surjhyp3}) are satisfied,
	and (\ref{OrdFilSpecialize}) is an isomorphism by \cite[Lemma 3.1.3 (5)]{CaisHida1}.

	In exactly the same manner,
	the second assertion follows by appealing to 
	\cite[Lemma 3.1.2]{CaisHida1} with $A_r=\s_r$, $I_r=(E_r)$, $B=\s_{\infty}$, $B'=R_{\infty}$
	(viewed as a $B$-algebra via $\theta\circ\varphi$)
	and $M_r=\m_r^{\star}$, using (\ref{MrToHodgeMap}) and Proposition \ref{KeyComparison}
	together with Theorem \ref{dRMain}
	(see \cite[Theorem 3.2.3]{CaisHida1})
	to verify the requisite hypotheses 
	in this setting.
\end{proof}

\begin{proof}[Proof of Theorem $\ref{RecoverEtale}$ and Corollary $\ref{HidasThm}$]
	Applying Theorem \ref{comparison} to (the connected-\'etale sequence of) $\G_r$
	gives a natural isomorphism of short exact sequences 
	\begin{equation}
	\begin{gathered}
		\xymatrix{
			0 \ar[r] &{\m_r(\G_r^{\et})\tens_{\s_r,\varphi} \a_r } \ar[r]\ar[d]^-{\simeq} & 
			{\m_r(\G_r)\tens_{\s_r,\varphi} \a_r} \ar[r]\ar[d]^-{\simeq} & 
			{\m_r(\G_r^{\mult})\tens_{\s_r,\varphi} \a_r} \ar[r]\ar[d]^-{\simeq} & 0 \\
			0 \ar[r] & {H^1_{\et}(\G_r^{\et})\tens_{\Z_p} \a_r} \ar[r] &
			{H^1_{\et}(\G_r)\tens_{\Z_p} \a_r} \ar[r] & 
			{H^1_{\et}(\G_r^{\mult})\tens_{\Z_p}\a_r}\ar[r] & 0
		}
	\end{gathered}	
		\label{etalecompdiag}
	\end{equation}
	Due to Theorem \ref{MainThmCrystal}, the terms in the top row of \ref{etalecompdiag}
	are free of ranks $d'$, $2d'$, and $d'$ over $\wt{\a}_r[\Delta/\Delta_r]$, respectively,
	so we conclude from \cite[Lemma 3.1.3]{CaisHida1} (using $A=\Z_p[\Delta/\Delta_r]$
	and $B=\a_r[\Delta/\Delta_r]$ in the notation of that result) 
	that $H^1_{\et}(\G_r^{\star})$ is a free $\Z_p[\Delta/\Delta_r]$-module
	of rank $d'$ for $\star=\{\et,\mult\}$ and that $H^1_{\et}(\G_r)$ is free of rank $2d'$
	over $\Z_p[\Delta/\Delta_r]$.  Using the fact that 
	$\Z_p\rightarrow \a_r$ is faithfully flat, it then follows 
	from the surjectivity of the vertical maps in (\ref{BTindLimPB}) 
	(which was noted in the proof of Theorem \ref{MainThmCrystal})
	that the canonical trace mappings $H^1_{\et}(\G_r^{\star})\rightarrow H^1_{\et}(\G_{r'}^{\star})$
	for $\star\in \{\et,\mult,\Null\}$ are surjective for all $r\ge r'$.
	Applying \cite[Lemma 3.1.2]{CaisHida1} with $A_r=\Z_p$, $M_r:=H^1_{\et}(\G_r^{\star})$,
	$I_r=(0)$, $B=\Z_p$ and $B'=\wt{\a}$, we conclude that $H^1_{\et}(\G_{\infty}^{\star})$
	is free of rank $d'$ (respectively $2d'$) over $\Lambda$ for $\star=\et,$ $\mult$ (respectively $\star=\Null$),
	that the specialization mappings 
	\begin{equation*}
	\xymatrix{
		{H^1_{\et}(\G_{\infty}^{\star})\tens_{\Lambda} \Z_p[\Delta/\Delta_r]} \ar[r] & 
		{H^1_{\et}(\G_r^{\star})}
		}
	\end{equation*}
	are isomorphisms, and that the canonical mappings for $\star\in \{\et,\mult,\Null\}$
	\begin{equation}
		\xymatrix{
			{H^1_{\et}(\G_{\infty}^{\star})\tens_{\Lambda} \Lambda_{\wt{\a}}} \ar[r] & 
			{\varprojlim_r \left(H^1_{\et}(\G_r^{\star})\tens_{\Z_p} \wt{\a}\right)}
		}\label{etaleswitcheroo}
	\end{equation}
	are isomorphisms.  Invoking the isomorphism (\ref{limitetaleseq})
	gives Corollary \ref{HidasThm}. By \cite[Lemma 3.1.2]{CaisHida1} with $A_r=\s_r$, $M_r=\m_r(\G_r^{\star})$,
	$I_r=(0)$, $B=\s_{\infty}$ and $B'=\wt{\a}$, we similarly conclude from (the proof of) Theorem 
	\ref{MainThmCrystal} that the canonical mappings for $\star\in \{\et,\mult,\Null\}$
	\begin{equation}
		\xymatrix{
			{\m_{\infty}^{\star}\tens_{\s_{\infty},\varphi} \Lambda_{\wt{\a}}} \ar[r] & 
			{\varprojlim_r \left(\m_r(\G_r^{\star})\tens_{\s_r} \wt{\a}\right)}
		}\label{crystalswitcheroo}
	\end{equation}
	are isomorphisms.
	Applying $\otimes_{\a_r} \wt{\a}$ to the diagram (\ref{etalecompdiag}),
	passing to inverse limits, and using the isomorphisms
	(\ref{etaleswitcheroo}) and (\ref{crystalswitcheroo}) gives (again invoking (\ref{limitetaleseq}))
	the isomorphism (\ref{FinalComparisonIsom}).  
	Using the fact that the inclusion $\Z_p\hookrightarrow \wt{\a}^{\varphi=1}$
	is an equality, the isomorphism (\ref{RecoverEtaleIsom}) follows immediately from 
	(\ref{FinalComparisonIsom}) by taking $F\otimes\varphi$-invariants.	
\end{proof}

Using Theorems \ref{RecoverEtale} and \ref{CrystalDuality} we can give a new proof of
Ohta's duality theorem \cite[Theorem 4.3.1]{OhtaEichler} for the $\Lambda$-adic ordinary 
filtration of ${e^*}'H^1_{\et}$ (see Corollary \ref{OhtaDuality}):

\begin{theorem}\label{OhtaDualityText}
	There is a canonical $\Lambda$-bilinear and perfect duality pairing	
	\begin{equation}
		\langle \cdot,\cdot\rangle_{\Lambda}: {e^*}'H^1_{\et}\times {e^*}'H^1_{\et}\rightarrow \Lambda
		\quad\text{determined by}\quad
		\langle x,y\rangle_{\Lambda} \equiv \sum_{\delta\in \Delta/\Delta_r} 
		(x , w_r {U_p^*}^r\langle\delta^{-1}\rangle^*y)_r \delta \bmod I_r
		\label{EtaleDualityPairing}
	\end{equation}
	with respect to which the action of $\H^*$ is self-adjoint; here,
	$(\cdot,\cdot)_r$ is the usual cup-product pairing on $H^1_{\et,r}$ and 
	$I_r:=\ker(\Lambda\twoheadrightarrow \Z_p[\Delta/\Delta_r])$.  
	Writing $\nu:\scrG_{\Q_p}\rightarrow \H^*$ for the character 
	$\nu:=\chi\langle\chi\rangle \lambda(\langle p\rangle_N)$, the 
	pairing $(\ref{EtaleDualityPairing})$ induces a canonical
	$\scrG_{\Q_p}$ and $\H^*$-equivariant isomorphism of exact sequences 
	\begin{equation*}
		\xymatrix{
			0 \ar[r] & {({e^*}'H^1_{\et})^{\I}(\nu)}
			\ar[d]^-{\simeq} \ar[r] & 
			{{e^*}'H^1_{\et}(\nu)}\ar[d]^-{\simeq} \ar[r] &
			{({e^*}'H^1_{\et})_{\I}(\nu)}
			\ar[d]^-{\simeq}\ar[r] & 0 \\
			0 \ar[r] & {\Hom_{\Lambda}(({e^*}'H^1_{\et})_{\I},\Lambda)} \ar[r] & 
			{\Hom_{\Lambda}({e^*}'H^1_{\et},\Lambda)} \ar[r] &
			{\Hom_{\Lambda}(({e^*}'H^1_{\et})^{\I},\Lambda)}\ar[r] & 0
		}
	\end{equation*}
\end{theorem}

\begin{proof}
	The proof is similar to that of Proposition \ref{DieudonneDuality}, using 
	Corollary \ref{HidasThm} and applying \cite[Lemma 3.1.4]{CaisHida1} ({\em cf.} 
	the proofs of \cite[3.2.4]{CaisHida1} and \cite[Theorem 4.3.1]{OhtaEichler} and of \cite[Proposition 4.4]{SharifiConj}).    
	Alternatively, one can prove Theorem \ref{OhtaDualityText} by appealing to Theorem \ref{CrystalDuality}
	and the isomorphism (\ref{RecoverEtaleIsom}) of Theorem \ref{RecoverEtale}.
\end{proof}

\begin{proof}[Proof of Theorem $\ref{SplittingCriterion}$]
	Suppose first that (\ref{DieudonneLimitFil}) admits a $\Lambda_{\s_{\infty}}$-linear
	splitting $\m_{\infty}^{\mult}\rightarrow \m_{\infty}$
	which is compatible with $F$, $V$, and $\Gamma$.
	Extending scalars along $\Lambda \rightarrow \Lambda_{\wt{\a}}\xrightarrow{\varphi}\Lambda_{\wt{\a}}$ 
	and taking $F\otimes\varphi$-invariants
	yields, by Theorem \ref{RecoverEtale}, a $\Lambda$-linear and $\scrG_{\Q_p}$-equivariant map
	$({e^*}'H^1_{\et})_{\I}\rightarrow {e^*}'H^1_{\et}$
	whose composition with the canonical projection 
	${e^*}'H^1_{\et}\twoheadrightarrow ({e^*}'H^1_{\et})_{\I}$
	is necessarily the identity.
	
	Conversely, suppose that the ordinary filtration of ${e^*}'H^1_{\et}$ is $\Lambda$-linearly
	and $\scrG_{\Q_p}$-equivariantly split.  Applying $\otimes_{\Lambda} \Z_p[\Delta/\Delta_r]$
	to this splitting gives, thanks to Corollary \ref{HidasThm} and the isomorphism
	(\ref{inertialinvariantsseq}), a $\Z_p[\scrG_{\Q_p}]$-linear splitting of
	\begin{equation*}
		\xymatrix{
			0 \ar[r] & {T_pG_r^{\mult}} \ar[r] & {T_pG_r} \ar[r] & {T_pG_r^{\et}}\ar[r] & 0
		}
	\end{equation*}
	which is compatible with change in $r$ by construction.
	By $\Gamma$-descent and Tate's theorem, there is a natural isomorphism
	\begin{equation*}
			{\Hom_{\pdiv_{R_r}^{\Gamma}}(\G_r^{\et},\G_r)}\simeq {\Hom_{\Z_p[\scrG_{\Q_p}]}(T_pG_r^{\et},T_pG_r)}
	\end{equation*}
	and we conclude that the connected-\'etale sequence of $\G_r$ is split (in the category 
	$\pdiv_{R_r}^{\Gamma}$), compatibly with change in $r$.  Due to the functoriality
	of $\m_r(\cdot)$, this in turn implies that
	the top row of (\ref{BTindLim}) is split in $\BT_{\s_r}^{\Gamma}$,
	compatibly with change in $r$, which is easily seen to imply the splitting of (\ref{DieudonneLimitFil}).
\end{proof}

\bibliographystyle{amsalpha_noMR}
\bibliography{mybib}
\end{document}